\documentclass{article}

\usepackage[utf8]{inputenc}
\usepackage{graphicx}
\usepackage{geometry}
\usepackage{microtype}        
\usepackage{mathtools}        

\usepackage[dvipsnames,svgnames,table]{xcolor}

\usepackage[colorlinks=true,linkcolor=blue,citecolor=blue,urlcolor=blue]{hyperref}

\usepackage{cleveref}         

\usepackage{tabularx}
\usepackage{booktabs}         

\usepackage{subcaption}       

\usepackage{algorithm}
\usepackage{algpseudocode}

\usepackage{pgfplots}
\usepgfplotslibrary{groupplots}
\pgfplotsset{compat=1.18}
\usepackage{pgfplotstable}

\usepackage{appendix}
\usepackage{paralist}
\usepackage{cancel}

\usepackage{authblk}

\usepackage{aliascnt}
\usepackage{nlrb-macros}

\crefname{equation}{equation}{equations}
\Crefname{equation}{Equation}{Equations}
\crefname{section}{section}{sections}
\Crefname{section}{Section}{Sections}
\crefname{figure}{figure}{figures}
\Crefname{figure}{Figure}{Figures}
\crefname{subfigure}{panel}{panels}
\Crefname{subfigure}{Panel}{Panels}
\crefname{table}{table}{tables}
\Crefname{table}{Table}{Tables}
\crefname{algorithm}{algorithm}{algorithms}
\Crefname{algorithm}{Algorithm}{Algorithms}
\crefname{lemma}{lemma}{lemmas}
\Crefname{lemma}{Lemma}{Lemmas}
\crefname{theorem}{theorem}{theorems}
\Crefname{theorem}{Theorem}{Theorems}
\crefname{proposition}{proposition}{propositions}
\Crefname{proposition}{Proposition}{Propositions}
\crefname{corollary}{corollary}{corollaries}
\Crefname{corollary}{Corollary}{Corollaries}
\crefname{remark}{remark}{remarks}
\Crefname{remark}{Remark}{Remarks}

\newcommand{\email}[1]{\texttt{#1}}

\newif\ifshowchanges
\showchangesfalse          

\definecolor{RefereeBlue}{HTML}{0072B2}   
\definecolor{CrucialRed} {HTML}{D55E00}   
\definecolor{MinorGreen} {HTML}{009E73}   

\ifshowchanges
  \DeclareRobustCommand{\refmod}[1]{\textcolor{RefereeBlue}{#1}}
  \DeclareRobustCommand{\addmod}[1]{\textcolor{CrucialRed}{#1}}
  \DeclareRobustCommand{\minmod}[1]{\textcolor{MinorGreen}{#1}}
\else
  \DeclareRobustCommand{\refmod}[1]{#1}
  \DeclareRobustCommand{\addmod}[1]{#1}
  \DeclareRobustCommand{\minmod}[1]{#1}
\fi

\newenvironment{refmodblock}{\ifshowchanges\color{RefereeBlue}\fi}{}
\newenvironment{addmodblock}{\ifshowchanges\color{CrucialRed}\fi}{}
\newenvironment{minmodblock}{\ifshowchanges\color{MinorGreen}\fi}{}

\title{ Nonlinear compressive reduced basis approximation : \\ when Taylor meets Kolmogorov}

\definecolor{CP}{RGB}{0,70,160}

\author[1,2]{Joubine Aghili}
\author[1]{Hassan Ballout }
\author[3]{Yvon Maday\footnote{\email{yvon.maday@sorbonne-universite.fr}, corresponding author}}
\author[1]{Christophe Prud'homme}

\affil[1]{IRMA, Université de Strasbourg, CNRS UMR 7501, 7 rue René Descartes, 67084 Strasbourg, France}
\affil[2]{INRIA Nancy-Grand Est, MACARON Project, Strasbourg, France}
\affil[3]{Sorbonne Université, CNRS, Universit\'e Paris Cit\'e, Laboratoire Jacques-Louis Lions (LJLL), F-75005 Paris, France}

\date{2026}

\begin{document}

\maketitle

\begin{abstract}
  This paper investigates a class of model reduction methods for efficiently approximating the solution of parameter-dependent partial differential equations, where the parameter is a 
  vector $\vmu \in \mathbb{R}^p$. In classical settings where the Kolmogorov $N$-width decays sufficiently fast, it is effective to approximate the solution as a sum of $N$ separable terms, each being the product of a parameter-dependent coefficient and a space- (or space-time-) dependent function. This leads to reduced-order models with $N$ degrees of freedom and computational complexity of order ${\mathcal O}(N^3)$.

However, when the Kolmogorov $N$-width decays slowly, $N$ must be large to achieve acceptable accuracy, and such cubic complexity becomes prohibitive. 
The linear complexity measure in terms of Kolmogorov $N$-width has to be replaced by a more involved width: the Gelfand width for instance with its associated notion of sensing number.
Recent nonlinear reduction approaches based on this notion have proposed to decompose the $N$ coordinates into two groups: $n$ free variables and $N-n$ dependent variables, where the latter are expressed as nonlinear functions of the former. Several works have focused on the case where these $N-n$ functions are homogeneous quadratic forms of the $n$ main variables, with optimization strategies for choosing $n$ given a target accuracy.

\addmod{The main contribution of this paper is a rigorous analysis of the local sensing number. We show that, in a sufficiently small neighbourhood of a reference state, the local sensing dimension equals the number $p$ of parameters. This conclusion is local: over wider parameter ranges the sensing dimension may increase because of the global geometry of the embedded manifold. From a practical perspective, however, the local result provides a natural baseline: rather than fixing $n=p$ rigidly, a modest enlargement $n=p+k$, with $k$ small, can improve the robustness and accuracy of the recovery while retaining a latent dimension of the order of the parameter dimension. This should be understood as a practical design principle rather than as a universal upper bound on the sensing dimension. We also show that exact filtering of the tangent space yields a homogeneous quadratic relation between tangent and curvature coefficients, while this homogeneity is not enforced in raw empirical coordinates; restoring the constant and linear blocks that the homogeneous ansatz omits is a minimal relaxation of it, and removes its sensitivity to the choice of reference state. These predictions are illustrated on two test cases of contrasting regimes, a parametrized elliptic problem whose Kolmogorov width decays fast and a transport-dominated one for which it does not. Together they clarify the local justification and the practical limitations of nonlinear reduced models, and motivate either a slightly enlarged latent dimension or more expressive nonlinear mappings over wider parameter ranges.}

This work contributes a theoretical foundation for such nonlinear reduction strategies and highlights the need for further investigations into their reliability, interpretability, and safeguards, particularly in order to apply this type of approach and push back the Kolmogorov Barrier.
\end{abstract}

\tableofcontents

\section{Introduction}

Complexity reduction methods for parameterized systems have received increasing attention in recent years, particularly in the context of approximating solutions to parameterized problems modeled by partial differential equations (PDEs) or data assimilation. These methods often rely, directly or indirectly, on the notion of Kolmogorov $N$-width of some abstract set of solutions, which provides the optimal framework for achieving reduction using linear approximations. Reduced basis methods have sometimes been presented as a way to mitigate, at least for certain classes of problems, Bellman’s “curse of dimensionality.” The underlying argument is that for many parametrized partial differential equations, the set of solutions does not fill the ambient high-dimensional function space, but rather concentrates on an intrinsically low-dimensional manifold \cite{rbpp, benner2015survey, bookRb, hesthaven2016certified, maday2020reduced}. This was demonstrated by Cohen and DeVore \cite{cohen2016kolmogorov} by linking the rapid decrease in the Kolmogorov $N$-width of certain families of solutions to their regularity as a function of the parameters of the problem. Thus, when the $N$-width decays exponentially fast,  reduced basis methods can indeed break the computational burden that would otherwise be associated with Bellman’s curse.

As a result, in many cases, particularly for elliptic or parabolic problems, the associated Kolmogorov $N$-width decays rapidly with $N$, the dimension of the linear space, allowing effective reduction methods to be employed. However, another concept of ``curse'' has quickly emerged, not related to the combinatorial explosion with dimension, but the insufficient linear expressiveness in the face of a manifold of inherently nonlinear solutions. This is the case, e.g.
when problems exhibit convection phenomena or the parameter space becomes large, the decay of the Kolmogorov $N$-width is significantly slowed down, which impedes vanilla linear subspace approximations.
 Consequently, for a desired precision, traditional reduction methods may lose competitiveness compared to classical approximation techniques (such as continuous or discontinuous finite elements, finite volume methods, or spectral methods). This issue is known as the \textit{Kolmogorov barrier} (see, e.g. \cite{ahmed2020breaking, pp8, barnett2022quadratic}).

 The underlying idea behind reduction methods is to observe that classical approximation methods are versatile in the sense that the same method is capable of approximating PDE solutions in a wide range of applications, including fluid dynamics, structural mechanics, and wave propagation
\minmod{which is appealing when one is interested in evaluating one, or a few, solution of the given PDE. Here, on the contrary, we focus on a specific class of problems characterized by a parametrized partial differential equation (PDE),
 expressed in residual form as find $u(\vmu)\in  X$ s.t.
\begin{equation}
   \label{eq:1}
   \mathfrak{R}(u(\vmu) ; \vmu) = 0,
 \end{equation}
where $X$ is a given Hilbert space. Let us then introduce the solution manifold 
$\mathcal{S}:= \{ u(\vmu) \in X \mid \vmu \in \mathcal{P} \} \subset X$, as being the set of all solutions $u(\vmu)$ as the parameter $\vmu= (\mu_1, \mu_2, ..., \mu_p)$ varies over the admissible compact parameter set $\mathcal{P} \subset \mathbb{R}^p$.
In all what follows, we assume that $u$ is a continuous parameter-to-solution map \(  \vmu \mapsto u(\vmu) \), so that $\mathcal{S}$ is compact. the stronger regularity required by some of the results below is stated where it is used. For the sake of  convenience only, we further assume that the elements $u\in \mathcal{S}$ satisfy $\|u\|_X \le 1$. When many solutions in $\mathcal{S}$ are sought, the versatility of
classical approximation methods may hinder their efficiency, and one can alternatively look for a tailored sequence of discrete spaces that aligns with this objective}. Such a method will not perform well for other types of problems, but is optimized for the given parameterized system~\eqref{eq:1}. To achieve this,  when the (linear) Kolmogorov $N$-width is not small, one may adopt a more sophisticated framework for complexity reduction, transitioning from linear to nonlinear approaches.

Before presenting the Nonlinear Compressive Reduced Basis Approximation (NCRBA) method, we will first briefly review the main strategies for overcoming the Kolmogorov barrier in transport problems.
We will not dwell too much on local strategies based on piecewise linear approximations of the solution manifold by introducing adaptability, where we refer e.g. to  \cite {amsallem2012nonlinear, dihlmann2011model, maday2013locally, san2015principal, carlberg2015adaptive,borggaard2016goal,ladeveze2016reduced, badias2017local, peherstorfer2020model, geelen2022localized}, but will focus on global approaches where appropriate intrinsic nonlinear transformations can be used to rectify the manifold of solutions. In the context of transport-dominated problems, where the Kolmogorov $N$-width of the solution manifold decays only as $\mathcal{O}(1/\sqrt{N})$ \refmod{\cite{pinkus1985,ohlberger.rave2016}}, linear model reduction techniques often fail to provide efficient approximations. One of the earliest attempts to overcome this limitation via transformations of the physical domain can be traced back to Ohlberger and Rave's pioneering work on nonlinear approximation spaces \cite{ohlberger2013nonlinear}. Since then, several transformation-based strategies have been proposed to improve the approximability of the solution manifold. Notably, several methods exploit information on characteristic lines to define nonlinear Lagrangian mappings \cite{mojgani2017lagrangian, reiss2018shifted, lu2020lagrangian}. An alternative approach, grounded in the theory of integrable systems, involves constructing a time-dependent test basis via Schrödinger operator Lax pairs \cite{gerbeau2014approximated}.

A broader and increasingly influential class of techniques is based on optimal transport theory to align solution snapshots and reduce variability across the parameter space \cite{iollo2014advection, bernard2018reduced, ehrlacher2020nonlinear, rim2023manifold, iollo2025point}. Closely related are registration-based methods, which reinterpret the transport problem as a geometric alignment task and apply variational approaches to learn optimal deformation maps \cite{taddei2020registration, ferrero2022registration}.

In the field of nonlinear structural dynamics, an innovative reduction strategy was introduced in \cite{jain2017quadratic}. It relies on representing the solution as the sum of two contributions: a first component lying in the tangent subspace spanned by the vibration modes, and a second contribution expressed as a homogeneous quadratic correction involving modal derivatives. This leads to the following parameterization:
\begin{equation}\label{rixen}
\tilde{\mathbf{u}}(\vmu) = \mathbf{u}_{\mathrm{ref}}  + \mathbf{V}\boldsymbol{\alpha}(\vmu) +
\mathbf{W} \boldsymbol{\alpha}(\vmu) \boldsymbol{\alpha}(\vmu)^\top,
\end{equation}
where $\mathbf u_{\mathrm{ref}}$ is a fixed reference state, $\mathbf{V}$ collects the dominant modes, and the symmetric matrix $\mathbf{W}$ encodes the curvature of the manifold through modal derivative information.

More recently, inspired by both a priori knowledge-based constructions and a posteriori data-driven recognition strategies, several works—starting with the pioneering contribution of Lee and Carlberg \cite{lee2020model}—have explored machine-learning-based nonlinear model reduction. These approaches exploit deep architectures such as convolutional autoencoders to learn encoders, decoders, or both, in order to parameterize low-dimensional manifolds. Several alternative methodologies have emerged to address this nonlinear complexity in reduced basis methods based on different aspects of advanced machine learning techniques. 
For instance, Franco et al. \cite{pp5} and Vitullo et al. \cite{pp3} explored neural network-based strategies for problems with microstructural complexity, demonstrating their potential for accurate and efficient model reduction. 
Other notable contributions include the operator inference framework by Kramer et al. \cite{pp7}, which directly learns nonlinear reduced models from data. 
Despite these advancements, existing methods often require a large amount of data to effectively train the neural network and provide a reliable and robust approximation of the solution. Furthermore, current approaches frequently require learning all the modes, which increases computational demands, expands training set sizes, and may introduce potential instability.

Remarkably, as highlighted in \cite{fresca2021comprehensive}, some of these methods do not even require solving the underlying parametric equation in the online stage. \refmod{On the one hand, although a priori error analyses have been developed for some of such   approaches \cite{brivio2024error}, they do not generally provide readily computable error bounds for a given trained model. On the other hand, a posteriori certification is, in principle, possible through residual-based error estimators. However, evaluating these estimators may entail computational costs that depend on the high-fidelity discretization, unless suitable offline-online decompositions or hyper-reduction techniques are available, potentially compromising the computational efficiency of fully data-driven approaches. In parallel, manifold-based approaches similar in spirit to~\eqref{rixen} have emerged, but with $\mathbf{V}$ and $\mathbf{W}$ now obtained through learning instead of analytical derivation \cite{barnett2022quadratic, geelen2023operator}. Further refinements and performance enhancements have been proposed, for instance in \cite{greedy_quad}. In this respect, these methods offer an attractive alternative: by exploiting the structure of the solution manifold, they may reduce the amount of training data required while naturally  providing a framework for the efficient evaluation of residual-based a posteriori error estimators.
}

In this context, the NCRBA method, inspired by the notions of \textit{Gelfand width} (with its associated notion of sensing number)  and \textit{nonlinear Kolmogorov $N$-width}, see \cite{DeVore1998}, has been introduced in Barnett et al. \cite{barnett2023}, Cohen et al. \cite{pp10}. This approach builds on the observation that, when a solution is expressed in a linear reduced basis (such as one arising from a Kolmogorov-type approach, as previously described, but leading to an approximation of size $N$, which may be too large to be practical), actually only the first few $n$ coefficients in this basis represent genuine degrees of freedom. The remaining $N-n$ coefficients, though necessary for the desired approximation accuracy, are not independent variables but can be deduced from the first $n$ coefficients via a functional relationship. The primary challenge thus lies in establishing a methodology that enables the reconstruction of the $N-n$ higher-order coefficients from the $n$ fundamental coefficients.

Note that, as a general rule, the effectiveness of nonlinear approximation methods relies on the accurate estimation of tangent spaces to the underlying solution manifold $\cS$. 
This geometric perspective has deep roots in manifold learning theory, where tangent space estimation has been extensively studied. 
The work of Tyagi et al. \cite{tyagi2013tangent} established theoretical foundations for tangent space estimation on smooth Riemannian manifolds using local principal component analysis (PCA), providing convergence guarantees under specific sampling conditions. 
These results were subsequently strengthened by Aamari and Levrard \cite{aamari2019nonasymptotic}, who derived optimal non-asymptotic convergence for $\mathcal{C}^k$ manifolds. 
Recent advances have incorporated curvature-aware methods, such as the CA-PCA approach proposed by Gilbert and O'Neill \cite{gilbert2023ca}, which adapt to local manifold geometry for improved estimation accuracy. 
These geometric insights directly inform the selection and construction of reduced basis components in our approach, where the coefficients associated with the first $n$ modes correspond to coordinates in the estimated tangent space of the solution manifold.

This paper presents several key contributions that advance both the theoretical understanding and practical implementation of the NCRBA method, which continues the study of the method and extends the understanding of its theoretical formalism of Ballout et al. \cite{ballout2024nonlinear, ballout2025combined}. In particular, we answer the question of choosing the number $n$ of reduced basis modes based on the number of parameters $p$ on which the problem depends. Specifically:
\begin{itemize}

\item[a.] We establish a theoretical link between the NCRBA method and the classical Taylor series expansion to rigorously justify the selection of the first reduced basis components as genuine degrees of freedom. While existing work, such as that by Lee and Carlberg \cite{lee2020model}, effectively applies neural networks for nonlinear approximation, our method improves robustness by anchoring the choice of reduced basis components in the well-established Taylor series theory. This connection \minmod{provides a rigorous local justification for the selection of the leading reduced coordinates and motivates richer recovery mappings over wider parameter ranges.}

\item[b.] Our analysis reveals that, at least locally, the number of effective modes (expressed in terms of sensing number) — or degrees of freedom — coincides with the dimension of the parameter space. 
This provides a stronger theoretical foundation than previously employed heuristic methods. In our context, our results on the sensing number complement those established by Franco et al. \cite{pp5}
related to the nonlinear Kolmogorov width.

\item[c.] We demonstrate that polynomial approximations, although effective locally, have limited validity in broader parameter spaces. This observation highlights a key limitation of polynomial-based reduction approaches and underscores the need for adaptable, localized methods, such as the NCRBA.

\item[d.] \minmod{We show that the homogeneity of the quadratic ansatz used in the literature is justified only under exact tangent filtering, which a data-driven basis does not achieve. Restoring the constant and linear blocks is a minimal relaxation of it, and whenever such methods are employed there is no reason to discard these affine terms.}

\item[e.] By establishing a parallel between the NCRBA method and the concept of sensing number in nonlinear Kolmogorov theory, we strengthen the theoretical foundation of our approach and enhance its predictive capability in scenarios where traditional Kolmogorov methods are limited. 

\end{itemize}

Recall that we are considering the case where the Kolmogorov $N$-width decays slowly, implying that the optimal linear subspace minimizing the worst-case projection error is of large dimension, and, in practice, generally not accessible. The goal of the preliminary stage of any reduced basis technology is thus to construct a surrogate reduced space that approximates this optimal subspace as closely as possible, while carefully avoiding unnecessary increases in dimension beyond what is required to meet the target accuracy.

Several strategies have been proposed for this purpose, often as a decisive offline step that governs the efficiency of subsequent reduced-order modeling: for a fixed reduced dimension, the resulting accuracy can differ by orders of magnitude depending on the method used to build the reduced basis. \minmod{Two of them are classical. The Singular Value Decomposition (SVD), often referred to as the Proper Orthogonal Decomposition (POD), compresses a set of precomputed snapshots and is optimal in the mean over that set, but requires a high-fidelity solve at every parameter sampled. The greedy algorithm instead selects parameters one at a time, driven by an a posteriori error estimator, so that a large training set can be explored at the cost of one solve per selected parameter, the selection targeting the worst case rather than the mean. We propose to combine them. The \emph{Greedy-Sampled SVD} performs the greedy selection to decide at which parameters the solves are carried out, and compresses the resulting snapshots by SVD. It costs no more solves than an SVD over the same number of snapshots, and inherits the worst-case selection of the former together with the mean-error optimality of the latter.}

The subsequent sections of this paper are organized as follows. In \Cref{sec:linear-rb}, we recall the notion of Kolmogorov $N$-width and the basis constructions just described. In \Cref{sec:nonlinear-approximation}, we introduce notions of \emph{nonlinear widths} by means of encoders and decoders, among them the \emph{sensing numbers} on which the NCRBA method rests; we then fix the general form of nonlinear decoder shared by all the methods considered here, and present in that form both the NCRBA and the quadratic manifold methods of the literature. The principal contribution of the paper is developed in \Cref{sec:taylor}, where we treat the connection between the space spanned by the SVD modes of the solution manifold $\cS$ with the first monomials of the Taylor expansion around a given element, and deduce from it the local sensing dimension of $\cS$. The significance of the quadratic terms is underscored in
\Cref{sec:quadratic}, where the quadratic methods are confronted with that local analysis and an augmented quadratic ansatz is proposed. Useful results needed in the previous sections are recalled in \Cref{appendix:principal_angles} and \Cref{appendix:svd_modes_convergence}
. Finally, numerical validations are given within \Cref{sec:num}, on two test cases of contrasting regimes --- a parametrized elliptic problem whose Kolmogorov width decays fast \refmod{and a transport-dominated one for which it does not} --- more precisely:
\begin{inparaenum}[(i)]
  \item a comparison between the Greedy-Sampled SVD and the other basis construction methods,
  \item illustrations of the proximity of the space spanned by the SVD modes to those spanned
        by the Taylor modes, and of the corresponding behaviour of the reduced coefficients,
  \item the recovery of the high reduced coefficients from the low ones, and the dimension at
        which that recovery becomes possible,
  \item a comparison between the homogeneous and the augmented quadratic ansatz.
\end{inparaenum}

\section{Linear approximation and reduced basis construction}
\label{sec:linear-rb}
In this section we recall the classical framework of linear approximation, the notion of Kolmogorov $N$-width, and the construction of reduced spaces in practice. Throughout, $X$ denotes the ``natural'' Hilbert space associated with the underlying PDE.

Linear approximation methods aim at approximating the solution manifold $\mathcal{S}$ by an appropriate $N$-dimensional linear subspace $\mathcal{V}_N \subset X$, commonly referred to as the reduced space, \minmod{together with a fixed \emph{reference state} $u_{\mathrm{ref}} \in X$ about which the approximation is expanded: a solution $u$ is approximated by an element of $u_{\mathrm{ref}} + \mathcal{V}_N$. Equivalently, $\mathcal{V}_N$ approximates the centred manifold
\(
  \mathcal{S}_{\mathrm{ref}} := \{\,v=  u - u_{\mathrm{ref}} \;:\; u \in \mathcal{S} \,\},
\)
the natural presentation corresponding to $u_{\mathrm{ref}} = 0$, for which $\mathcal{S}_{\mathrm{ref}} = \mathcal{S}$.}

The performance of such an approximation is quantified by the Kolmogorov $N$-width \minmod{of a compact set $\mathcal{K} \subset X$},
\begin{equation}
\label{eq:kolmogorov_width}
    d_N(\mathcal{K}) := \inf_{\substack{\mathcal{V} \subset X \\ \dim(\mathcal{V}) \le N}}
    \sup_{u \in \mathcal{K}} \inf_{w \in \mathcal{V}} \| u - w \|_X ,
\end{equation}
\minmod{the quantity of interest here being $d_N(\mathcal{S}_{\mathrm{ref}})$,} which characterizes both the optimal approximation error attainable by any linear reduced space of dimension $N$ and the corresponding (generally non-unique) optimal subspaces. In the Hilbertian setting considered here, the best approximation of any $u \in \mathcal{K}$ in a given subspace $\mathcal{V}$ is its orthogonal projection $P_{\mathcal{V}}(u)$.
Determining a subspace achieving the infimum in~\eqref{eq:kolmogorov_width} is, however, an open and computationally intractable problem in general. We therefore focus on two well-established constructions, the \emph{SVD} and the \emph{greedy} algorithms, the former being the one on which the analysis of the following sections rests. Finally, we propose an enhanced strategy, the \emph{Greedy-Sampled SVD} (GSS), which combines the advantages of both approaches and aims to provide reduced spaces that are closer to the optimal ones\minmod{; its final step being an SVD, it carries the structure analysed in \Cref{sec:taylor}, which a basis spanned by selected solutions does not}.

We assume from now on that the problem has been discretized by a high-fidelity method of
sufficient accuracy, and identify the exact solution $u(\boldsymbol{\mu})$ of~\cref{eq:1} with its high-fidelity approximation $u_{\mathcal{N}}(\boldsymbol{\mu})$, represented by its vector of degrees of freedom $\mathbf{u}_{\mathcal{N}}(\boldsymbol{\mu}) \in X \cong \mathbb{R}^{\mathcal{N}}$. 
\begin{minmodblock}
The discrete space $X$ carries the inner product
\begin{equation}
\label{eq:X_inner}
  (\mathbf{x},\mathbf{y})_X = \mathbf{x}^\top \mathbf{X} \mathbf{y}, \qquad
  \mathbf{X} \in \mathbb{R}^{\mathcal{N}\times \mathcal{N}}
  \ \text{symmetric positive definite},
\end{equation}
$\mathbf{X}$ being the discrete counterpart of the inner product of the continuous space, and
$\|\cdot\|_X$ the induced norm.

In this discrete setting, the reduced space is given by a basis matrix $\mathbf V\in\mathbb R^{\mathcal N\times N}$, $N\ll\mathcal N$, orthonormal for $(\cdot,\cdot)_X$, that is $\mathbf V^{\top}\mathbf{X}\mathbf V=\mathbf I$, and the reference state by its coordinate vector $\mathbf u_{\mathrm{ref}}\in\mathbb R^{\mathcal N}$. Writing $\mathbf v:=\mathbf u-\mathbf u_{\mathrm{ref}}$ for a state $\mathbf u\in\mathbb R^{\mathcal N}$ centred at that reference,
\begin{equation}
  \label{eq:linearapprox}
  \widetilde{\mathbf{u}} = \mathbf u_{\mathrm{ref}} + \mathbf{V}\boldsymbol{\alpha},
  \qquad
  \boldsymbol{\alpha}=\mathbf V^{\top}\mathbf{X}\,\mathbf v
  \in\mathbb R^{N} ,
\end{equation}
the columns of $\mathbf V$ spanning $\mathcal V_N$ and $\boldsymbol\alpha$ collecting the reduced (generalized) coordinates. We write $\Pi_{\mathbf V}:=\mathbf V\mathbf V^{\top}\mathbf{X}$ for the associated $X$-orthogonal projection --- the matrix form of $P_{\mathcal V_N}$ above, with $\mathcal V_N=\mathrm{Range}(\mathbf V)$ --- so that
$\widetilde{\mathbf u}=\mathbf u_{\mathrm{ref}}+\Pi_{\mathbf V}\mathbf v$.

Three choices of $\mathbf u_{\mathrm{ref}}$ occur in this work. The first two are the classical conventions: $\mathbf u_{\mathrm{ref}}=\boldsymbol 0$, used in reduced basis methods for parametric PDEs, where the reduced space is built to approximate the whole solution manifold rather than deviations from its mean; and $\mathbf u_{\mathrm{ref}}=\bar{\mathbf u}$, the empirical mean of the sampled solutions (defined in \Cref{sec:svd}), customary in turbulence modelling, fluid dynamics and signal processing, where the modes then represent fluctuations about the mean field. The third, $\mathbf u_{\mathrm{ref}}=\mathbf u^{*}\in\mathcal S$, takes as reference a solution of the problem itself; it is introduced in \Cref{sec:taylor} and used throughout the analysis that follows. The choice is a modelling decision rather than part of the definition of any method, and its consequences are examined in \Cref{sec:problem_formulation}.
\end{minmodblock}

\subsection{Singular value decomposition}
\label{sec:svd}
The SVD reduced space is built from a finite \emph{training set} $\mathcal{P}_M=\{\boldsymbol\mu^1,\dots,\boldsymbol\mu^M\}\subset\mathcal{P}$, sufficiently large and suitably chosen to represent the parameter domain, on which the high-fidelity problem is solved. \minmod{The corresponding elements of $\mathcal S_{\mathrm{ref}}$ are collected in the \emph{snapshot matrix}}
\[
\mathbf{S} = \left[ \mathbf{s}_{1}, \ldots,\mathbf{s}_{M} \right]
\in \mathbb{R}^{\mathcal{N} \times M},
\qquad
\mathbf{s}_{i} :=  \minmod{\mathbf v(\boldsymbol\mu^{i})}.
\]
In general $\mathcal{N} \ge M$, although this is not required. \minmod{The mean reference state introduced above is
$\bar{\mathbf u}=\tfrac1M\sum_{i=1}^{M}\mathbf u_{\mathcal N}(\boldsymbol\mu^{i})$.}

The SVD is a Euclidean notion, whereas $X$ carries the inner product~\eqref{eq:X_inner}. \minmod{We therefore pick $\mathbf B$ with $\mathbf X=\mathbf B^{\top}\mathbf B$, a Cholesky
factor for instance, and decompose the transformed snapshot matrix
$\widetilde{\mathbf S}=\mathbf B\mathbf S$,}
\begin{equation}
  \label{eq:svd}
  \widetilde{\mathbf{S}} = \widetilde{\mathbf{U}}\,\boldsymbol\Sigma\,\mathbf{Z}^{\top},
\end{equation}
where
\begin{itemize}
    \item $\widetilde{\mathbf{U}} = [\widetilde{\boldsymbol{\varphi}}_1, \dots,
          \widetilde{\boldsymbol{\varphi}}_{\mathcal{N}}] \in
          \mathbb{R}^{\mathcal{N} \times \mathcal{N}}$ is an orthogonal matrix whose columns
          are the left singular vectors of $\widetilde{\mathbf{S}}$;
    \item $\boldsymbol\Sigma \in \mathbb{R}^{\mathcal{N} \times M}$ is rectangular diagonal
          and contains the singular values, classically ranked in decreasing order,
          $\sigma_k \ge \sigma_{k+1}$ for all $k = 1, \dots, M-1$;
    \item $\mathbf{Z} \in \mathbb{R}^{M \times M}$ is an orthogonal matrix whose columns are
          the right singular vectors.
\end{itemize}
\minmod{The modes of the original variables are recovered as
$\mathbf U=\mathbf B^{-1}\widetilde{\mathbf U}
=[\boldsymbol\varphi_1,\dots,\boldsymbol\varphi_{\mathcal N}]$, which satisfies
$\mathbf U^{\top}\mathbf{X}\mathbf U=\mathbf I$: the $\boldsymbol\varphi_k$ are orthonormal for
$(\cdot,\cdot)_X$. Two choices of $\mathbf B$ differ by an orthogonal matrix, so that $\boldsymbol\Sigma$, $\mathbf Z$ and
$\mathbf U$ are independent of it. The Euclidean case is $\mathbf{X}=\mathbf B=\mathbf I$. For $1\le a\le b$ we write $\mathbf U_{(:,a:b)}$ for the submatrix of $\mathbf U$ formed by its columns $a$ to $b$; blocks of modes selected in this way are the central objects of \Cref{sec:taylor}.}

\begin{minmodblock}
In practice, one can alternatively consider the $M\times M$
eigenvalue problem for the correlation matrix
\[
\mathbf{G} = \widetilde{\mathbf{S}}^{\top}\widetilde{\mathbf{S}}
= \mathbf{S}^\top \mathbf{X} \mathbf{S} \in \mathbb{R}^{M\times M},
\qquad
\mathbf{G}\,\mathbf{z}_k = \lambda_k \mathbf{z}_k,
\qquad \lambda_1 \geq \lambda_2 \geq \cdots \geq 0 ,
\]
which is cheaper than a decomposition of the full snapshot matrix, together with
\[
\boldsymbol{\varphi}_k = \frac{1}{\sqrt{\lambda_k}} \, \mathbf{S}\,\mathbf{z}_k,
\qquad (\boldsymbol{\varphi}_k,\boldsymbol{\varphi}_\ell)_X = \delta_{k\ell},
\qquad \sigma_k=\sqrt{\lambda_k},
\]
for every positive eigenvalue $\lambda_k$. These are the SVD modes, and neither expression
requires $\mathbf B$ to be formed.
\end{minmodblock}

For $N\le M$, the construction~\eqref{eq:linearapprox} is used with
$\mathbf V=[\boldsymbol{\varphi}_1,\dots,\boldsymbol{\varphi}_N]$, so that
$\mathcal{V}_N = \mathrm{span}\{\boldsymbol{\varphi}_1,\dots,\boldsymbol{\varphi}_N\}$.
\begin{minmodblock}
The corresponding coordinates,
\begin{equation}
\label{eq:rb-coeffs}
\alpha_k(\boldsymbol{\mu}) := \big(\mathbf{v}(\boldsymbol{\mu}),\,\boldsymbol{\varphi}_k\big)_X ,
\qquad k=1,\dots,N,
\end{equation}
are the \emph{reduced coefficients} of a solution. They are the central objects of this work:
the question addressed in the following sections is how many of them are genuinely
independent, i.e. true degrees of freedom.
\end{minmodblock}
The space $\mathcal{V}_N$ satisfies
\begin{equation}
\label{eq:svd_opt}
\sum_{i=1}^{M}\left\| \mathbf{s}_{i} - P_{\mathcal{V}_N}(\mathbf{s}_{i})\right\|_{X}^2
= \inf_{\substack{\mathcal{V} \subset X \\ \dim(\mathcal{V}) \le N}}
\sum_{i=1}^{M}
\left\| \mathbf{s}_{i} - P_{\mathcal{V}}(\mathbf{s}_{i})\right\|_{X}^2
= \sum_{k=N+1}^{M} \lambda_k ,
\end{equation}
so that $\mathcal{V}_N$ is the best $N$-dimensional approximation space for the snapshots in
the $X$-norm. This is an optimality property in average over the training set, \ie in
$\ell^2(\mathcal{P}_M)$, rather than in the $\ell^\infty(\mathcal{P}_M)$ sense involved in the
definition~\eqref{eq:kolmogorov_width}.

We end this subsection by noting that there is a large literature on various ways of performing that sampling and the calculation of the SVD modes, including the method of snapshots, adaptive sampling strategies, and more recently the smart use of randomized and probabilistic approaches for large-scale reduced-order modeling, see, e.g. \cite{sirovich1987turbulence, bui2008model, amsallem2011online, halko2011finding, holmes2012turbulence, bookRb, balabanov2019randomized, cohen2020reduced, chellappa2021training, carere2021weighted, balabanov2022block, billaud2024probabilistic}.

\subsection{Greedy Algorithms}
\label{sec:greedy}
Several other widely adopted approaches to construct a reduced space fall within the family of \emph{greedy algorithms}. Unlike the SVD method, these algorithms build the reduced basis iteratively. At each iteration $n$, a new basis vector is added by selecting the parameter value that maximizes a suitable representation of the $X$-projection error. This strategy ensures that the reduced space is enriched in the directions where the current reduced model performs worst. 
Concretely, the greedy procedure requires computing one high-fidelity solution per iteration. Thus, to build a reduced basis of dimension \( N \), only \( N \) high-fidelity solves are needed. This represents a significant computational advantage compared to the SVD approach, which typically requires computing \( M \gg N \) snapshot solutions in advance.

The various greedy approaches differ in the way this representation of the projection error is performed. Typically, it is based on an \emph{a posteriori error estimator} that provides a surrogate to a projection error on the current reduced space. This evaluation through such a posteriori error estimator leads to the notion of \emph{weak greedy algorithm}.
An essential ingredient of the weak greedy algorithm is thus the availability of an error estimator \( \Delta_n(\boldsymbol{\mu}) \), which estimates the projection error on the current \( n \)-dimensional reduced space \( \mathcal{V}_n \) (built in previous iterations), with a cost that is a function of $n$ (say $\mathcal O(n^3)$). Optimally, it should provide a bound for the error induced by projecting the centred solution \( \mathbf{v}(\boldsymbol{\mu}) \) onto the current  reduced space \( \mathcal{V}_n \), \ie, there exists $\gamma\in \refmod{(0,1]}$
\begin{equation}
\label{eq:greedy_error_estimator}
0 < \gamma \, \Delta_n(\boldsymbol{\mu})\le \left\| \mathbf{v}(\boldsymbol{\mu}) - P_{\mathcal{V}_n}\left( \mathbf{v}(\boldsymbol{\mu}) \right) \right\|_X \le \Delta_n(\boldsymbol{\mu}), \quad \forall \boldsymbol{\mu} \in \mathcal{P}.
\end{equation}

Another essential ingredient to ensure the computational efficiency of the greedy algorithm is, as in \Cref{sec:svd}, to replace the continuous parameter domain \( \mathcal{P} \) by a finite \emph{training set} \minmod{\( \mathcal{P}_g \subset \mathcal{P} \), of cardinality \( M_g \),} which can be chosen very large thanks to the marginal cost of the prediction and error estimator, unlike \( \mathcal{P}_M \), each of whose parameters costs a high-fidelity solve. This avoids solving a global optimization problem over \( \mathcal{P} \) at each iteration and instead reduces the selection of the next parameter to a discrete search over \( \mathcal{P}_g \), leading to the following enumeration problem:
\begin{equation}
\label{eq:greedy_selection}
    \boldsymbol{\mu}^{(n+1)} = \arg\max_{\boldsymbol{\mu} \in \mathcal{P}_g} \Delta_n(\boldsymbol{\mu}).
\end{equation}
Then, $\mathbf{v}(\boldsymbol{\mu}^{(n+1)})$ is computed and used to enrich $\mathcal{V}_n$.

In Algorithm~\ref{alg:weak_greedy}, we present the \emph{weak greedy algorithm} that contains some orthogonalization procedure (see Algorithm~\ref{alg:ortho}). For the efficient construction of the error estimator \( \Delta_n(\boldsymbol{\mu}) \), we refer the reader to~\cite{rbpp}.

\begin{algorithm}[h]
\caption{Weak Greedy Algorithm}
\label{alg:weak_greedy}
\begin{algorithmic}[1]
\State \textbf{Input:} Training set \( \mathcal{P}_g \subset \mathcal{P} \), tolerance \( \varepsilon > 0 \), maximum dimension \( N_{\max} \), initial parameter \( \boldsymbol{\mu}^{(1)} \)
\State \textbf{Output:} Basis \(  \mathbf{V} \in \mathbb{R}^{\mathcal{N}\times N} \)
\State  \(\mathbf{V}= []\), \( n = 0 \), \( \Delta = \varepsilon +1 \)
\While{ \( n < N_{\max} \) \textbf{and} \( \Delta > \varepsilon \) }
    \State \( n \gets n + 1 \)
    \State Compute \( \mathbf{v}(\boldsymbol{\mu}^{(n)})= \mathbf{u}_\mathcal{N}(\boldsymbol{\mu}^{(n)}) - \mathbf{u}_{\mathrm{ref}} \)
    \State \(\boldsymbol{\xi}_n = \text{GramSchmidt}\left(\mathbf{V}, \mathbf{v}(\boldsymbol{\mu}^{(n)})\right)\)
    \State \( \mathbf{V} \gets [\mathbf{V},\,\boldsymbol{\xi}_n]\)
    \State Select \( \boldsymbol{\mu}^{(n+1)} \in \mathcal{P}_g \) such that:
    \[
    \boldsymbol{\mu}^{(n+1)} = \arg\max_{\boldsymbol{\mu} \in \mathcal{P}_g} \Delta_n(\boldsymbol{\mu})
    \]
    \State Set \( \Delta \gets \Delta_n(\boldsymbol{\mu}^{(n+1)}) \)

\EndWhile
\State $N\gets n$ 
\end{algorithmic}
\end{algorithm}

\begin{algorithm}[h]
\caption{Gram--Schmidt orthonormalization}
\label{alg:ortho}
\begin{algorithmic}[1]
\Function{GramSchmidt}{$\mathbf{Q}, \mathbf{x}$}
    \Comment{\( \mathbf{Q} \) \( X \)-orthonormal, possibly empty}
    \If{$\mathbf{Q} = []$}
        \State $\mathbf{z} = \mathbf{x}$
    \Else
        \State $\mathbf{z} = (\mathbf{I}-\Pi_{\mathbf{Q}})\,\mathbf{x}$
    \EndIf
    \State $\mathbf{z} \gets \mathbf{z}/\left\| \mathbf{z}\right\|_X$
    \State \Return \(\mathbf{z}\)
\EndFunction
\end{algorithmic}
\end{algorithm}

From~\cref{eq:greedy_selection}, one can clearly observe that the greedy method operates in a \emph{worst-case} sense—namely in the \( \ell^\infty(\mathcal{P}_g) \) norm—similar to the Kolmogorov \( N \)-width (in $L^\infty(\mathcal{P})$). This contrasts with the SVD-based approach (see \cref{sec:svd}), which provides an \emph{optimal basis in the mean-square} (\ie, \( \ell^2(\mathcal{P}_M) \)) sense.

In this direction, the numerical analysis of these algorithms (see~\cite{binev2011convergence, deVore2013})
proves that, algebraic (resp.\ exponential) decay to zero of the 
Kolmogorov $N$-width is asymptotically satisfied by the sequence of spaces \( \{ \mathcal{V}_n \} _n\) obtained by the weak-greedy algorithm, provided that the training sets $\mathcal{P}_g$ are suitably chosen. 
More precisely, let us introduce a comparable quantity as $d_{n}(\mathcal S_{\mathrm{ref}})$ associated to the series of spaces \( \{ \mathcal{V}_n \} _n\) defined in the greedy algorithm:
\begin{equation}
\Gamma_n( \mathcal S_{\mathrm{ref}}) = \sup_{\boldsymbol{\mu}\in \mathcal{P}_g}\inf_{\mathbf{w}_{n}\in \mathcal V_{n}} \norm{\mathbf{v}(\boldsymbol{\mu}) - \mathbf{w}_{n}}_X.
\end{equation}
In particular, it holds that, for any $n\ge 1$, (at least in the limit $M_g\rightarrow \infty$)
\begin{itemize}
\item  If, for some \minmod{$s >0$}, $d_n(\mathcal{S}_{\mathrm{ref}})\le C_0 n^{-s}$, then $\Gamma_{n}(\mathcal{S}_{\mathrm{ref}}) \le C_1 n^{-s}$ with $C_1= 2^{5s+1}\gamma^{-2} C_0$ ,
\item if, for some \minmod{$s >0$}, $d_n(\mathcal{S}_{\mathrm{ref}})\le C_0 e^{-c_0 n^s}$, then $\Gamma_{n}(\mathcal{S}_{\mathrm{ref}}) \le\sqrt{2C_0} \gamma^{-1} e^{-c_1n^s}$ where $c_1=2^{-1-2s} c_0$,
\end{itemize}
\refmod{where $\gamma$ is the error bound efficiency constant that appears in \eqref{eq:greedy_error_estimator}} which is a very powerful result, at least whenever the decay of the Kolmogorov $N$-width is fast enough. 

In this paper, we are rather interested in cases where the Kolmogorov barrier is present and thus we do not have this exponential or geometric rate for  $d_n(\mathcal{S}_{\mathrm{ref}})$, we have to be more precise in this comparison, and use  the sharper statement also present in the above references (see also \cite{maday2020reduced})
: for any $n\ge 1$, we have
$$ \Gamma_n(\mathcal{S}_{\mathrm{ref}})\le c \gamma^{-1} \min_{1\le m< n}   d_m^{\frac{n-m}{n}}(\mathcal{S}_{\mathrm{ref}}).$$
In particular $\Gamma_{2n}(\mathcal{S}_{\mathrm{ref}}) \le \sqrt{2}\gamma^{-1} \sqrt{d_n (\mathcal{S}_{\mathrm{ref}})}$, for any $n \ge 1$. This latter bound greatly relativizes the optimality of the greedy method. In fact, it quickly becomes apparent that the SVD method provides better approximation results than the greedy methods, at least for small values of $n$. Indeed, the optimal approaches of Kolmogorov and SVD use first basis functions that mix the pure states $\mathbf{v}(\boldsymbol{\mu})$, \ie, are linear combinations of many of them, while the greedy methods are based on spaces spanned by pure states.

However, the optimal subspace achieving the Kolmogorov width remains unreachable in practice.  Indeed, the greedy algorithm selects one function per iteration, which only yields a \emph{step by step} update to the reduced space rather than performing a global optimization over all \( N \)-dimensional subspaces. The next subsection presents a way to get the best of the greedy and SVD approaches.

\subsection{GSS: Greedy-Sampled SVD}
\label{GSS}
\minmod{Combining greedy selection with a POD/SVD compression is used in model order reduction: the POD-Greedy method \cite{Haasdonk2013, hesthaven2016certified} compresses the time trajectory inside the greedy loop, while the parallel batch greedy algorithm \cite{Reich2024} compresses batches of snapshots in a steady setting.} In our approach, we use the Greedy algorithm to select a representative set of parameter-dependent high-fidelity solutions and then apply a unique SVD to these snapshots in order to extract an optimally compressed basis that mixes the pure states. This two-step strategy allows us to combine:
\begin{inparaenum}[(i)]
  \item the explorative sampling power of Greedy,
  \item the optimal compression property of SVD.
\end{inparaenum}

\minmod{Let $\mathcal{P}_g$ be the greedy training set of \Cref{sec:greedy}, of cardinality
$M_g$. Running $M$ greedy iterations on it selects a subset $\mathcal{P}_M\subset\mathcal{P}_g$
of $M$ parameters, to which the SVD of \Cref{sec:svd} is then applied, the first $N$ modes
forming the reduced basis. Typically $M_g\gg M>N$.} The full procedure is detailed in Algorithm~\ref{alg:greedypp}. It is important to note that when \( M = N \), applying the SVD
to the Greedy-selected snapshots does not change the reduced space. In this case, both the classical Greedy method and the GSS yield the same \( N \)-dimensional subspace, even though the two methods produce \emph{different bases} for that space. We recommend \( M = 2N \) as a rule of thumb.

\begin{algorithm}[H]
\caption{GSS}
\label{alg:greedypp}
\begin{algorithmic}[1]
\State \textbf{Input:} Training set \( \mathcal{P}_g \) of size \( M_g \), number of greedy iterations \( M \), reduced basis dimension \( N < M \), initial parameter \( \boldsymbol{\mu}^{(1)} \)
\State \textbf{Output:} Basis \( \mathbf{V}\in \mathbb{R}^{\mathcal{N}\times N} \)
\State \( \mathbf{s}_{1} = \mathbf{v}(\boldsymbol{\mu}^{(1)}) \), \quad
       \( \mathbf{S} = [\mathbf{s}_{1}] \), \quad
       \( \boldsymbol{\Xi} = [\mathbf{s}_{1}/\|\mathbf{s}_{1}\|_X] \)
\For{\( n = 2 \) to \( M \)}
    \State Select \( \boldsymbol{\mu}^{(n)} \in \mathcal{P}_g \) such that:
    \[
    \boldsymbol{\mu}^{(n)} = \arg\max_{\boldsymbol{\mu} \in \mathcal{P}_g} \Delta_{n-1}(\boldsymbol{\mu})
    \]
    \State Compute \( \mathbf{s}_{n} = \mathbf{v}(\boldsymbol{\mu}^{(n)}) \)
    \State \( \boldsymbol{\xi}_n = \text{GramSchmidt}\left(\boldsymbol{\Xi},\, \mathbf{s}_{n}\right) \)
    \State \( \mathbf{S}\gets [\mathbf{S},\, \mathbf{s}_{n}] \)
    \State \( \boldsymbol{\Xi} \gets [\boldsymbol{\Xi},\,\boldsymbol{\xi}_n] \), used for computing \( \Delta_n(\boldsymbol{\mu}) \)
\EndFor
\State Apply the SVD of \Cref{sec:svd} to \( \mathbf{S} \in \mathbb{R}^{\mathcal{N}\times M} \) and extract the first \( N \) modes as the basis \( \mathbf{V} \).
\end{algorithmic}
\end{algorithm}

\paragraph{Advantages over SVD}  
As discussed in \Cref{sec:svd,sec:greedy}, a limitation of standard SVD is the need to solve the high-fidelity model over a dense sampling of the parameter space, typically chosen a priori using heuristics (\eg, uniform grid or random sampling). This can lead to inefficient coverage of the parameter domain, especially when the manifold exhibits anisotropic behavior. 

In contrast, the GSS approach uses the error estimator to guide the sampling towards the most ``informative'' parameters, reducing the number of high-fidelity solves \(M \ll M_g\) and increasing the relevance of the snapshots. By performing an SVD over these smartly chosen snapshots, the method benefits from both parameter adaptivity and optimal mode compression.

\paragraph{Advantages over Greedy}  
The GSS approach enriches the classical Greedy algorithm by introducing a post-processing compression step based on SVD, which provides an \emph{optimal low-rank approximation} of the selected snapshots, in the sense of~\cref{eq:svd_opt}.

As explained above, SVD modes are not just selected pure solutions but \emph{linear combinations of multiple solutions}, enabling the construction of "richer" basis functions. This distinction is particularly beneficial when every snapshot exhibits \emph{localized features}, as is often the case in transport-dominated problems. By combining several locally supported solutions, SVD tends to produce modes that better capture the global behavior of the system.

The GSS approach, therefore, enhances the classical Greedy method by introducing this additional compression and optimality step, without losing the parameter-adaptive strength of the greedy selection strategy.

\section{Nonlinear approximation}
\label{sec:nonlinear-approximation}
Even if care is taken to select the best possible  reduced basis, the span of which is as close as possible to the optimal spaces used in the definition of
Kolmogorov $N$-width, linear approximation techniques may suffer from the \emph{Kolmogorov barrier}, as discussed in \Cref{sec:linear-rb}. 
This is why we turn to \emph{nonlinear approximation} methods, which aim to mitigate this barrier and reduce the associated approximation complexity.

In this section, we first recall some nonlinear notions of width that generalize the classical Kolmogorov \( N\)-width to nonlinear settings. \minmod{We then fix the general form of nonlinear decoder common to the methods considered in this work, and present in that form both the \emph{NCRBA} method \cite{barnett2023, pp10, ballout2024nonlinear}, which is the main nonlinear method of interest here, and the quadratic manifold methods of \cite{barnett2022quadratic, greedy_quad}.}

\subsection{Nonlinear widths}
\label{sec:nonlinear-widths}
Model order reduction techniques can be analyzed using the general paradigm of \emph{encoder–decoder} pairs. The goal is to approximate each $u$ in a compact $\mathcal{K}\subset X$ by a low-dimensional representation using:

\begin{itemize}
    \item an \emph{encoder} $E : X \to \mathbb{R}^n$ mapping to a latent space,
    \item a \emph{decoder} $D : \mathbb{R}^n \to X$ reconstructing an approximation,
\end{itemize}

such that $u \approx D(E(u))$. The approximation quality is measured by the worst-case distortion:
\begin{equation}
\label{eq:recons_error}
    \mathcal{E}_n(\mathcal{K}; E, D) := \sup_{u \in \mathcal{K}} \| u - D(E(u)) \|_X.
\end{equation}
By minimizing this quantity over admissible encoder–decoder pairs, one obtains different notions of approximation widths.

If no restriction is imposed on $E$ and $D$ beyond continuity, we recover the \emph{nonlinear Kolmogorov $n$-width} defined in \cite{pp12} as:
\[
\delta_n(\mathcal{K}) := \inf_{\substack{E \in \mathcal{C}(\mathcal{K}, \mathbb{R}^n) \\ D \in \mathcal{C}(\mathbb{R}^n, X)}} \sup_{u \in \mathcal{K}} \| u - D(E(u)) \|_X.
\]
The smallest value of $n$ such that perfect reconstruction is possible defines \emph{the minimal latent dimension}:
\begin{equation}
\label{eq:nonlinear_dim}
    n_\delta(\mathcal{K}) := \min \left\{ n \in \mathbb{N} \;\middle|\; \delta_n(\mathcal{K}) = 0 \right\}.
\end{equation}

\begin{minmodblock}
In~\cite{pp5}, the authors establish that for a solution manifold $\mathcal{S}$ defined as the image of a \emph{Lipschitz} parameter-to-solution map \( u : \vmu \mapsto u(\vmu) \), with a compact parameter set \(\mathcal{P} \subset \mathbb{R}^p\), one has
\begin{equation}
\label{eq:nonlinear_dim_upper_bound}
n_\delta(\mathcal{S}) \le 2p + 1.
\end{equation}
If, additionally, there exists an interior point $\vmu_0 \in \mathcal{P}$ at which the parameter-to-solution map is locally injective, then
\begin{equation}
\label{eq:nonlinear_dim_bounds}
p \le n_\delta(\mathcal{S}) \le 2p + 1.
\end{equation}
If, more strongly, the map $u:\mathcal P\to\mathcal S$ is globally injective, then, since it is Lipschitz
and $\mathcal P$ is compact, it is a homeomorphism between $\mathcal P$ and
$\mathcal S$. In that case,
\[
n_\delta(\mathcal S)=p
\]
(see also \cite{pesenson2016}).
\end{minmodblock}

This framework provides a nonlinear generalization of the (linear) \emph{Kolmogorov
width}~\eqref{eq:kolmogorov_width}, \minmod{written here at a generic dimension $n$ so as to be compared with the widths above,} where the decoder $D$ is restricted to be linear and $E$
is the orthogonal projection onto an $n$-dimensional subspace $\mathcal{V}_n \subset X$.

In this work, we are particularly interested in an intermediate notion between the fully linear and fully nonlinear frameworks, namely the \emph{sensing numbers}. In this setting, the encoder is restricted to be linear---defined by \(n\) continuous linear forms of the dual space $X'$: $E = (\ell_1,\dots, \ell_n)$, while the decoder is allowed to be nonlinear. The corresponding width is defined as:
\begin{equation}
\label{eq:sensing_numbers}
    s_n(\mathcal{K}) := \inf_{\substack{D \in \mathcal{C}(\mathbb{R}^n, X) \\ \ell_1, \dots, \ell_n \in X'}} \sup_{u \in \mathcal{K}} \left\| u - D\left( \ell_1(u), \dots, \ell_n(u) \right) \right\|_X.
\end{equation}
\minmod{The associated minimal sensing dimension $n_s(\mathcal{K})$ is defined as
in~\eqref{eq:nonlinear_dim}, \emph{mutatis mutandis}.} This last notion retains the simplicity and interpretability of linear encoders, while benefiting from the expressive power of nonlinear decoders. As such, sensing numbers provide a promising compromise between classical linear approaches, governed by $d_n$, and fully nonlinear reductions associated with $\delta_n$.

By construction, the following inequalities hold:
\begin{equation}
\label{eq:width_comparison}
    \delta_n(\mathcal{K}) \le s_n(\mathcal{K}) \le d_n(\mathcal{K}),
\end{equation}
which implies, for a solution manifold $\mathcal{S}$ satisfying the assumption of~\cref{eq:nonlinear_dim_bounds},
\begin{equation}
\label{eq:dim_lower_bound}
    p \leq n_\delta(\mathcal{S}) \leq n_s(\mathcal{S}).
\end{equation}

\begin{addmodblock}
An upper bound of the same order as~\eqref{eq:nonlinear_dim_upper_bound} holds for the sensing
numbers, although it does not follow from it: the encoder being constrained to be linear, the
construction behind~\eqref{eq:nonlinear_dim_upper_bound} does not apply.

\begin{proposition}[Global bound on the sensing dimension]
\label{prop:global-sensing}
Let $X$ be a Hilbert space, let $\mathcal P\subset\mathbb R^{p}$ be compact,
and let $\vmu\mapsto u(\vmu)$ be Lipschitz from $\mathcal P$ into $X$.
Then, for $\mathcal S:=u(\mathcal P)$,
\begin{equation}
\label{eq:sensing_upper_bound}
    n_s(\mathcal S)\le 2p+1.
\end{equation}
\end{proposition}

\begin{proof}
Since $\mathcal P\subset\mathbb R^p$ is compact and
$\vmu\mapsto u(\vmu)$ is Lipschitz, the solution set
$\mathcal S=u(\mathcal P)$ is compact and Lipschitz maps do not increase upper box-counting dimension. Hence
\[
    \overline{\dim}_B(\mathcal S)
    \le
    \overline{\dim}_B(\mathcal P)
    \le p .
\]
Since $2p+1>2\,\overline{\dim}_B(\mathcal S)$, the linear embedding
theorem of Hunt and Kaloshin~\cite{hunt-kaloshin} yields a bounded linear
map
\[
    E:X\to\mathbb R^{2p+1}
\]
that is injective on $\mathcal S$.

Writing $E=(\ell_1,\ldots,\ell_{2p+1})$, with $\ell_i\in X'$, shows that
$E$ is an admissible encoder in~\eqref{eq:sensing_numbers}.
Since $\mathcal S$ is compact, $E|_{\mathcal S}$ is a homeomorphism onto
the compact, hence closed, set $E(\mathcal S)$. By the Dugundji extension theorem,
$(E|_{\mathcal S})^{-1}$ extends to a continuous map
$D:\mathbb R^{2p+1}\to X$. Consequently,
\[
    D(E(u))=u,\qquad u\in\mathcal S,
\]
and therefore $s_{2p+1}(\mathcal S)=0$.
\end{proof}

Combined with~\eqref{eq:dim_lower_bound}, and under the local injectivity assumption made
there, this gives
\begin{equation}
\label{eq:sensing_dim_bounds}
    p\le n_s(\mathcal S)\le 2p+1 ,
\end{equation}
i.e., the exact analogue of~\eqref{eq:nonlinear_dim_bounds} for sensing numbers. The smooth finite-dimensional case of \Cref{prop:global-sensing} is classical: a generic
linear projection of a compact $\mathcal C^{2}$ manifold of dimension $m$ embedded in
$\mathbb R^{\mathcal N}$ into $\mathbb R^{2m+1}$ admits a nonlinear inverse --- the projection
argument underlying Whitney's embedding theorem --- and was put to work for dimensionality
reduction in~\cite{broomhead-kirby}. \Cref{prop:global-sensing} is, however, an existence statement, the underlying embedding
theorem being a prevalence result: it asserts that $2p+1$ linear forms suffice without
exhibiting any, and says nothing about whether a computable family of measurements --- the
coefficients on the leading SVD modes, say --- attains the bound. 
\end{addmodblock}

\begin{lemma}[Translation invariance of sensing numbers]
\label[lemma]{lem:translation_invariance}
For any $a \in X$ and any compact $\mathcal{K} \subset X$,
\[
n_s(\mathcal{K}+a)=n_s(\mathcal{K}).
\]
\end{lemma}

\begin{proof}
We first show that for each $n\in\mathbb N$,
\begin{equation}\label{eq:sn-equality}
s_n(\mathcal{K}+ a)=s_n(\mathcal{K}).
\end{equation}
For any fixed $\varepsilon>0$, by definition of $s_n$ we can choose an encoder--decoder pair $(E_\varepsilon,D_\varepsilon)$ with
$E_\varepsilon:X\to\mathbb R^n$ linear and $D_\varepsilon\in\mathcal C(\mathbb R^n,X)$ such that
\[
\sup_{u \in\mathcal{K}}\,\|u -D_\varepsilon(E_\varepsilon(u))\|_X \le s_n(\mathcal{K})+\varepsilon.
\]
Define a decoder $D_{a, \varepsilon}:\mathbb R^n\to X$ by
\[
D_{a, \varepsilon}(x):=a + D_\varepsilon\bigl(x - E_\varepsilon(a)\bigr),
\]
and keep the same linear encoder $E_\varepsilon$. For any $v\in\mathcal{K}+a$ we can write $v=u+a$ with $u\in\mathcal{K}$, and then
\[
D_{a, \varepsilon}(E_\varepsilon(v))
=a + D_\varepsilon\bigl(E_\varepsilon(u+a)-E_\varepsilon(a)\bigr)
=a + D_\varepsilon\bigl(E_\varepsilon(u)\bigr).
\]
Hence
\[
\|v - D_{a, \varepsilon}(E_\varepsilon(v))\|_X
=\|u+a - (a + D_\varepsilon(E_\varepsilon(u)))\|_X
=\|u - D_\varepsilon(E_\varepsilon(u))\|_X.
\]
Taking the supremum over $v\in\mathcal{K}+a$ (equivalently $u\in\mathcal{K}$) gives
\[
\sup_{v\in\mathcal{K}+a}\|v - D_{a, \varepsilon}(E_\varepsilon(v))\|_X
=\sup_{u\in\mathcal{K}}\|u - D_\varepsilon(E_\varepsilon(u))\|_X
\le s_n(\mathcal{K})+\varepsilon,
\]
so $s_n(\mathcal{K}+a)\le s_n(\mathcal{K})+\varepsilon$. As $\varepsilon>0$ is arbitrary,
$s_n(\mathcal{K}+a)\le s_n(\mathcal{K})$.

Applying the same argument to the set $\mathcal{K}+a$ translated by
$-a$ yields $s_n(\mathcal{K}) = s_n\bigl((\mathcal{K}+a)-a\bigr)
\le s_n(\mathcal{K}+a)$. Therefore \cref{eq:sn-equality} holds. Finally, by definition $n_s(\mathcal{K})=\min\{n:\ s_n(\mathcal{K})=0\}$.
Using \cref{eq:sn-equality}, we conclude
$n_s(\mathcal{K}+a)=n_s(\mathcal{K})$.
\end{proof}

\minmod{In particular $n_s(\mathcal S_{\mathrm{ref}})=n_s(\mathcal S)$ whatever the reference state, so the centring of \Cref{sec:linear-rb} is immaterial for the sensing dimension; an analogous argument applies to $\delta_n$ and $n_\delta$, whereas the linear width $d_n$ does
depend on it.}

In~\cite{ballout2024nonlinear}, it was shown numerically for a specific test case that \( n_s(\mathcal{S}) = p \) using a linear encoder based on projection onto SVD modes. In the present work, we rigorously establish that, under suitable regularity assumptions, the sensing dimension of a sufficiently small neighbourhood of any point of $\mathcal{S}$ equals $p$. The optimal encoding dimension is therefore attainable with linear encoders locally, the global one remaining bounded only by~\eqref{eq:sensing_dim_bounds}.

\subsection{A general nonlinear decoder}
\label{sec:general-decoder}
\begin{minmodblock}
We keep the setting of \Cref{sec:linear-rb} and retain only $n\ll N$ coordinates as unknowns,
so that $\mathbf V\in\mathbb R^{\mathcal N\times n}$ in~\eqref{eq:linearapprox}; the accuracy
lost by the smaller basis is recovered not by further coordinates but by a nonlinear
correction. This is the setting of the sensing numbers $s_n$ of~\eqref{eq:sensing_numbers}: the encoder remains linear, while the decoder is allowed to be
nonlinear. The correction is built from a feature map
$\mathbf\Psi:\mathbb R^{n}\to\mathbb R^{m}$ and added to~\eqref{eq:linearapprox},
\begin{equation}
  \label{eq:nonlinearapprox}
  \widetilde{\mathbf{u}} = \mathbf u_{\mathrm{ref}} + \mathbf{V} \boldsymbol{\alpha}
  + \mathbf{W}\,\mathbf{\Psi}(\boldsymbol{\alpha}) \in \mathbb{R}^{\mathcal N},
\end{equation}
with $\mathbf W\in\mathbb R^{\mathcal N\times m}$, where $m$ is the dimension of the nonlinear
correction term, fixed by the choice of $\mathbf\Psi$ and in general unrelated to $n$.

For a symmetric matrix $\mathbf A=(a_{ij})$ of size $d$, let
$\vecsym(\mathbf A)\in\mathbb R^{d(d+1)/2}$ denote the vector containing the upper
triangular part of $\mathbf A$ row by row, that is
$\vecsym(\mathbf{A}):=(a_{11},a_{12},\cdots,a_{1d},a_{22},\cdots,a_{2d},\cdots,a_{dd})^T$.
Two feature maps are used throughout: for $\mathbf a\in\mathbb R^{d}$,
\begin{equation}
\label{eq:feature-maps}
\mathbf\Psi_{\mathrm{hom}}(\mathbf a):=\vecsym(\mathbf a\mathbf a^\top)
\in\mathbb R^{\frac{d(d+1)}{2}},
\qquad
\mathbf\Psi_{\mathrm{full}}(\mathbf a):=
\begin{bmatrix} 1\\[0.2em] \mathbf a\\[0.2em] \vecsym(\mathbf a\mathbf a^\top)\end{bmatrix}
\in\mathbb R^{1+d+\frac{d(d+1)}{2}} ,
\end{equation}
the first being \emph{homogeneous} quadratic, the second carrying in addition a constant and
a linear block.

The methods below differ only in which of the two factors in $\mathbf W\,\mathbf\Psi(\boldsymbol{\alpha})$ is prescribed and which is determined from the data. In the NCRBA of \Cref{sec:nlcrbm}, the columns of $\mathbf W$ are prescribed --- they are further reduced modes, beyond the $n$ retained in $\mathbf V$ --- and $\mathbf\Psi$ is learned. In the quadratic methods of \Cref{sec:quadmethods}, $\mathbf\Psi$ is prescribed to be $\mathbf\Psi_{\mathrm{hom}}$, so that $m=\tfrac{n(n+1)}{2}$, and $\mathbf W$ is fitted to the data. The local analysis of \Cref{sec:taylor} bears on both settings, since it characterizes the leading-order local relation between the reduced coordinates near a point of $\mathcal S$, independently of how a particular numerical method represents that relation.
\end{minmodblock}

\begin{remark}[\minmod{Offline and online coordinates}]
\label{rem:online}
\begin{minmodblock}
The coefficients $\alpha_k(\boldsymbol\mu)$ of~\eqref{eq:rb-coeffs} were defined as
projections of the centred solution $\mathbf v(\boldsymbol\mu)$. That is
how they are obtained offline, on the training set, where the solutions are available, and it
is in that sense that the maps $\mathbf\Psi$ of this section are constructed and assessed.
Online, however, $\mathbf u_{\mathcal N}(\boldsymbol\mu)$ (and thus $\mathbf v(\boldsymbol\mu)$) is precisely what is unknown, and the $n$ retained coordinates $\boldsymbol\alpha(\boldsymbol\mu)$ are no longer projections but the degrees of freedom of the reduced problem: they are determined by $n$ conditions on the residual $\mathfrak R$ of~\eqref{eq:1},
\begin{equation}
\label{eq:CRBeq}
\mathcal T_n\Bigl({\mathfrak R}(\widetilde{\mathbf u};\boldsymbol{\mu})\Bigr)=0,
\end{equation}
$\mathcal T_n$ being a rank-$n$ testing operator, for instance the Galerkin
projection associated with $\mathrm{Range}(\mathbf V)$. More generally, the $n$ online coordinates may instead be determined through another system of $n$ reduced equations or optimality conditions, such as those arising from residual minimization. The resulting problem is nonlinear and is solved iteratively. No such problem is solved in the present work, which is concerned with the validity of the representation~\eqref{eq:nonlinearapprox} itself; accordingly, the experiments of \Cref{sec:num} supply the exact projected coordinates and test only the recovery of the remaining ones.
\end{minmodblock}
\end{remark}

\subsection{Nonlinear Compressive Reduced Basis Approximation}
\label{sec:nlcrbm}

The underlying idea in the NCRBA method is that, although the representation proposed by the linear approximation method to provide a good approximation is too large, it is nevertheless optimal in a certain sense. The starting point is therefore the approximation of \Cref{sec:linear-rb} in the SVD space $\mathcal V_N$, with the coefficients~\eqref{eq:rb-coeffs}; in the ``vanilla'' version of the reduced basis method all $N$ of them are degrees of freedom of the reduced problem, in the sense of \Cref{rem:online}.

In the context where the complexity of $\cS$, linked to the complexity of the set of
parameters, is lower, the NCRBA suggests that only a few coefficients are true unknowns
(degrees of freedom) and that the other coefficients necessary for the accuracy of the
approximation are given functions of these degrees of freedom. More precisely, the first
modes, from $1$ to $n<N$, are the degrees of freedom, from which the remaining ones can be
recovered:
\[
\forall j= n+1,\dots, N, \quad
\alpha_j(\boldsymbol{\mu}) = \psi^j(\alpha_1(\boldsymbol{\mu}),\dots,\alpha_n(\boldsymbol{\mu})),
\]
the $\psi^j$ being maps that have to be suitably determined to recover the higher reduced coefficients.
\minmod{This is~\eqref{eq:nonlinearapprox} with
$\mathbf V=[\boldsymbol\varphi_1,\dots,\boldsymbol\varphi_n]$,
$\mathbf W=[\boldsymbol\varphi_{n+1},\dots,\boldsymbol\varphi_N]$,
$\mathbf\Psi=(\psi^{n+1},\dots,\psi^{N})^{\top}$ and $m=N-n$:}
\begin{equation}
\label{eq:nlcrba}
\widetilde{\mathbf u}^{\,n,N}({\boldsymbol{\mu}})
= \mathbf u_{\mathrm{ref}} + \mathbf V\,\boldsymbol{\alpha}(\boldsymbol\mu)
+ \mathbf W\,\mathbf\Psi(\boldsymbol{\alpha}(\boldsymbol\mu))
= \mathbf u_{\mathrm{ref}} + \sum_{i=1}^n \alpha_{i}(\boldsymbol{\mu}) \boldsymbol\varphi_i
+ \sum_{j=n+1}^N \psi^{j}(\alpha_{1}(\boldsymbol{\mu}),\ldots,\alpha_{n}(\boldsymbol{\mu}))
\boldsymbol\varphi_j .
\end{equation}
\minmod{The correction being carried by reduced modes, the approximation stays in the affine
space $\mathbf u_{\mathrm{ref}}+{\mathcal V}_N$.}

\medskip

In what follows, we shall investigate different ways to express the $\psi^{j}$, $j=n+1,\dots, N$ as a function of the $\alpha_{i}(\boldsymbol{\mu})$, $i=1,\dots,n$, together with the proper choice for $n\le N$. This can be $i$) a quadratic expression, $i_1$) homogeneous (as in \cite{barnett2022quadratic, geelen2023operator,greedy_quad}),
or $i_2$) non homogeneous, $ii$) a polynomial expression \cite{geelen2024learning}, or $iii$) finally a more general, learnable expression, as is proposed e.g. in \cite{pp10, barnett2023,  bensalah2025nonlinear, de2025nonlinear}.

\subsection{Quadratic methods}
\label{sec:quadmethods}
\begin{minmodblock}
Several nonlinear model reduction methods instantiate~\eqref{eq:nonlinearapprox} with the homogeneous quadratic feature map $\mathbf\Psi_{\mathrm{hom}}$ of~\eqref{eq:feature-maps},
for which $\boldsymbol{\alpha}\in\mathbb R^{n}$ and $m=\tfrac{n(n+1)}{2}$. Two main approaches in the literature employ this map differently.
\end{minmodblock}
The first, introduced in~\cite{barnett2022quadratic}, uses the first \(n\) SVD modes of the snapshot matrix \( \mathbf{S} \in \mathbb{R}^{\mathcal{N} \times M}\) as the linear basis \( \mathbf{V} \), and determines the quadratic correction matrix \( \mathbf{W}
\) by solving the least-squares problem
\begin{equation}
\label{eq:quad_opt}
\mathbf{W} = \arg\min_{\mathbf{W}' \in \mathbb{R}^{\mathcal{N} \times m}}
\sum_{i=1}^{M}\bigl\| \Pi_{\mathbf{V}}\mathbf{s}_i
+ \mathbf{W}' \mathbf{\Psi}_{\text{hom}}(\mathbf{V}^\top \mathbf{X}\,\mathbf{s}_i)
- \mathbf{s}_i \bigr\|_X^2.
\end{equation}
We refer to this method as the Quadratic SVD-based Method (QSVDM).

A second line of work~\cite{greedy_quad} observes that the SVD basis is not necessarily optimal for a given feature map. The nonlinear decoder \(\mathbf{\Psi}_{\text{hom}}\) operates on the projected data \( \mathbf V^{\top}\mathbf X\,\mathbf S \), so if \( \mathbf{V} \) is chosen independently of \(\mathbf{\Psi}_{\text{hom}}\), relevant information may be lost. To address this, the authors propose building \( \mathbf{V} \) specifically for the quadratic decoding task through a greedy selection of \(n\) vectors from the first \(n_{\mathrm{pool}} \gg n\) SVD modes. We refer to this approach as the \emph{Quadratic Greedy-based Method} (QGM).

\section{Approximation with Taylor polynomial expansions}
\label{sec:taylor}

In this section, we work in an abstract setting where  ${\mathcal S}\subset X = \mathbb{R}^{\mathcal{N}} $ is an arbitrary embedded manifold. In particular, ${\mathcal S}$ need not be associated with any parametrized PDE. In addition, the subsequent theory is developed independently of any discretization. This allows us to clarify the statements regardless of the properties of the map that to every parameter $\boldsymbol{\mu} \in {\mathcal P}$ associates the solution to the PDE. The purpose of this section is to analyse the SVD of a $q$-dimensional manifold sampled in a shrinking neighbourhood of one of its points, and to identify what the successive blocks of modes represent: the first $q$ span a subspace converging to the tangent space to ${\mathcal S}$ and, once the tangent directions are filtered out, a further $\tfrac12 q(q+1)$ span one converging to the curvature space. We then determine the functional relation between the coefficients on these two blocks. \minmod{When $\mathcal S$ is the solution manifold of~\eqref{eq:1}, we have $q = p$ provided the solution map $\boldsymbol\mu\mapsto u(\boldsymbol\mu)$ is an immersion near a reference parameter $\vmu^*$, i.e. the sensitivities $\partial u/\partial\mu_i(\boldsymbol\mu^*)$, $i=1,\dots,p$, are linearly independent in $X$; in general $q\le p$}. In a vein similar to our contribution -- though focused on a different application -- the work of \cite{tyagi2013tangent} offers particularly interesting insights. 

A consequence of our analysis is that the sensing dimension $n_s$ of a sufficiently small neighbourhood of any point of ${\mathcal S}$ equals $q$. In particular, the approximation~\eqref{eq:nlcrba} is valid, asymptotically with $n = q$ and the functions $\psi^{j}$, for $j=q+1,\dots,q+\frac{q(q+1)}{2}$  are (close to) quadratic in the $q$ first coefficients. In addition, assuming enough regularity on the manifold,  the approximation will improve by increasing the number of modes after the $q+\frac{q(q+1)}{2}$ first ones and increasing the degree of the polynomial approximation. Note, however, that this holds only asymptotically: for a wider variation of the parameters, we have to go beyond polynomial representation for the $\psi^{j}$, $j$ larger than $n = q$.

\subsection{ Notations and Assumptions}
\label{sec:assumptions}
\begin{minmodblock}
\begin{remark}[Conventions]
Throughout this section we work in the transformed variables of \Cref{sec:svd},
writing $\mathbf v$ for $\mathbf B\mathbf v$, so that $(\cdot,\cdot)_X$ is the Euclidean inner
product of $\mathbb R^{\mathcal N}$, $\Pi_{\mathbf V}=\mathbf V\mathbf V^{\top}$, and the
decompositions below are ordinary singular value decompositions. Every statement transfers to
the original variables through
$\boldsymbol\varphi_k=\mathbf B^{-1}\widetilde{\boldsymbol\varphi}_k$, the convergence rates
being unchanged since $\mathbf B$ is a fixed isomorphism.
The matrix norm $\|\cdot\|$ stands for any sub-multiplicative and unitarily invariant norm;
in particular, both the spectral norm $\|\cdot\|_2$ and the Frobenius norm $\|\cdot\|_F$
satisfy the required properties.
\end{remark}
\end{minmodblock}

Consider a manifold ${\mathcal S}$ of finite dimension $q$ (the precise value of which is not known), embedded in $\mathbb{R}^{\mathcal{N}}$ (with dimension \minmod{$ {\mathcal{N}}\gg q$}). Let $\mathbf{u}^*$ be a point of this manifold and let $U:{\mathcal Q}^0\subset\mathbb{R}^q\longrightarrow\mathbb{R}^{\mathcal{N}}$ (the inverse of a chart) be a local parametrization of ${\mathcal S}$ around it, with $ {\mathcal Q}^0 $ being an open subset of $ \mathbb{R}^q $ that contains the \minmod{closed unit ball $ \overline{\mathcal B}(\boldsymbol{0}, 1) $ of $ \mathbb{R}^q $, and such that $ U(\boldsymbol{0}) = \mathbf{u}^* $. There exists a real number $r^* \in (0,1]$
such that $U$ restricted to $\overline{\mathcal B}(\boldsymbol{0}, r^*)$ is a homeomorphism onto its image, which is a neighbourhood of $\mathbf{u}^*$ in
${\mathcal S}$. We assume that $U$ is of class ${\mathcal C}^k$ on
${\mathcal Q}^0$ for some $k \ge 3$ (third-order regularity is what the Taylor
remainders of \Cref{sec:taylor.order2} require), and that at every point of ${\mathcal Q}^0$ the Jacobian of $U$ has
full rank $q$, so that its columns span a tangent space of dimension $q$, while the second fundamental form associated with the Hessian has maximal rank, so
that the corresponding curvature space has dimension $\tfrac12 q(q+1)$.}

\minmod{To locally perform the singular value decomposition of the manifold ${\mathcal S}$, we first
select a discrete dataset
$\{\boldsymbol{\nu}^i\}_{i=1}^M\subset\overline{\mathcal B}(\boldsymbol 0,1)$
such that
\[
\mathbf P:=
[\boldsymbol{\nu}^1,\ldots,\boldsymbol{\nu}^M]\in\mathbb R^{q\times M},
\qquad
\operatorname{rank}(\mathbf P)=q.
\]
Thus the sampling excites every tangent direction. A stronger quadratic richness condition
will be introduced separately when the second-order modes are considered. The corresponding
points on the manifold are} $ \{\mathbf{u}_i\}_{i=1}^M $, where for each $ i $, $ \mathbf{u}_i = U(\boldsymbol{\nu}^i) $. Next, for any $ r $, $ 0 < r \leq r^* $, we consider the family of sampling of $ {\mathcal S} $ around $ \mathbf{u}^* $, indexed by $r$, that converges to $ \mathbf{u}^* $ as $ r \to 0 $, defined by the set of discrete points $ \{\mathbf{u}^r_i\}_{i=1}^M $, where for each $ i $, $ \mathbf{u}^r_i = U(r \cdot \boldsymbol{\nu}^i) $. This set of points can be considered as a sampling of \minmod{\(\mathcal{S}^r= \{  U(\boldsymbol{\nu}), \boldsymbol{\nu} \in \, \overline{\mathcal B}(\boldsymbol{0}, r)\subset \mathbb{R}^q\} \subset \mathcal{S}\)}.
The values of \( M \) and \( r^* \) are not addressed in this discussion; refer to \cite{tyagi2013tangent} for detailed information.

\subsection{Tangent space of the centered manifold}
\label{sec:tangentSpace}

\minmod{We now centre the manifold at $\mathbf u^{*}$, that is, we take the reference state
$\mathbf u_{\mathrm{ref}}=\mathbf u^{*}$ of \Cref{sec:linear-rb}, and consider,} for any given $r $, $ 0 < r \leq r^* $, the family of discrete points $ \{\mathbf{v}^r_i = \mathbf{u}^r_i - \mathbf{u}^*\}_{i=1}^M $ on the shifted manifold $ {\mathcal S}^{*,r} := {\mathcal S}^r  - \mathbf{u}^*  = \{ \mathbf{v}= \mathbf{u}-\mathbf{u}^*, \mathbf{u}\in {\mathcal S}^r \}\subset \mathcal{S}^*:=\mathcal S  - \mathbf{u}^* $.
The tangent space $T_0^*:= T_{\boldsymbol{0}} {\mathcal S}^* = T_{\mathbf{u}^*} {\mathcal S} $ at the origin is spanned by the partial derivatives of the parameterization function $U$ \minmod{at $\boldsymbol{0}$}. Formally, we have:
\[
T_0^* = \mathrm{span} \left( \frac{\partial U}{\partial \nu_1}(\boldsymbol{0}), \frac{\partial U}{\partial \nu_2}(\boldsymbol{0}), \dots, \frac{\partial U}{\partial \nu_q}(\boldsymbol{0}) \right),
\]
where the vectors $\frac{\partial U}{\partial \nu_j}(\boldsymbol{0})$  are the columns of the Jacobian matrix $\mathbf{J}_U(\boldsymbol{0}) \in \mathbb{R}^{\mathcal{N} \times q}$.

\subsubsection{Convergence of the first order modes}
\label{sec:modes_conv_1}
As $r\to 0$, all points $ \mathbf{u}_i^r $ (resp. $ \mathbf{v}_i^r $ ) become arbitrarily close to each other on the manifold ${\mathcal S} $ (resp. ${\mathcal S}^* $) and converge to $\mathbf{u}^*$ (resp. $      \boldsymbol{0} $). At first-order, this gives, for any $i, 1\le i\le M$  :
\[
\mathbf{v}_i^r = r \sum_{j=1}^q   \nu^i_j  \frac{\partial U}{\partial \nu_j}(\boldsymbol{0}) + \mathcal O(r^2),
\]
\minmod{as $r\to 0$, uniformly in $i$ since $U\in\mathcal C^k$ on the compact set
$\overline{\mathcal B}(\boldsymbol{0},1)$, i.e. $\|\boldsymbol\nu^i\|\le 1$.}

Let $\mathbf{S}^r= [\mathbf{v}^r_1, \mathbf{v}^r_2, \dots, \mathbf{v}^r_M]$ be a snapshot matrix of $ {\mathcal S}^{*,r}$. Using the above approximation for each $\mathbf{v}_i^r$, we have the following matrix approximation 
\begin{equation}
\label{eq:matrix_taylor_1}
    \mathbf{S}^r = \mathbf{U}^r \boldsymbol\Sigma^r \mathbf{Z}^{r \top} = r \mathbf{A}_0 + \mathcal{O}(r^2), \qquad
\mathbf{A}_0:=\mathbf{J}_U(\boldsymbol{0})\,\mathbf{P},
\end{equation}
\minmod{Both factors having rank $q$ by \Cref{sec:assumptions}, so does $\mathbf{A}_0$.} Therefore, for $r$ small enough, (i) there is a clear gap in the size of the  singular values present in $\boldsymbol\Sigma^r$ between the $q^{\text{th}}$ and the $(q+1)^{\text{th}}$, which allows to determine the value of $q$ if it is not known, (ii) from the theory of Davis-Kahan, the subspace $\mathcal U^r_{1:q}$ spanned by the first $q$ left singular vectors of $\mathbf{S}^r$ converges in $\mathcal{O}(r)$ to the tangent space $T_0^*$ (\Cref{thm:tangent_convergence} below).

The formalization of this convergence of the distance between two subspaces (or associated generating vectors or matrices) involves the notion of  
\emph{principal angles}, the definition of which is extensively recalled in Appendix~\ref{appendix:principal_angles}.

Denoting by $\mathbf{A}_0=\mathbf{U}^*\boldsymbol\Sigma^*(\mathbf{Z}^*)^{\top}$ the singular value decomposition of $\mathbf{A}_0$, and setting
\begin{equation}  
\label{eq:U1-blocks}
    \mathbf{U}_{(1)}^r := \mathbf{U}^r_{(:,1:q)}= \begin{bmatrix}\boldsymbol{\varphi}_1^r, & \cdots & ,\boldsymbol{\varphi}_q^r\end{bmatrix}, 
\qquad
 \mathbf{U}_{(1)}^* := \mathbf{U}^*_{(:,1:q)} = \begin{bmatrix}\boldsymbol{\varphi}_1^*, & \cdots & ,\boldsymbol{\varphi}_q^*\end{bmatrix},
\end{equation}
whose columns are the first \(q\) left singular vectors of \(\mathbf{S}^r\) and of \(\mathbf{A}_0\) respectively, one has
$\mathrm{Range}(\mathbf{U}_{(1)}^r) = \mathcal U^r_{1:q}$ and
$\mathrm{Range}(\mathbf{U}_{(1)}^*) = T_0^*$.

\begin{theorem}[Convergence to the tangent space]
\label{thm:tangent_convergence}
As $r \to 0$, the subspace $\mathcal U^r_{1:q}$ converges to the tangent space $T_0^*$ with rate $\mathcal{O}(r)$, i.e.
\[
\|\sin \mathbf{\Theta}\big( \mathcal U^r_{1:q},\, T_0^* \big)\| = \mathcal{O}(r).
\]
\end{theorem}
\begin{proof}
\begin{minmodblock}
Apply \Cref{cor:gram-dk} with $\mathbf{F}^r:=\tfrac1r\mathbf{S}^r$,
$\mathbf{F}:=\mathbf{A}_0$ and $k:=q$. The first-order
expansion~\eqref{eq:matrix_taylor_1} gives
$\tfrac1r\mathbf{S}^r=\mathbf{A}_0+\mathcal{O}(r)$, and the scaling by $1/r$
leaves the singular vectors unchanged. The matrix $\mathbf{A}_0$ has rank $q$,
hence $\sigma_{q+1}(\mathbf{A}_0)=0$ and the gap
$\delta_q=\sigma_q(\mathbf{A}_0)^2>0$ is independent of $r$. In view of
\eqref{eq:U1-blocks}, the corollary gives the claim.
\end{minmodblock}
\end{proof}

However, the convergence of the individual SVD modes and thus the SVD basis is more delicate and not immediate: the (norm 1) singular vectors are not uniquely defined, one can only guarantee convergence up to an orthogonal transformation, i.e., for each sufficiently small $r>0$, there exists an orthogonal matrix \(\mathbf{O}_{(1)}^r \in O(q)\), where $O(q)$ stands for the orthogonal group, such that
\begin{equation}
\label{eq:basis_convergence}
\|\mathbf{U}_{(1)}^r - \mathbf{U}_{(1)}^*\mathbf{O}_{(1)}^r\| = \mathcal{O}(r).
\end{equation}
This is established in \Cref{prop:basis-alignment}, where
\(\mathbf{O}_{(1)}^r\) is obtained as the solution of an orthogonal Procrustes problem; the role of simple versus multiple singular values is discussed in Appendix~\ref{appendix:svd_modes_convergence}.

\subsubsection{Convergence of the coefficients}
\label{sec:coeff}
With the notations above, and writing $\mathbf{v}(\boldsymbol\nu) := U(\boldsymbol\nu)-\mathbf{u}^*$,
\minmod{the linear approximation~\eqref{eq:linearapprox} of \Cref{sec:linear-rb}, taken with
$\mathbf{u}_{\mathrm{ref}}=\mathbf{u}^{*}$, $\mathbf{V}=\mathbf{U}_{(1)}^r$ and $N=q$,} reads
\begin{equation}\label{eqn:15}
U(\boldsymbol{\nu}) \approx \mathbf{u}^* + \sum_{j=1}^q \alpha_j^r (\boldsymbol{\nu}) \boldsymbol{\varphi}^{r}_j = \mathbf{u}^* +\mathbf{U}_{(1)}^r \,\boldsymbol{\alpha}_{(1)}^r(\boldsymbol{\nu}),
\end{equation}
where 
\begin{equation}\label{eq:aaa1}
\boldsymbol{\alpha}_{(1)}^r(\boldsymbol{\nu}):= \left( \alpha_1^r(\boldsymbol{\nu}), \dots, \alpha_q^r(\boldsymbol{\nu}) \right)
=  \mathbf{U}_{(1)}^{r\top}\mathbf{v}(\boldsymbol{\nu}),
\end{equation}
is the vector of projection coefficients of $\mathbf{v}(\boldsymbol{\nu})$ onto the basis $\mathbf{U}_{(1)}^r$.
We then introduce the projection coefficients onto the limiting tangent basis 
$\mathbf{U}_{(1)}^{*}$:
\begin{equation}
\label{eq:aaa2}
\boldsymbol{\alpha}_{(1)}^*(\boldsymbol{\nu})
:= \mathbf{U}_{(1)}^{*\top} \mathbf{v}(\boldsymbol{\nu}).
\end{equation}
From the assumed regularity on $U$, we observe that the first-order Taylor expansion is
\begin{equation}
\label{eq:taylor1}
U(\boldsymbol{\nu}) =  \mathbf{u}^* +  \sum_{j=1}^q \nu_j \frac{\partial U}{\partial \nu_j}(\boldsymbol{0}) +\; \mathcal{R}_2(\boldsymbol{\nu}) = \mathbf{u}^* + \mathbf{J}_U(\boldsymbol{0})\, \boldsymbol{\nu} +\; \mathcal{R}_2(\boldsymbol{\nu}),
\qquad
\|\mathcal{R}_2(\boldsymbol{\nu})\|= \mathcal{O}(\|\boldsymbol{\nu}\|^2).
\end{equation}
\minmod{Let
\begin{equation}
\label{eq:transition_matrix}
    \mathbf{A}^* := \mathbf{U}_{(1)}^{*\top}\mathbf{J}_U(\boldsymbol{0})
\end{equation}
be the change-of-basis matrix between the tangent basis
$\{\tfrac{\partial U}{\partial \nu_j}(\boldsymbol{0})\}_{j=1}^q$ and the orthonormal basis
$\{\boldsymbol{\varphi}_j^*\}_{j=1}^q$, so that $\mathbf{A}^*\boldsymbol\nu$ collects the
coordinates on the $\boldsymbol{\varphi}_j^*$ of the linear part of~\eqref{eq:taylor1}. It is
invertible, since the columns of $\mathbf{U}_{(1)}^*$ and $\mathbf{J}_U(\boldsymbol 0)$ form
two bases of the same $q$-dimensional tangent space $T_0^*$.} In addition, the basis convergence result \cref{eq:basis_convergence} implies:
\begin{equation}
\label{eq:basis_convergence2}
    \mathbf{U}_{(1)}^{r\top} = (\mathbf{O}_{(1)}^r)^\top \mathbf{U}_{(1)}^{*\top} + \mathbf{E}_{(1)}^r, 
\qquad \|\mathbf{E}_{(1)}^r\| = \mathcal{O}(r).
\end{equation}
\minmod{We also set
\begin{equation}
\label{eq:Ar}
\mathbf{A}^r := (\mathbf{O}_{(1)}^r)^\top \mathbf{A}^* ,
\end{equation}
which, like the change of basis $\mathbf{A}^*$, is invertible, with
$\|(\mathbf{A}^r)^{-1}\|=\|(\mathbf{A}^*)^{-1}\|$, the norm being unitarily invariant.}

We now state the result on the convergence of the coefficients.

\begin{proposition}[Convergence of the coefficients]
\label{prop:coeff_convergence}
For all $\boldsymbol{\nu}\in \overline{\mathcal{B}}(\boldsymbol{0},r)$, the coefficients $\boldsymbol{\alpha}_{(1)}^r(\boldsymbol{\nu})$  of the approximation of $U(\boldsymbol{\nu})$  in the basis  $\mathbf{U}_{(1)}^r$ satisfy
\begin{equation}
\label{eq:coeff_convergence}
\boldsymbol{\alpha}_{(1)}^r(\boldsymbol{\nu}) 
= (\mathbf{O}_{(1)}^r)^\top\boldsymbol{\alpha}_{(1)}^*(\boldsymbol{\nu}) + \mathbf{e}_1^r
= \mathbf{A}^r\,\boldsymbol{\nu} + \mathbf{e}_1^r + \mathbf{e}_2^r ,
\end{equation}
where $\mathbf{e}_1^r$ satisfies $\|\mathbf{e}_1^r\| = \mathcal{O}(r \|\boldsymbol{\nu}\|)$ and $\|\mathbf{e}_2^r\| =\mathcal{O}(\|\boldsymbol{\nu}\|^2)$.
Reciprocally, we have
\begin{equation}
\label{eq:nu-vs-alpha1r}
\boldsymbol{\nu}=(\mathbf{A}^r)^{-1}\boldsymbol{\alpha}_{(1)}^r(\boldsymbol{\nu}) +\mathcal O\big(r\|\boldsymbol{\alpha}_{(1)}^r(\boldsymbol{\nu})\|+\|\boldsymbol{\alpha}_{(1)}^r(\boldsymbol{\nu})\|^2\big).
\end{equation}
\end{proposition}
The first equality in~\eqref{eq:coeff_convergence} shows that, as $r \to 0$, the empirical coefficients $\boldsymbol{\alpha}_{(1)}^r(\boldsymbol{\nu})$ converge \minmod{(up to an orthogonal alignment)} to the projection  coefficients $\boldsymbol{\alpha}_{(1)}^*(\boldsymbol{\nu})$ on the tangent space, both being  nonlinear functions of $\boldsymbol{\nu}$. The second equality expresses that $\boldsymbol{\alpha}_{(1)}^r$ is linear in $\boldsymbol{\nu}$, with matrix $\mathbf{A}^r$, up to the basis misalignment $\mathcal{O}(r\|\boldsymbol{\nu}\|)$ and the Taylor remainder
$\mathcal{O}(\|\boldsymbol{\nu}\|^2)$; in particular, for
$\boldsymbol{\nu}\in\overline{\mathcal{B}}(\boldsymbol{0},r)$, $\boldsymbol{\alpha}_{(1)}^r$
is linear in $\boldsymbol{\nu}$ up to an error $\mathcal{O}(r^2)$.
\begin{proof}
From the definitions~\eqref{eq:aaa1} of $\boldsymbol{\alpha}_{(1)}^r(\boldsymbol{\nu})$ and~\eqref{eq:aaa2} of $\boldsymbol{\alpha}_{(1)}^*(\boldsymbol{\nu})$ together with the basis convergence~\eqref{eq:basis_convergence2}, we obtain
\[
\boldsymbol{\alpha}_{(1)}^r(\boldsymbol{\nu}) 
= (\mathbf{O}_{(1)}^r)^\top \mathbf{U}_{(1)}^{*\top}\mathbf{v}(\boldsymbol{\nu})
  + \mathbf{e}_1^r(\boldsymbol{\nu})
  = (\mathbf{O}_{(1)}^r)^\top \boldsymbol{\alpha}_{(1)}^*(\boldsymbol{\nu}) + \mathbf{e}_1^r(\boldsymbol{\nu}),
\]
where $ \mathbf{e}_1^r(\boldsymbol{\nu}):=\mathbf{E}_{(1)}^r \mathbf{v}(\boldsymbol{\nu})$.
Since $\|\mathbf{E}_{(1)}^r\| = \mathcal{O}(r)$ and, \minmod{by~\eqref{eq:taylor1},} $\|\mathbf{v}(\boldsymbol\nu)\| = \mathcal O(\|\boldsymbol\nu\|)$, we directly obtain the first equality.

Then, from the Taylor expansion in \cref{eq:taylor1} together with
\minmod{\eqref{eq:transition_matrix} and \eqref{eq:Ar}}, we write
\[
\boldsymbol{\alpha}_{(1)}^r(\boldsymbol{\nu}) 
= (\mathbf{O}_{(1)}^r)^\top \mathbf{U}_{(1)}^{*\top}\big(\mathbf{J}_U(\boldsymbol{0})\,\boldsymbol{\nu}+\mathcal{R}_2(\boldsymbol{\nu}) \big)
  + \mathbf{e}_1^r(\boldsymbol{\nu})
  = \mathbf{A}^r\,\boldsymbol{\nu}+ \mathbf{e}_1^r(\boldsymbol{\nu}) + \mathbf{e}_2^r(\boldsymbol{\nu}),
\]
with
$
\mathbf{e}_2^r(\boldsymbol{\nu})
:= (\mathbf{O}_{(1)}^r)^\top \mathbf{U}_{(1)}^{*\top}\mathcal{R}_2(\boldsymbol{\nu}) ,
$ that, again from~\eqref{eq:taylor1}, satisfies the bound
$\|\mathbf{e}_2^r(\boldsymbol{\nu})\|
\le c_2\,\|\boldsymbol{\nu}\|^2
= \mathcal{O}(\|\boldsymbol{\nu}\|^2)
$ and ends the proof of~\eqref{eq:coeff_convergence}.

Collecting the two error terms, we deduce that, uniformly on $\minmod{\overline{\mathcal{B}}(\boldsymbol{0},r)}$,
\[
\boldsymbol{\alpha}_{(1)}^r(\boldsymbol{\nu}) = \mathbf{A}^r\,\boldsymbol{\nu}
+\mathcal O\big(r\|\boldsymbol{\nu}\|+\|\boldsymbol{\nu}\|^2\big),
\]
hence
\[
\boldsymbol{\nu}=(\mathbf{A}^r)^{-1}\boldsymbol{\alpha}_{(1)}^r(\boldsymbol{\nu})+\mathcal O\big(r\|\boldsymbol{\nu}\|+\|\boldsymbol{\nu}\|^2\big).
\]
\minmod{Finally, taking norms in the previous identity and using~\eqref{eq:Ar}, we obtain, for some constant $C>0$
independent of $r$ and $\boldsymbol\nu$,
\[
\|\boldsymbol\nu\|\bigl(1 - Cr - C\|\boldsymbol\nu\|\bigr)
\;\le\; \|(\mathbf{A}^*)^{-1}\|\,\|\boldsymbol{\alpha}^r_{(1)}(\boldsymbol\nu)\| .
\]
Since $\|\boldsymbol\nu\|\le r$ on $\overline{\mathcal B}(\boldsymbol{0},r)$, the
left-hand factor is bounded below by $\tfrac12$ as soon as $r\le 1/(4C)$, so
that
$\|\boldsymbol\nu\| = \mathcal O\bigl(\|\boldsymbol{\alpha}^r_{(1)}(\boldsymbol\nu)\|\bigr)$
uniformly on $\overline{\mathcal B}(\boldsymbol{0},r)$. Substituting this bound into
the remainder of the previous identity yields~\eqref{eq:nu-vs-alpha1r}.}
\end{proof}

\subsubsection{Local invertibility of the coefficient map}
\label{sec:bijection}
In \Cref{sec:coeff}, we have seen that the first $q$ coefficients $\boldsymbol{\alpha}_{(1)}^r(\boldsymbol{\nu})$  of the approximation of $U(\boldsymbol{\nu})$  in the SVD basis  $\mathbf{U}_{(1)}^r$ are certainly functions of $\boldsymbol{\nu}$, whose expression has no reason to be simple, however, they behave linearly in $\boldsymbol{\nu}$ in the asymptotic regime $r \to 0$.
In this section, we go further and establish a local bijection between these coefficients and the parameter $\boldsymbol{\nu}$. 
In other words, the mapping $\boldsymbol{\nu} \mapsto \boldsymbol{\alpha}_{(1)}^r(\boldsymbol{\nu})$ is locally invertible. 

\begin{lemma}
\label[lemma]{lem:jacob_inverse}
Let $\boldsymbol{\alpha}_{(1)}^r:\minmod{\overline{\mathcal{B}}(\boldsymbol{0},r)}\to\mathbb{R}^q$
be defined as in~\eqref{eq:aaa1}. Then $\boldsymbol{\alpha}_{(1)}^r$
\minmod{is of class ${\mathcal C}^k$ and there exist $r_0>0$ and $K>0$ such that,
for every $0<r\leq r_0$, the Jacobian}
$\mathbf{J}_{\boldsymbol{\alpha}_{(1)}^r}(\boldsymbol{\nu})$ is invertible at
every point $\boldsymbol{\nu}$ of the ball
$\minmod{\overline{\mathcal{B}}(\boldsymbol{0},r)}$, \minmod{with}
\[
\sup_{\|\boldsymbol{\nu}\|\le r}\,
\big\|\big(\mathbf{J}_{\boldsymbol{\alpha}_{(1)}^r}(\boldsymbol{\nu})\big)^{-1}\big\|
\ \le\ K.
\]
\end{lemma}
\begin{proof}
\minmod{
Each component $\alpha^r_j(\boldsymbol\nu)=\big(U(\boldsymbol\nu)-\mathbf{u}^*,\boldsymbol\varphi^r_j\big)_X$
is the composition of $U$, of class $\mathcal C^k$ on $\mathcal Q^0$ by
assumption, with a bounded linear functional; since
$\overline{\mathcal B}(\boldsymbol{0},r)\subset\mathcal Q^0$, it follows that
$\boldsymbol{\alpha}^r_{(1)}\in\mathcal C^k$. In particular
$\mathbf{J}_{\boldsymbol{\alpha}^r_{(1)}}$ is continuous, a fact used below and in
\Cref{lem:alpha_injective}}. For any $j=1,\dots, q$, the gradient of \( \alpha_j^r(\boldsymbol{\nu}) \) is given by:
\[
\nabla_{\boldsymbol{\nu}}\alpha_j^r(\boldsymbol{\nu}) =  \mathbf{J}_U(\boldsymbol{\nu})^\top \boldsymbol{\varphi}_j^r.
\]
Hence, the Jacobian matrix of \( \boldsymbol{\alpha}_{(1)}^r \) at \( \boldsymbol{\nu} \) reads:
\[
\mathbf{J}_{\boldsymbol{\alpha}_{(1)}^r}(\boldsymbol{\nu}) =
\begin{bmatrix}
\nabla_{\boldsymbol{\nu}} \alpha_1^r(\boldsymbol{\nu})^\top \\
\vdots \\
\nabla_{\boldsymbol{\nu}} \alpha_q^r(\boldsymbol{\nu})^\top
\end{bmatrix}
= \mathbf{U}_{(1)}^{r\top} \mathbf{J}_U(\boldsymbol{\nu}).
\]
Using the basis convergence result~\eqref{eq:basis_convergence2} and~\eqref{eq:Ar}, we have:
\[
\mathbf{J}_{\boldsymbol{\alpha}_{(1)}^r}(\boldsymbol{\nu})-\mathbf{A}^r
= \mathbf{U}_{(1)}^{r\top}\big(\mathbf{J}_U(\boldsymbol{\nu})-\mathbf{J}_U(\boldsymbol{0})\big)
  + \mathbf{E}_{(1)}^r\,\mathbf{J}_U(\boldsymbol{0}).
\]
Taking the supremum over $\|\boldsymbol{\nu}\|\le r$, for any unitarily invariant norm we obtain
\[
\sup_{\|\boldsymbol{\nu}\|\le r}\big\|\mathbf{J}_{\boldsymbol{\alpha}_{(1)}^r}(\boldsymbol{\nu})-\mathbf{A}^r\big\|
\ \le\ \omega_1(r) + C_1 r,
\]
where $\omega_1(r):=\sup_{\|\boldsymbol{\nu}\|\le r}\|\mathbf{J}_U(\boldsymbol{\nu})-\mathbf{J}_U(\boldsymbol{0})\|\to 0$ as $r\to 0$ by continuity of $\mathbf{J}_U$ at $\boldsymbol{0}$,
and $C_1>0$ is independent of $r$.
\minmod{Since $\omega_1(r)+C_1r\to 0$ as $r\to 0$, there exists $r_0>0$ such that, for
every $0<r\le r_0$,}
\[
\theta_r:=\|(\mathbf{A}^r)^{-1}\|\;\sup_{\|\boldsymbol{\nu}\|\le r}\big\|\mathbf{J}_{\boldsymbol{\alpha}_{(1)}^r}(\boldsymbol{\nu})-\mathbf{A}^r\big\|\leq \|(\mathbf{A}^*)^{-1} \|\ (\omega_1(r) + C_1 r)  < \minmod{\frac{1}{2}},
\]
then $\mathbf{I} + (\mathbf{A}^r)^{-1}\big(\mathbf{J}_{\boldsymbol{\alpha}_{(1)}^r}(\boldsymbol{\nu})-\mathbf{A}^r\big) = (\mathbf{A}^r)^{-1}\mathbf{J}_{\boldsymbol{\alpha}_{(1)}^r}(\boldsymbol{\nu})$ reads as a perturbation of identity that is invertible with an inverse that can be written as a convergent Neumann series with norm upper bounded by $(1-\theta_r)^{-1}$.

For every $\boldsymbol{\nu}\in\minmod{\overline{\mathcal{B}}(\boldsymbol{0},r)}$,
\[
\mathbf{J}_{\boldsymbol{\alpha}_{(1)}^r}(\boldsymbol{\nu})
= \mathbf{A}^r\Big(\mathbf{I} + (\mathbf{A}^r)^{-1}\big(\mathbf{J}_{\boldsymbol{\alpha}_{(1)}^r}(\boldsymbol{\nu})-\mathbf{A}^r\big)\Big)
\]
is thus invertible and satisfies the following uniform bound
\[
\big\|\big(\mathbf{J}_{\boldsymbol{\alpha}_{(1)}^r}(\boldsymbol{\nu})\big)^{-1}\big\|
\leq\Big\|\Big(\mathbf{I} + (\mathbf{A}^r)^{-1}\big(\mathbf{J}_{\boldsymbol{\alpha}_{(1)}^r}(\boldsymbol{\nu})-\mathbf{A}^r\big)\Big)^{-1}\Big\|\,
\big\|(\mathbf{A}^r)^{-1}\big\|
\le (1-\theta_r)^{-1} \|(\mathbf{A}^r)^{-1}\|.
\]
Since $\theta_r<\tfrac12$ \minmod{by the choice of $r_0$}, we obtain the explicit uniform estimate, valid for all $\boldsymbol{\nu}\in\minmod{\overline{\mathcal{B}}(\boldsymbol{0},r)}$ :
\[
\big\|\big(\mathbf{J}_{\boldsymbol{\alpha}_{(1)}^r}(\boldsymbol{\nu})\big)^{-1}\big\|
\le (1-\theta_r)^{-1} \|(\mathbf{A}^*)^{-1}\|
\le 2\,\|(\mathbf{A}^*)^{-1}\|,
\]
which ends the proof by setting $K:= 2\,\|(\mathbf{A}^*)^{-1}\| $.
\end{proof}

At this level we can state that, for any $r$ small enough, there exists $r'>0$, with $r'\le r$ such that
the map $\boldsymbol{\alpha}_{(1)}^r$ 
is a $C^k$–diffeomorphism from
$\mathcal B(\boldsymbol{0},r') $ onto its image. In order to prove that the results remain 
true for $r'=r$, we need some global property, i.e., prove that
\begin{lemma}
\label[lemma]{lem:alpha_injective}
\minmod{
Shrinking $r_0$ from \Cref{lem:jacob_inverse} if necessary, the map
$\boldsymbol{\alpha}_{(1)}^r:\overline{\mathcal{B}}(\boldsymbol{0},r)\to\mathbb{R}^q$
is injective for every $0<r\le r_0$.}
\end{lemma}
\begin{proof}
\minmod{
Let $K>0$ be given by \Cref{lem:jacob_inverse}. Since
$\mathbf{J}_{\boldsymbol{\alpha}_{(1)}^r}(\boldsymbol{\nu})
=\mathbf{U}^{r\top}_{(1)}\mathbf{J}_U(\boldsymbol{\nu})$ and
$\mathbf{U}^{r}_{(1)}$ has orthonormal columns, we have, for all
$\boldsymbol{\nu}_1,\boldsymbol{\nu}_2$,
\[
\big\|\mathbf{J}_{\boldsymbol{\alpha}_{(1)}^r}(\boldsymbol{\nu}_1)
-\mathbf{J}_{\boldsymbol{\alpha}_{(1)}^r}(\boldsymbol{\nu}_2)\big\|
\;\le\;
\big\|\mathbf{J}_U(\boldsymbol{\nu}_1)-\mathbf{J}_U(\boldsymbol{\nu}_2)\big\| ,
\]
a bound independent of $r$. As $\mathbf{J}_U$ is uniformly continuous on the
compact set $\overline{\mathcal{B}}(\boldsymbol{0},r^*)$, the quantity
\[
\omega_2(r):=\sup_{\boldsymbol{\nu}_1,\boldsymbol{\nu}_2\in
\overline{\mathcal{B}}(\boldsymbol{0},r)}
\big\|\mathbf{J}_{\boldsymbol{\alpha}_{(1)}^r}(\boldsymbol{\nu}_1)
-\mathbf{J}_{\boldsymbol{\alpha}_{(1)}^r}(\boldsymbol{\nu}_2)\big\|
\]
tends to $0$ as $r\to 0$. Shrinking $r_0$ if necessary, we may therefore assume
that $\omega_2(r)\le \tfrac{1}{2K}$ for every $0<r\le r_0$, while
\[
\sup_{\|\boldsymbol{\nu}\|\le r}
\big\|(\mathbf{J}_{\boldsymbol{\alpha}_{(1)}^r}(\boldsymbol{\nu}))^{-1}\big\|
\le K
\]
still holds by \Cref{lem:jacob_inverse}.
}

Let $\boldsymbol{\nu}_1,\boldsymbol{\nu}_2\in\overline{\mathcal{B}}(\boldsymbol{0},r)$
and set $h:=\boldsymbol{\nu}_1-\boldsymbol{\nu}_2$,
$\gamma(t):=\boldsymbol{\nu}_2+th$ ($t\in[0,1]$), which stays in
$\overline{\mathcal{B}}(\boldsymbol{0},r)$ by convexity. By the fundamental
theorem of calculus and the chain rule,
\[
\boldsymbol{\alpha}_{(1)}^r(\boldsymbol{\nu}_1)-\boldsymbol{\alpha}_{(1)}^r(\boldsymbol{\nu}_2)
= \int_0^1 \mathbf{J}_{\boldsymbol{\alpha}_{(1)}^r}(\gamma(t))\,h\,dt
= \mathbf{J}_{\boldsymbol{\alpha}_{(1)}^r}(\boldsymbol{\nu}_2)\,h \;+\; \int_0^1\!\big(\mathbf{J}_{\boldsymbol{\alpha}_{(1)}^r}(\gamma(t))-\mathbf{J}_{\boldsymbol{\alpha}_{(1)}^r}(\boldsymbol{\nu}_2)\big)\,h\,dt.
\]
Taking norms and using the reverse triangle inequality yields
\[
\big\|\boldsymbol{\alpha}_{(1)}^r(\boldsymbol{\nu}_1)-\boldsymbol{\alpha}_{(1)}^r(\boldsymbol{\nu}_2)\big\|
\;\ge\;
\big\|\mathbf{J}_{\boldsymbol{\alpha}_{(1)}^r}(\boldsymbol{\nu}_2)h\big\|
\;-\;
\int_0^1 \big\|\mathbf{J}_{\boldsymbol{\alpha}_{(1)}^r}(\gamma(t))-\mathbf{J}_{\boldsymbol{\alpha}_{(1)}^r}(\boldsymbol{\nu}_2)\big\|\,\|h\|\,dt.
\]
\minmod{
Since $\|h\| = \|\mathbf{J}_{\boldsymbol{\alpha}_{(1)}^r}(\boldsymbol{\nu}_2)^{-1}
\mathbf{J}_{\boldsymbol{\alpha}_{(1)}^r}(\boldsymbol{\nu}_2)h\|
\le K\,\|\mathbf{J}_{\boldsymbol{\alpha}_{(1)}^r}(\boldsymbol{\nu}_2)h\|$ by the
uniform inverse bound, we get
$\|\mathbf{J}_{\boldsymbol{\alpha}_{(1)}^r}(\boldsymbol{\nu}_2)h\|\ge K^{-1}\|h\|$;
by the definition of $\omega_2(r)$ and since $\gamma(t)\in
\overline{\mathcal{B}}(\boldsymbol{0},r)$, the integral term is at most
$\omega_2(r)\|h\|$.
}
Hence
\[
\big\|\boldsymbol{\alpha}_{(1)}^r(\boldsymbol{\nu}_1)-\boldsymbol{\alpha}_{(1)}^r(\boldsymbol{\nu}_2)\big\|
\;\ge\; \big(K^{-1}-\omega_2(r)\big)\,\|h\|.
\]
With $\omega_2(r)\le \tfrac{1}{2K}$, we get
\[
\big\|\boldsymbol{\alpha}_{(1)}^r(\boldsymbol{\nu}_1)-\boldsymbol{\alpha}_{(1)}^r(\boldsymbol{\nu}_2)\big\|
\;\ge\; \tfrac{1}{2K}\,\|\boldsymbol{\nu}_1-\boldsymbol{\nu}_2\|.
\]
Therefore $\boldsymbol{\alpha}_{(1)}^r(\boldsymbol{\nu}_1)=\boldsymbol{\alpha}_{(1)}^r(\boldsymbol{\nu}_2)$ implies $\boldsymbol{\nu}_1=\boldsymbol{\nu}_2$, \ie, $\boldsymbol{\alpha}_{(1)}^r$ is injective on $\minmod{\overline{\mathcal{B}}(\boldsymbol{0},r)}$. 
\end{proof}

\begin{proposition}[Local bijectivity of the coefficient map]
\label{prop:Ck_diffeo_ball}
\minmod{
Let $r_0$ be as in \Cref{lem:alpha_injective}. Then, for every $0<r\le r_0$, the map
$\boldsymbol{\alpha}_{(1)}^r$ is a $\mathcal C^k$--diffeomorphism from
$\mathcal B(\boldsymbol 0,r)$ onto the open set
$\boldsymbol{\alpha}_{(1)}^r(\mathcal B(\boldsymbol 0,r))\subset\mathbb R^q$,
and a homeomorphism from $\overline{\mathcal B}(\boldsymbol 0,r)$ onto the
compact set
$\mathcal{A}^r:=\boldsymbol{\alpha}_{(1)}^r(\overline{\mathcal B}(\boldsymbol 0,r))$.
}
\end{proposition}
\begin{proof}
\minmod{
Let $0<r\le r_0$. There holds:
\begin{itemize}
  \item by \Cref{lem:jacob_inverse}, $\boldsymbol{\alpha}_{(1)}^r$ is of class
  $\mathcal C^k$ and its Jacobian is invertible at every point of
  $\overline{\mathcal B}(\boldsymbol 0,r)$, with uniformly bounded inverse,
  \item by \Cref{lem:alpha_injective}, $\boldsymbol{\alpha}_{(1)}^r$ is
  injective on $\overline{\mathcal B}(\boldsymbol 0,r)$.
\end{itemize}
The inverse function theorem then shows that $\boldsymbol{\alpha}_{(1)}^r$ is an
open map on $\mathcal B(\boldsymbol 0,r)$, hence a $\mathcal C^k$--diffeomorphism
onto its open image. Moreover, $\boldsymbol{\alpha}_{(1)}^r$ is a continuous
bijection from the compact set $\overline{\mathcal B}(\boldsymbol 0,r)$ onto
$\mathcal{A}^r$, which is therefore compact, and any such map is a homeomorphism.
}
\end{proof}

\subsubsection{Sensing numbers}
\label{sec:sensing}
The above proposition implies that all the information on the elements $U(\boldsymbol{\nu})$ in the manifold $\mathcal{S}$ close enough to $\mathbf{u}^*$, is present in the $q$ first coefficients of its SVD decomposition, since we can recover $\boldsymbol{\nu}$ from $\boldsymbol{\alpha}_{(1)}^r (\boldsymbol{\nu}) $.

From the perspective of sensing numbers~\eqref{eq:sensing_numbers}, this will translate into the conclusion that the minimal sensing dimension, at least locally, satisfies
$n_s(\mathcal S^{r}) = q$. 
This property is particularly relevant in the context of parametric PDEs and the NCRBA, since it establishes the validity of the reduced approximation~\eqref{eq:nlcrba}, and that all the information is contained locally in the first $n=q$ coefficients, with $q = p$ under the immersion assumption stated at the beginning of this section.

\begin{proposition}[Centered manifold: local sensing dimension]
\label{prop:sensing_centered}
The centered manifold $\mathcal S^*$ has locally a sensing dimension equal to $q$ around $\mathbf{u}^*$, \ie, 
for sufficiently small $r>0$,
\[
n_s(\mathcal S^{*,r}) = q.
\]
\end{proposition}
\begin{proof}
For $r>0$ small enough, let $E:\mathbb{R}^{\mathcal{N}} \to\mathbb{R}^q$ be the linear encoder given by the projection onto the first $q$ SVD modes of the centered snapshot matrix $\mathbf{S}^r$:
\[
E:\ \mathbf{x} \mapsto 
\mathbf{U}_{(1)}^{r\top}\mathbf{x} .
\]
The continuity of $E$ is immediate, since it is linear. \minmod{By \Cref{prop:Ck_diffeo_ball}, $\mathcal{A}^r$ is compact and
$\boldsymbol{\alpha}_{(1)}^r:\overline{\mathcal{B}}(\boldsymbol{0},r)\to\mathcal{A}^r$
is a homeomorphism; note that $\mathcal{A}^r = E(\mathcal{S}^{*,r})$, since
$\mathcal{S}^{*,r} = \{\mathbf{v}(\boldsymbol\nu),\ \boldsymbol\nu\in\overline{\mathcal{B}}(\boldsymbol{0},r)\}$.}
We then define the restriction of the decoder $D$ on $\mathcal{A}^r$ as
\[
D|_{\mathcal{A}^r}:\ \mathcal{A}^r \longrightarrow \mathbb{R}^{\mathcal{N}}, 
\qquad 
D|_{\mathcal{A}^r} := \mathbf{v}\circ(\boldsymbol{\alpha}_{(1)}^r)^{-1}.
\]  
One can further define a continuous decoder on the whole space $\mathbb{R}^q$ 
by continuously extending $D|_{\mathcal{A}^r}$ to $\mathbb{R}^q$, this is guaranteed by the Tietze extension theorem.
Then, for all  $\boldsymbol{\nu}\in \overline{\mathcal{B}}(\boldsymbol{0},r)$,
\[
(D\circ E)\bigl(\mathbf{v}(\boldsymbol{\nu})\bigr)
= \mathbf{v}\bigl((\boldsymbol{\alpha}_{(1)}^r)^{-1}(\boldsymbol{\alpha}_{(1)}^r(\boldsymbol{\nu}))\bigr)
= \mathbf{v}(\boldsymbol{\nu}).
\]
Using this couple of Encoder-Decoder the reconstruction error in \cref{eq:recons_error} is zero with $n=q$, so 
$n_s(\mathcal{S}^{*,r})\le q$. Conversely, from \cref{eq:dim_lower_bound} $n_s(\mathcal{S}^{*,r})\ge q$. We conclude $n_s(\mathcal{S}^{*,r})=q$.
\end{proof}

\begin{corollary}[Non-centered manifold: local sensing dimension]
\label{cor:sensing_noncentered}
For sufficiently small $r>0$,
\[
n_s(\mathcal S^{r}) = q.
\]
\end{corollary}

\begin{proof}
By \Cref{prop:sensing_centered}, $n_s(\mathcal{S}^{*,r})=q$. 
Since $\mathcal S^{r} = \mathbf{u}^* + \mathcal{S}^{*,r}$, the translation invariance
\Cref{lem:translation_invariance} yields
\[
n_s(\mathcal S^{r}) = n_s(\mathbf{u}^*+\mathcal{S}^{*,r}) = n_s(\mathcal{S}^{*,r}) = q.
\]
\end{proof}

\begin{addmodblock}
\begin{remark}[Exact versus empirical tangent coordinates]
\label{rem:exact-vs-empirical}
The identity $n_s(\mathcal S^{r})=q$ admits a shorter proof. Working directly with the exact
tangent basis $\mathbf U^{*}_{(1)}$, the coefficient map $\boldsymbol\alpha^{*}_{(1)}$
of~\eqref{eq:aaa2} has Jacobian
$\mathbf A^{*}=\mathbf U^{*\top}_{(1)}\mathbf J_U(\boldsymbol 0)$ at the origin, invertible
by~\eqref{eq:transition_matrix}; the inverse function theorem then makes
$\boldsymbol\alpha^{*}_{(1)}$ a local diffeomorphism, so that $q$ coordinates suffice, and none
of the convergence analysis of \Cref{sec:coeff} is required.

We have deliberately followed the longer route because it establishes the
corresponding property for the \emph{empirical} coefficients
$\boldsymbol\alpha^{r}_{(1)}$ of~\eqref{eq:aaa1}, built from the left
singular vectors of a finite sample. The exact tangent space is generally
not available in practice, whereas $\boldsymbol\alpha^{r}_{(1)}$ are the
coordinates computed and used by the NCRBA method of
\Cref{sec:nlcrbm}. The additional conclusion relevant to NCRBA is therefore
that these empirical coordinates themselves attain the minimal local
sensing dimension. This is what
\Cref{thm:tangent_convergence,prop:Ck_diffeo_ball}
provides for sufficiently small $r$.
\end{remark}
\end{addmodblock}

\begin{remark}
\label{rem:sensing}
\begin{enumerate}
\begin{addmodblock}
\item \textbf{Sensing numbers versus nonlinear Kolmogorov widths.}
Consider, for $0<L\leq 2\pi$, the closed arc of the unit circle of
angular length $L$,
\[
\mathcal S_L
=
\bigl\{
\mathbf{u}(\theta)=(\cos\theta,\sin\theta)
:\theta\in[-L/2,L/2]
\bigr\}
\subset\mathbb R^2.
\]
For $L<2\pi$, $\mathcal S_L$ is a compact one-dimensional smooth
manifold with boundary, homeomorphic to a closed interval, and its
parametrization by $\theta$ is injective. Hence
\[
n_\delta(\mathcal S_L)=1=q,
\qquad 0<L<2\pi.
\]
For $L=2\pi$, $\mathcal S_{2\pi}$ is the full unit circle and
$n_\delta(\mathcal S_{2\pi})=2$, as shown in~\cite{pp5}.

Let now $a=(\cos\varphi,\sin\varphi)$ and consider the linear form
$\ell_a(x)=a^\top x$. Along the arc,
\[
(\ell_a\circ\mathbf{u})(\theta)=\cos(\theta-\varphi).
\]
This map is injective on $[-L/2,L/2]$ for some direction $a$ if and only if $L\leq\pi$: for $L\leq\pi$ the choice $\varphi=\pi/2$ gives $\sin\theta$, which is injective there, whereas for $L>\pi$ no direction works, since
$\cos(\theta-\varphi)$ is injective only on intervals of length at most $\pi$. Consequently, one linear form suffices for arcs of angular length at most $\pi$,
whereas no one-dimensional linear encoder is injective when $L>\pi$.
Since the two coordinate forms recover $\mathbf{u}$ trivially, we obtain
\[
\begin{array}{c|c|c}
 & n_\delta(\mathcal S_L) & n_s(\mathcal S_L) \\ \hline
0<L\leq\pi      & 1=q & 1=q \\
\pi<L<2\pi      & 1=q & 2   \\
L=2\pi          & 2   & 2
\end{array}
\]

This example highlights three differences between the sensing and nonlinear
latent dimensions. First, \emph{injectivity of the parametrization is not
sufficient to guarantee $n_s=q$}. Indeed, for $\pi<L<2\pi$, the
parametrization of $\mathcal S_L$ is continuous and injective, which is
sufficient to obtain
\[
n_\delta(\mathcal S_L)=q=1
\]
by~\cite{pp5}, whereas
\[
n_s(\mathcal S_L)=2>q.
\]
Thus, the additional restriction to a linear encoder makes the sensing
dimension sensitive to the geometry of the embedding.

Second, \emph{$n_s$ is not a topological invariant}. All the sets
$\mathcal S_L$, $L<2\pi$, are homeomorphic to a closed interval and
therefore have the same nonlinear latent dimension
$n_\delta(\mathcal S_L)=1$, whereas their sensing dimension changes from
$1$ to $2$ at $L=\pi$. Hence, unlike $n_\delta$, the sensing dimension
can distinguish geometrically different embeddings of topologically
equivalent sets.

Third, \emph{the local sensing dimension does not determine the global
one}. For $L>\pi$, every sufficiently small sub-arc has
\[
n_s=q=1,
\]
in agreement with \Cref{cor:sensing_noncentered}, while
$n_s(\mathcal S_L)=2$. Equivalently, for this parametrized family,
\[
L\longmapsto n_s(\mathcal S_L)
\]
jumps at $L=\pi$, whereas
\[
L\longmapsto n_\delta(\mathcal S_L)
\]
remains equal to $1$ for $L<2\pi$ and jumps only when the arc closes
into the full circle at $L=2\pi$. Thus, the obstruction created by the
linearity of the encoder appears strictly before the topological
obstruction detected by the nonlinear width.
\end{addmodblock}

\item \textbf{Non-centered case:} 
    Even though we proved that locally around $\mathbf{u}^*\in\mathcal S$ one has $n_s(\mathcal S^r)=q$, 
    the situation differs from the centered case (\ie\ working on $\mathcal{S}^{*,r} = \mathcal S^r-\mathbf{u}^*$). 
    In the non-centered setting, the first $q$ SVD modes of the snapshot matrix are not necessarily aligned with the tangent space at $\mathbf{u}^*$, and therefore do not guarantee a bijection between the coefficients and the parameters $\boldsymbol{\nu}$. \minmod{Note that this does not contradict \Cref{cor:sensing_noncentered}: the sensing number itself is unchanged by translation (\Cref{lem:translation_invariance}),
and what fails is only that the first $q$ SVD modes of the non-centered snapshot
matrix no longer provide an admissible linear encoder.}
    In fact, one additional mode (for a total of $q+1$) may be required even locally: the first singular mode converges to the direction of $\mathbf{u}^*$ itself, while the next $q$ modes, along with $\mathbf{u}^*$, allow us to generate the tangent affine space. 
    
    \item \textbf{Conclusion:} More than $q$ SVD modes may be necessary to recover the full information carried by the manifold, \ie\ the full parameter value $\boldsymbol\nu$: \minmod{in the non-centered local setting $q+1$ modes suffice, whereas in the non-local setting no bound of this form holds in general.}

\end{enumerate}
\end{remark}

\subsection{Second order Taylor expansion of the centered manifold}
\label{sec:taylor.order2}
In this section, we extend the first-order analysis of the local manifold representation in \cref{sec:tangentSpace} to include second-order effects.
Our goal is to characterize how the second-order empirical SVD modes and their associated projection coefficients behave. This higher-order analysis reveals the emergence of the curvature subspace in the SVD structure and a corresponding quadratic dependence in the reduced coordinates.

\subsubsection{Convergence of the filtered second-order SVD modes}
Under regularity assumptions on $U$, we can extend \cref{eq:matrix_taylor_1} to second order as:
\begin{equation}\label{eq:matrix_taylor_2}
\mathbf{S}^r \;=\; r\,\mathbf{A}_0 \;+\; \frac{r^2}{2}\,\mathbf{Q}_0 \;+\; \mathcal O(r^3).
\end{equation}
Let $\mathcal I:=\{(j,\ell)\,:\,1\le j\le \ell\le q\}$ and $m=|\mathcal I|=\tfrac{q(q+1)}{2}$.
\begin{minmodblock}
The feature maps~\eqref{eq:feature-maps} are used here with $d=q$, so that
$\mathbf\Psi_{\mathrm{hom}}$ takes values in $\mathbb R^{m}$. We shall use, for
$\mathbf A\in\mathbb R^{q\times q}$, the (unique) matrix
$\mathbf T_{\mathbf A}\in\mathbb R^{m\times m}$ defined by
\begin{equation}
\label{eq:TA}
\vecsym\!\big(\mathbf A\,\mathbf{Y}\,\mathbf A^{\top}\big)
=\mathbf T_{\mathbf A}\,\vecsym(\mathbf{Y}),
\qquad \forall\, \mathbf{Y} \in\mathbb S^{q},
\end{equation}
$\mathbb S^{q}$ being the space of real symmetric $q\times q$ matrices.
Equivalently
$\mathbf T_{\mathbf A}\,\mathbf\Psi_{\mathrm{hom}}(\mathbf a)
=\mathbf\Psi_{\mathrm{hom}}(\mathbf A\mathbf a)$, and
$\mathbf A\mapsto\mathbf T_{\mathbf A}$ is multiplicative,
$\mathbf T_{\mathbf A\mathbf B}=\mathbf T_{\mathbf A}\mathbf T_{\mathbf B}$,
so that $\mathbf T_{\mathbf A^{-1}}=\mathbf T_{\mathbf A}^{-1}$ when $\mathbf A$ is invertible.

The quadratic feature matrix $\mathbf{P}_2\in\mathbb R^{m\times M}$ is then
\[
\mathbf{P}_2:=\Big[\ \mathbf\Psi_{\mathrm{hom}}(\boldsymbol{\nu}^1)\ ,\ \cdots\ ,\ \mathbf\Psi_{\mathrm{hom}}(\boldsymbol{\nu}^M)\ \Big].
\]
\end{minmodblock}
Define the (symmetrized) quadratic Taylor basis matrix $\mathbf{Q}_{\mathrm{base}}\in\mathbb R^{\mathcal{N}\times m}$ by
\[
\mathbf{Q}_{\mathrm{base}}
:=\Big[\ (2-\delta_{j\ell})\,\partial_{\nu_j}\partial_{\nu_\ell}U(\boldsymbol{0})\ \Big]_{(j,\ell)\in\mathcal I},
\]
with the same column ordering as $\vecsym(\cdot)$.
Finally, define
\[
\mathbf{Q}_0:=\mathbf{Q}_{\mathrm{base}}\,\mathbf{P}_2\in\mathbb R^{\mathcal{N}\times M}.
\]
Equivalently, the $i$-th column of $\mathbf{Q}_0$, $i=1,\dots, M$, is $D^2U(\boldsymbol{0})[\boldsymbol{\nu}^i,\boldsymbol{\nu}^i]$.

Here, we focus on the second block of left singular vectors of $\mathbf{S}^r$, which captures the curvature effects of the manifold beyond the first SVD order. To isolate this second-order behavior, we perform the analysis on a projected snapshot matrix $\widehat{\mathbf{S}}^{\,r} \;:=\; (\mathbf{I}-\Pi_0)\,\mathbf{S}^r$, where $\Pi_0= \mathbf{U}_{(1)}^*  \mathbf{U}_{(1)}^{*\top}$ is the orthogonal projection onto the tangent space. Since $(\mathbf{I}-\Pi_0)\mathbf{J}_U(0)=0$, left-multiplying \cref{eq:matrix_taylor_2} by $(\mathbf{I}-\Pi_0)$ cancels the linear term 
(remember the definition of $\mathbf{A}_0$ in~\eqref{eq:matrix_taylor_1}):
\[
\widehat{\mathbf{S}}^{\,r}
=(\mathbf{I}-\Pi_0)\mathbf{S}^r
=\frac{r^2}{2}\,(\mathbf{I}-\Pi_0)\mathbf{Q}_0+\mathcal O(r^3)
=\frac{r^2}{2}\,(\mathbf{I}-\Pi_0)\mathbf{Q}_{\mathrm{base}}\mathbf{P}_2+\mathcal O(r^3).
\]
By the assumption on the second fundamental form associated with the Hessian matrix, $(\mathbf{I}-\Pi_0)\mathbf{Q}_{\mathrm{base}}$ has (full) rank $m$. 
We further assume that the sampling is \emph{quadratically rich} in the sense that $\mathrm{rank}(\mathbf{P}_2)= m$ (in particular $M\ge m$),
so that $\mathrm{Range}\big((\mathbf{I}-\Pi_0)\mathbf{Q}_0\big)=\mathrm{Range}\big((\mathbf{I}-\Pi_0)\mathbf{Q}_{\mathrm{base}}\big)$. Next, we define
\begin{equation} \label{eqU2*}
 \widehat{\mathbf{U}}_{(2)}^r := \begin{bmatrix}\hat{\boldsymbol{\varphi}}_1^r, & \cdots & ,\hat{\boldsymbol{\varphi}}_m^r\end{bmatrix},\quad
 \mathbf{U}_{(2)}^* := \begin{bmatrix}\boldsymbol{\varphi}_{q+1}^* ,& \cdots & ,\boldsymbol{\varphi}_{q+m}^*\end{bmatrix},
\end{equation}
where \(\{\hat{\boldsymbol{\varphi}}_j^r\}_{j=1}^m\) and $\{\boldsymbol{\varphi}_{q+j}^*\}_{j=1}^m$ denote the first \(m\) singular vectors of $\widehat{\mathbf{S}}^{\,r}$ and of \((\mathbf{I}-\Pi_0) \mathbf{Q}_0\), respectively. The limit curvature modes are numbered $q+1,\dots,q+m$ so as to continue the tangent block in the combined basis $[\mathbf{U}_{(1)}^*,\mathbf{U}_{(2)}^*]$; they are, however, the left singular vectors of $(\mathbf{I}-\Pi_0)\mathbf{Q}_0$ and not of $\mathbf{A}_0$. The basis $\mathbf{U}_{(2)}^* \in \mathbb{R}^{\mathcal{N} \times m}$ spans the curvature subspace associated with the second-order variations of the manifold, defined as
\begin{equation}\label{curlC22}
 \mathcal C \;:=\; \mathrm{span}\Big\{(\mathbf{I}-\Pi_0)\,\partial_{\nu_j}\partial_{\nu_\ell}U(\boldsymbol{0})\;:\;1\le j\le \ell\le q\Big\}  = \mathrm{Range}(\mathbf{U}_{(2)}^*). 
\end{equation}

\begin{theorem}[Convergence of the filtered curvature block]
\label{thm:DK_curvature}
As $r \to 0$, the subspace $\widehat{\mathcal U}^{r (2)}_{1:m}$ spanned by the $m$ first  left singular vectors of $\widehat{\mathbf{S}}^{\,r}$
converges to the curvature subspace $\mathcal C$ with rate $\mathcal{O}(r)$, i.e.
\begin{equation}
\label{22KK10}
\big\|\sin \mathbf{\Theta}\big(\widehat{\mathcal U}^{r (2)}_{1:m},\,\mathcal C\big)\big\|
=\; \mathcal O(r).
\end{equation}

\end{theorem}
\begin{proof}
\begin{minmodblock}
Apply \Cref{cor:gram-dk} with
$\mathbf{F}^r:=\tfrac{2}{r^2}\,\widehat{\mathbf{S}}^{\,r}$,
$\mathbf{F}:=(\mathbf{I}-\Pi_0)\mathbf{Q}_0$ and $k:=\tfrac{q(q+1)}{2}$. The
second-order expansion above gives
$\tfrac{2}{r^2}\widehat{\mathbf{S}}^{\,r}=(\mathbf{I}-\Pi_0)\mathbf{Q}_0
+\mathcal O(r)$, and the scaling by $2/r^2$ leaves the singular vectors
unchanged. Since $(\mathbf{I}-\Pi_0)\mathbf{Q}_{\mathrm{base}}$ has rank $m$ and
the sampling is quadratically rich, $(\mathbf{I}-\Pi_0)\mathbf{Q}_0$ has rank
$m$, so that $\sigma_{m+1}\bigl((\mathbf{I}-\Pi_0)\mathbf{Q}_0\bigr)=0$ and the
gap $\delta_m=\sigma_m\bigl((\mathbf{I}-\Pi_0)\mathbf{Q}_0\bigr)^2>0$ is
independent of $r$. As
$\widehat{\mathcal U}^{r (2)}_{1:m}=\mathrm{Range}(\widehat{\mathbf{U}}_{(2)}^r)$
and $\mathcal C=\mathrm{Range}(\mathbf{U}_{(2)}^*)$ by \eqref{curlC22}, the
corollary gives the claim.
\end{minmodblock}
\end{proof}

\begin{remark}[Unfiltered curvature block]
\label{rem:unfiltered-curvature}
A natural question is whether the same $\mathcal O(r)$ convergence holds for the \emph{unfiltered} second block of left singular vectors $\mathbf{U}_{(2)}^r$ defined by
\begin{equation} \label{eqU2r}
     \mathbf{U}_{(2)}^r := \begin{bmatrix}\boldsymbol{\varphi}_{q+1}^r, & \cdots & ,\boldsymbol{\varphi}_{q+m}^r\end{bmatrix},
\end{equation}
where the columns \(\{\boldsymbol{\varphi}_{q+j}^r\}_{j=1}^m\) denote the second set of   $m$  left singular vectors of $\mathbf{S}^r$, without first projecting out the tangent space. This requires a more delicate analysis involving alignment matrices and careful tracking of cross-terms. However, numerical experiments suggest that the same $\mathcal O(r)$ convergence rate holds in practice for $\mathbf{U}_{(2)}^r$ as well, see \cref{sec:conv.taylor}.
\end{remark}

\subsubsection{Quadratic law for the limit curvature coefficients}
\label{sec:2ndorder}

Having established the convergence of the second-order SVD modes $\widehat{\mathbf{U}}_{(2)}^r$ to the curvature subspace $\mathcal C$ in the previous section, we now turn to a fundamental question: \emph{what is the functional dependence of the curvature coefficients on the tangent coefficients?}

This question is central to nonlinear model reduction methods such as NCRBA (see \cref{sec:nlcrbm}) and quadratic-manifold ROMs~\cite{barnett2022quadratic, greedy_quad, geelen2023operator}, which postulate that higher-order coefficients can be expressed as polynomial functions of the leading reduced coordinates. 

Specifically, we show that the projection coefficients onto the curvature subspace --- corresponding to SVD modes $q+1$ through $q+m$ --- are quadratic in the tangent coefficients up to higher-order terms. 

We first consider the limit case, which is independent of $r$. In \Cref{sec:2ndorder-filtered-coeff}, we then extend the result to the filtered case.

\begin{theorem}[Quadratic law for the limit curvature coefficients]
\label{thm:quadratic-law-curvature}
With $\mathbf{U}_{(2)}^*\in\mathbb R^{\mathcal{N}\times m}$ as in~\eqref{eqU2*},
the limit coefficients
$\boldsymbol{\alpha}_{(2)}^*(\boldsymbol{\nu}):=\mathbf{U}_{(2)}^{*\top}\mathbf{v}(\boldsymbol{\nu})\in\mathbb R^m$
satisfy
\begin{equation} 
\label{eqthm10}
   \boldsymbol{\alpha}_{(2)}^*(\boldsymbol{\nu})
=
\mathbf R\,\mathbf T_{\mathbf A^*}^{-1}\,
\mathbf\Psi_{\mathrm{hom}}\!\big(\boldsymbol{\alpha}_{(1)}^*(\boldsymbol{\nu})\big)
+\mathcal O\big(\|\boldsymbol{\alpha}_{(1)}^*(\boldsymbol{\nu})\|^3\big),
\end{equation}
where $\boldsymbol{\alpha}_{(1)}^*(\boldsymbol{\nu})$ is defined in
\eqref{eq:aaa2}, $\mathbf{A}^*$ is defined in~\eqref{eq:transition_matrix}
and $\mathbf{R}\in\mathbb R^{m\times m}$ is given by
\[ \forall (j,\ell)\in\mathcal I,\ k=1,\dots,m, \quad
\mathbf{R}_{k,(j,\ell)}=\tfrac12\,(2-\delta_{j\ell})\big(\boldsymbol{\varphi}_{q+k}^*,\,
(\mathbf{I}-\Pi_0)\,\partial_{\nu_j}\partial_{\nu_\ell}U(\boldsymbol{0})\big)_X .
\]
In particular, the leading term is \emph{homogeneous quadratic} in $\boldsymbol{\alpha}_{(1)}^*(\boldsymbol{\nu})$. \minmod{Moreover, $\mathbf R$ is nonsingular. Hence this leading  term is nondegenerate: under the full-rank curvature assumption it cannot vanish identically.}
\end{theorem}
\begin{proof}
Projecting the Taylor expansion 
\begin{equation}\label{1-2+}
    \mathbf{v}(\boldsymbol{\nu})=\mathbf{J}_U(\boldsymbol{0})\boldsymbol{\nu}
+\tfrac12 D^2U(\boldsymbol{0})[\boldsymbol{\nu},\boldsymbol{\nu}]+\mathcal O(\|\boldsymbol{\nu}\|^3)
\end{equation}
onto $\mathbf{U}_{(1)}^*$ yields
$
\boldsymbol{\alpha}_{(1)}^*(\boldsymbol{\nu})
=\mathbf{A}^*\,\boldsymbol{\nu}+\mathcal O(\|\boldsymbol{\nu}\|^2),
$
hence, following the same lines as in \Cref{prop:coeff_convergence}, we deduce
$\boldsymbol{\nu}=(\mathbf{A}^*)^{-1}\boldsymbol{\alpha}_{(1)}^*(\boldsymbol{\nu})
+\mathcal O(\|\boldsymbol{\alpha}_{(1)}^*(\boldsymbol{\nu})\|^2)$.

Let us now combine the definition 
$\boldsymbol{\alpha}_{(2)}^*(\boldsymbol{\nu}):=\mathbf{U}_{(2)}^{*\top}\mathbf{v}(\boldsymbol{\nu})$ with
\eqref{1-2+}, and remark that $\mathbf{U}_{(2)}^*\perp \mathrm{Range}(\mathbf{J}_U(\boldsymbol{0}))$, it follows that the linear term vanishes and
\[
\boldsymbol{\alpha}_{(2)}^*(\boldsymbol{\nu})
=\tfrac12\,\mathbf{U}_{(2)}^{*\top}(\mathbf{I}-\Pi_0)D^2U(\boldsymbol{0})[\boldsymbol{\nu},\boldsymbol{\nu}]
+\mathcal O(\|\boldsymbol{\nu}\|^3)
=\mathbf{R}\,\mathbf\Psi_{\mathrm{hom}}(\boldsymbol{\nu})+\mathcal O(\|\boldsymbol{\nu}\|^3).
\]
Substituting $\boldsymbol{\nu}=(\mathbf A^*)^{-1}\boldsymbol{\alpha}_{(1)}^*(\boldsymbol{\nu})
+\mathcal O(\|\boldsymbol{\alpha}_{(1)}^*(\boldsymbol{\nu})\|^2)$, and using $\mathbf\Psi_{\mathrm{hom}}\big((\mathbf A^*)^{-1}\mathbf a\big)
=\mathbf T_{\mathbf A^*}^{-1}\mathbf\Psi_{\mathrm{hom}}(\mathbf a)$
from~\eqref{eq:TA}, yields the claim.
\end{proof}

\subsubsection{Quadratic law for the \emph{filtered} curvature coefficients}
\label{sec:2ndorder-filtered-coeff}

Let $\widehat{\mathbf{U}}_{(2)}^{r}$ be defined 
in~\eqref{eqU2*}, and
$\mathcal C$ be defined in~\eqref{curlC22}, from \Cref{thm:DK_curvature}, $\widehat{\mathcal U}^{r (2)}_{1:m}=\mathrm{Range}(\widehat{\mathbf{U}}_{(2)}^r)$
converges to $\mathcal C$
at rate $\mathcal O(r)$. 
By \Cref{prop:basis-alignment}, we deduce from~\eqref{22KK10} that 
there exists $\mathbf{O}_{(2)}^r\in O(m)$ such that
\begin{equation}
\label{eq:U2hat-align}
\widehat{\mathbf{U}}_{(2)}^r=\mathbf{U}_{(2)}^*\,\mathbf{O}_{(2)}^r+\mathbf{E}_{(2)}^r,
\qquad \|\mathbf{E}_{(2)}^r\|=\mathcal O(r).
\end{equation}
Moreover, because $\widehat{\mathbf{U}}_{(2)}^r$ is a left singular block of
$\widehat{\mathbf{S}}^{\,r}=(\mathbf{I}-\Pi_0)\mathbf{S}^r$, we have $\Pi_0\widehat{\mathbf{U}}_{(2)}^r=0$.
Since also $\Pi_0\mathbf{U}_{(2)}^*=0$, it follows that $\Pi_0\mathbf{E}_{(2)}^r=0$, i.e.
$\mathbf{E}_{(2)}^r=(\mathbf{I}-\Pi_0)\mathbf{E}_{(2)}^r$.

\begin{theorem}[Quadratic law for the filtered curvature coefficients]
\label{thm:filtered-quadratic-law}
The coefficients 
$\widehat{\boldsymbol{\alpha}}_{(2)}^r(\boldsymbol{\nu}):=\widehat{\mathbf{U}}_{(2)}^{r\top}\mathbf{v}(\boldsymbol{\nu})\in\mathbb R^m$
satisfy, uniformly for $\boldsymbol{\nu} \in \overline{\mathcal{B}}(\boldsymbol{0},r)$,
\begin{equation}
\label{eq:filtered-quadratic-law}
\widehat{\boldsymbol{\alpha}}_{(2)}^r(\boldsymbol{\nu})
=
\mathbf R^r\,\mathbf T_{\mathbf{A}^r}^{-1}\,
\mathbf\Psi_{\mathrm{hom}}\!\big(\boldsymbol{\alpha}_{(1)}^r(\boldsymbol{\nu})\big)
\;+\;
\mathcal O\!\Big(r\,\|\boldsymbol{\alpha}_{(1)}^r(\boldsymbol{\nu})\|^2
+\|\boldsymbol{\alpha}_{(1)}^r(\boldsymbol{\nu})\|^3\Big),
\end{equation}
where $\boldsymbol{\alpha}_{(1)}^r(\boldsymbol{\nu})$
is defined in~\eqref{eq:aaa1},  $\mathbf{R}^r := (\mathbf{O}_{(2)}^r)^\top \mathbf{R} \in\mathbb R^{m\times m}$, with $\mathbf{A}^r$ of~\eqref{eq:Ar} and $\mathbf{R}$ of \Cref{thm:quadratic-law-curvature}.
In particular, the leading term is \emph{homogeneous quadratic} in $\boldsymbol{\alpha}_{(1)}^r(\boldsymbol{\nu})$, and the
remainder is $\mathcal O(r^3)$ when $\|\boldsymbol{\nu}\|\le r$.
\end{theorem}

\begin{proof}
Let us reduce the coefficient to the normal component of $\mathbf{v}(\boldsymbol{\nu})$.
Using~\eqref{eq:U2hat-align},
\[
\widehat{\boldsymbol{\alpha}}_{(2)}^r(\boldsymbol{\nu})
=\widehat{\mathbf{U}}_{(2)}^{r\top}\mathbf{v}(\boldsymbol{\nu})
=(\mathbf{O}_{(2)}^r)^\top \underbrace{\mathbf{U}_{(2)}^{*\top}\mathbf{v}(\boldsymbol{\nu})}_{=:\,\boldsymbol{\alpha}_{(2)}^*(\boldsymbol{\nu})}
+ \mathbf{E}_{(2)}^{r\top}\mathbf{v}(\boldsymbol{\nu}) = 
(\mathbf{O}_{(2)}^r)^\top \boldsymbol{\alpha}_{(2)}^*(\boldsymbol{\nu})+ 
\mathbf{E}_{(2)}^{r\top}(\mathbf{I}-\Pi_0)\mathbf{v}(\boldsymbol{\nu}).
\]
From the Taylor expansion of $\mathbf{v}$ and $(\mathbf{I}-\Pi_0)\mathbf{J}_U(0)=0$,
\[
(\mathbf{I}-\Pi_0)\mathbf{v}(\boldsymbol{\nu})
=\frac12\,(\mathbf{I}-\Pi_0)D^2U(0)[\boldsymbol{\nu},\boldsymbol{\nu}]+\mathcal O(\|\boldsymbol{\nu}\|^3)
=\mathcal O(\|\boldsymbol{\nu}\|^2).
\]
Therefore, from~\eqref{eq:U2hat-align}
\begin{equation}
\label{eq:alpha2hat-vs-alpha2star}
\widehat{\boldsymbol{\alpha}}_{(2)}^r(\boldsymbol{\nu})
=(\mathbf{O}_{(2)}^r)^\top \boldsymbol{\alpha}_{(2)}^*(\boldsymbol{\nu})
+\mathcal O\big(r\,\|\boldsymbol{\nu}\|^2\big).
\end{equation}
Now we use the limit quadratic law of \Cref{thm:quadratic-law-curvature}, in the
form established in its proof,
\[
\boldsymbol{\alpha}_{(2)}^*(\boldsymbol{\nu})
=\mathbf{R}\,\mathbf\Psi_{\mathrm{hom}}(\boldsymbol{\nu})+\mathcal O(\|\boldsymbol{\nu}\|^3),
\]
hence~\eqref{eq:alpha2hat-vs-alpha2star} gives
\begin{equation}
\label{eq:alpha2hat-nu}
\widehat{\boldsymbol{\alpha}}_{(2)}^r(\boldsymbol{\nu})
=\mathbf{R}^r\,\mathbf\Psi_{\mathrm{hom}}(\boldsymbol{\nu})+\mathcal O\big(r\|\boldsymbol{\nu}\|^2+\|\boldsymbol{\nu}\|^3\big),
\qquad \mathbf{R}^r:=(\mathbf{O}_{(2)}^r)^\top\mathbf{R}.
\end{equation}
Next, we use the expression of $\boldsymbol{\nu}$ in terms of the empirical
first-order coefficients~\eqref{eq:nu-vs-alpha1r}. Since
$\mathbf\Psi_{\mathrm{hom}}$ is quadratic and
$\mathbf\Psi_{\mathrm{hom}}\big((\mathbf{A}^r)^{-1}\boldsymbol{\alpha}\big)=\mathbf T_{\mathbf{A}^r}^{-1}\mathbf\Psi_{\mathrm{hom}}(\boldsymbol{\alpha})$
by~\eqref{eq:TA}, we deduce
\[
\mathbf\Psi_{\mathrm{hom}}(\boldsymbol{\nu})
=
\mathbf T_{\mathbf{A}^r}^{-1}\,\mathbf\Psi_{\mathrm{hom}}(\boldsymbol{\alpha})
+\mathcal O\big(r\|\boldsymbol{\alpha}\|^2+\|\boldsymbol{\alpha}\|^3\big),
\]
with $\boldsymbol{\alpha}=\boldsymbol{\alpha}_{(1)}^r(\boldsymbol{\nu})$.
Combining with~\eqref{eq:alpha2hat-nu} and using again $\|\boldsymbol{\nu}\|=\mathcal O(\|\boldsymbol{\alpha}\|)$ yields \eqref{eq:filtered-quadratic-law}.
\end{proof}

\begin{addmodblock}
\begin{remark}[Unfiltered higher coefficients]
\label{rem:unfiltered-coefficients}
The homogeneous quadratic law of
\Cref{thm:quadratic-law-curvature,thm:filtered-quadratic-law} relies on filtering the tangent
space out \emph{exactly}, through $\mathbf I-\Pi_0$. For the ordinary unfiltered SVD this
cancellation is no longer exact, and, in contrast with \Cref{rem:unfiltered-curvature}, this
can be seen without assuming any convergence of the higher modes.

Indeed, \Cref{prop:Ck_diffeo_ball} shows that the first $q$ empirical coefficients provide
local coordinates. Hence the higher empirical coefficients admit a standard local expansion
in \( \mathbf a:=\boldsymbol{\alpha}_{(1)}^r(\boldsymbol\nu).\)
What matters here is the linear part of this expansion, which, unlike in the exactly filtered
setting, is not annihilated by construction. Define
\[
\boldsymbol{\alpha}_{(2)}^r(\boldsymbol\nu)
:=\mathbf U_{(2)}^{r\top}\mathbf v(\boldsymbol\nu),
\qquad
\mathbf J^r := \mathbf J_{\boldsymbol{\alpha}^r_{(1)}}(\boldsymbol 0)
= \mathbf U_{(1)}^{r\top}\mathbf J_U(\boldsymbol 0),
\]
with $\mathbf U_{(2)}^r$ as in \eqref{eqU2r}. Taylor expansion of $U$ gives
\[
\mathbf v(\boldsymbol\nu)
=
\mathbf J_U(\boldsymbol 0)\boldsymbol\nu
+\frac12D^2U(\boldsymbol 0)[\boldsymbol\nu,\boldsymbol\nu]
+\mathcal O(\|\boldsymbol\nu\|^3),
\]
while
\[
\boldsymbol\nu
=
(\mathbf J^r)^{-1}\mathbf a+\mathcal O(\|\mathbf a\|^2)
\]
uniformly on $\overline{\mathcal{B}}(\boldsymbol 0,r)$. Consequently,
\begin{equation}
\label{eq:local-graph}
\boldsymbol{\alpha}_{(2)}^r(\boldsymbol\nu)
=
\mathbf L^r\,\mathbf a
+\frac12\,\mathbf U_{(2)}^{r\top}D^2U(\boldsymbol 0)
\big[(\mathbf J^r)^{-1}\mathbf a,(\mathbf J^r)^{-1}\mathbf a\big]
+\mathcal O\!\left(
r\|\mathbf a\|^2+\|\mathbf a\|^3
\right),
\end{equation}
where
\begin{equation}
\label{eq:Lr}
\mathbf L^r
:=
\mathbf U_{(2)}^{r\top}\mathbf J_U(\boldsymbol 0)(\mathbf J^r)^{-1}.
\end{equation}
The constants in these estimates are independent of sufficiently small $r$ by
\Cref{lem:jacob_inverse} and the local boundedness of $D^2U$ and $D^3U$. The same
computation applies verbatim to any further block
$[\boldsymbol\varphi_{q+1}^r,\dots,\boldsymbol\varphi_N^r]$, $q<N\le\mathcal N$, hence to
every higher coefficient in~\eqref{eq:nlcrba}.

The linear coefficient is small, but unlike in the exactly filtered setting it is not forced
to vanish by construction. Since the columns of $\mathbf U_{(2)}^r$ are orthogonal to the
empirical tangent space, while those of $\mathbf J_U(\boldsymbol 0)$ span the exact one, we
have $\mathbf U_{(2)}^{r\top}\Pi_0^r=\boldsymbol 0$ with
$\Pi_0^r:=\mathbf U_{(1)}^r\mathbf U_{(1)}^{r\top}$, and
$\Pi_0\mathbf J_U(\boldsymbol 0)=\mathbf J_U(\boldsymbol 0)$, whence
\[
\mathbf U_{(2)}^{r\top}\mathbf J_U(\boldsymbol 0)
=\mathbf U_{(2)}^{r\top}(\Pi_0-\Pi_0^r)\,\mathbf J_U(\boldsymbol 0),
\]
so that \Cref{thm:tangent_convergence} yields $\|\mathbf L^r\|=\mathcal O(r)$. Equivalently,
$\boldsymbol{\alpha}_{(2)}^r(\boldsymbol\nu)
=\mathbf U_{(2)}^{r\top}(\mathbf I-\Pi_0^r)\mathbf v(\boldsymbol\nu)$: the raw block is
already deflated, but by $\Pi_0^r$ instead of $\Pi_0$. Exact filtering makes the linear coefficient
vanish identically; empirical deflation only makes it small.

Homogeneity is therefore not enforced in raw empirical coordinates: a linear contribution is
permitted. Since $\|\mathbf a\|=\mathcal O(r)$ on the ball, both this possible linear
contribution and the quadratic contribution have an $\mathcal O(r^2)$ upper bound, while the remainder is $\mathcal O(r^3)$. This does not assert that $\mathbf L^r\neq\boldsymbol 0$, nor that
either order is attained; whether the linear contribution is significant is problem
dependent and is examined numerically in \Cref{sec:enh_num}. Together these motivate, in
\Cref{sec:quadratic}, feature maps that do not impose homogeneity a priori.
\end{remark}
\end{addmodblock}

\section{Discussion on quadratic approximation}
\label{sec:quadratic}
In \Cref{sec:taylor.order2}, we have shown that, in an SVD-based representation, the second block of coefficients is asymptotically (for small variations of the parameter) quadratic in the first block. This observation provides a theoretical justification for the quadratic ansatz that is used in, e.g.,~\cite{barnett2022quadratic, greedy_quad, geelen2023operator}, to represent the expressions of the coefficients of the high modes as a function of the coefficients of the low modes. If indeed this is natural in the \textit{Taylor} regime, this has clearly limitations as 
\begin{itemize}
    \item this expansion is valid locally for the variation of the parameter close to a reference value,
    \item \minmod{over wider parameter ranges, higher-order or more general nonlinear terms may be
    required,}
    \item the number of low and high modes is not clear,
    \item \minmod{the homogeneity of the ansatz rests on an exact filtering of the tangent space, which is not available when the basis is built from data,}
    \item \minmod{the ansatz is expanded about a reference state whose choice is left open, and the homogeneous map cannot correct a poor one.}
\end{itemize}
\minmod{These approaches were described in \Cref{sec:quadmethods}; in what follows we discuss
the points above.}

\subsection{Relation with the Taylor analysis}
\label{sec:taylor-link}
\begin{addmodblock}
We now relate the quadratic methods of \Cref{sec:quadmethods} to the local analysis of \Cref{sec:taylor}, that is, to the choice $\mathbf u_{\mathrm{ref}}=\mathbf u^{*}$, for which $\mathbf S=\mathbf S^{r}$. As in \Cref{sec:taylor}, we work throughout this subsection in the transformed variables of \Cref{sec:svd}, so that $(\cdot,\cdot)_X$ is the Euclidean inner product and $\Pi_{\mathbf V}=\mathbf V\mathbf V^{\top}$.

By \Cref{thm:quadratic-law-curvature}, identifying the generalized coordinates with the limit tangent coefficients,
$\boldsymbol{\alpha}\equiv\boldsymbol{\alpha}_{(1)}^*(\boldsymbol{\nu})\in\mathbb R^{q}$
(so that $n=q$ and $m=q(q+1)/2$), the curvature coefficients obey
\[
\boldsymbol{\alpha}_{(2)}^*(\boldsymbol{\nu})
=\mathbf R\,\mathbf T_{\mathbf A^*}^{-1}\,\mathbf\Psi_{\mathrm{hom}}(\boldsymbol{\alpha})
+\mathcal O(\|\boldsymbol{\alpha}\|^3),
\]
which is precisely the ansatz~\eqref{eq:nonlinearapprox} with
$\mathbf\Psi=\mathbf\Psi_{\mathrm{hom}}$. The quadratic-manifold correction matrix can therefore be identified locally as $\mathbf W=\mathbf U_{(2)}^*\,\mathbf R\,\mathbf T_{\mathbf A^*}^{-1}$.

In practice the basis is built from the data rather than from the limit, which corresponds to
$\mathbf V=\mathbf U_{(1)}^r$, hence to
$\boldsymbol{\alpha}=\boldsymbol{\alpha}_{(1)}^r(\boldsymbol\nu)$ and $\Pi_{\mathbf V}=\Pi_0^r=\mathbf U_{(1)}^r\mathbf U_{(1)}^{r\top}$, the empirical counterpart of the exact tangent projection $\Pi_0$. Problem~\eqref{eq:quad_opt} can be rewritten as the fit of $\mathbf W\,\mathbf\Psi_{\mathrm{hom}}(\mathbf V^\top\mathbf S)$ to
$(\mathbf I-\Pi_{\mathbf V})\mathbf S$: the nonlinear correction is therefore fitted to the
part of the snapshots that the linear basis does not represent. With
$\Pi_{\mathbf V}=\Pi_0^r$, this residual is $\mathbf S^r$ deprived of its first $q$ singular
components, that is, the part of the snapshots spanned by the left singular vectors of
$\mathbf S^r$ beyond the first $q$ --- among them the unfiltered block
$\mathbf U_{(2)}^r$ of~\eqref{eqU2r}, whose coefficients are analysed in
\Cref{rem:unfiltered-coefficients}. As shown there, a linear contribution may remain
alongside the curvature term.

The homogeneous ansatz cannot represent such a contribution:
$\mathbf\Psi_{\mathrm{hom}}$ contains only the products $\alpha_j\alpha_\ell$, so that the correction $\mathbf W\,\mathbf\Psi_{\mathrm{hom}}(\boldsymbol{\alpha})$ contains no term linear in $\boldsymbol{\alpha}$.
This is consistent with the exact tangent filtering of \Cref{sec:2ndorder-filtered-coeff}, which removes the first-order tangent contribution. For a basis $\mathbf V$ built from data, however, the corresponding projection $\Pi_{\mathbf V}$ need not coincide with the exact tangent projection $\Pi_0$ and therefore need not remove that contribution exactly. The homogeneous quadratic structure identified in the Taylor analysis is thus tied to the ideal tangent-filtered setting and is not guaranteed for a data-driven approximation.
\end{addmodblock}

\subsection{Centring on the mean}
\label{sec:problem_formulation}
\begin{minmodblock}
In \Cref{sec:taylor-link} the missing contribution was linear, and with
$\mathbf u_{\mathrm{ref}}=\mathbf u^{*}$ it has a clear geometric origin: it arises from the
discrepancy between the empirical basis and the exact tangent space. We now turn to the other usual choice of reference state, the mean $\bar{\mathbf u}$ of the snapshots.

Mean centring is the setting most favourable to the homogeneous ansatz. After centring, the SVD coordinates have zero empirical mean and the unresolved residual is both of zero mean and orthogonal to them. Thus the target residual has no direct empirical constant or linear component in the retained coordinates. The quadratic features inherit neither property. Their mean involves the second moments $\overline{\alpha_j\alpha_\ell}$, which do not vanish, and their correlation with the coordinates involves the third moments $\overline{\alpha_j\alpha_\ell\alpha_k}$, which need not vanish, although they do under a centrally symmetric sampling. A constant and a linear contribution may therefore remain in the optimal joint fit, even under mean centring.

The two-dimensional example that follows is the symmetric case, in which the third moments vanish and the constant alone is missing. With $n=1$ the map $\mathbf\Psi_{\mathrm{hom}}$ reduces to the single feature $\alpha_1^{2}$, and no combination of features can produce that constant. The snapshots are given by
\end{minmodblock}
\begin{equation}
\label{eq:testcasequadratic}
  \mathbf{s}_i = 
  \begin{bmatrix}
    w_1\,a\,\mu^i \\[1ex]
    w_2\,\bigl(b + c\,(\mu^i)^2\bigr)
  \end{bmatrix},
  \qquad i=1,\dots,M,
\end{equation}
where $\{\mu^i\}_{i=1}^M$ is a symmetric sampling of a centred parametric interval
$[-\mu_{\max},\mu_{\max}]$ (i.e., \(\mu^i = -\mu^{M+1-i}
\)), and the weights satisfy $w_1 > w_2 > 0$.
We recall that in both quadratic approaches, an SVD is first performed. 
It is clear that two modes are needed to capture the solutions in
\eqref{eq:testcasequadratic}. However, nothing guarantees that these modes will be 
\(\mathbf{e}_1=(1,0)^\top\) and \(\mathbf{e}_2=(0,1)^\top\), nor that the quadratic dependence in the second coordinate will be preserved after the SVD.

To enforce this behavior, we impose the following discrete orthonormality constraints:
\begin{equation}
  \label{eq:constraints}
  \frac{1}{M}\sum_{i=1}^M (a\,\mu^i)^2 = 1,
  \qquad
  \frac{1}{M}\sum_{i=1}^M \bigl(b + c\,(\mu^i)^2\bigr)^2 = 1,
  \qquad
  \frac{1}{M}\sum_{i=1}^M \mu^i\bigl(b + c\,(\mu^i)^2\bigr) = 0.
\end{equation}

\minmod{The last condition holds automatically, the sampling being symmetric and
$b+c\,(\mu^i)^2$ even in $\mu^i$. Choosing
$c=-b\big/\big(\tfrac1M\sum_i(\mu^i)^2\big)$ gives the second coordinate zero mean as
well, so that $\bar{\mathbf u}=\boldsymbol 0$: the centred choice
$\mathbf u_{\mathrm{ref}}=\bar{\mathbf u}$ coincides here with $\boldsymbol 0$ and
$\mathbf S=[\mathbf{s}_1,\dots,\mathbf{s}_M]$. The origin nevertheless does not lie on the curve,
since $b\neq0$ --- centring on the mean places the reference state off $\mathcal S$ even
when that mean vanishes. The parameters $a,b$ are then fixed by the two normalisation
conditions in~\eqref{eq:constraints}.}

Under these assumptions, and thanks to the ordering condition $w_1 > w_2 > 0$,
the first two left singular vectors of $\mathbf{S}$ are precisely $\mathbf{e}_1$ and $\mathbf{e}_2$. In this setting, we have $n=1$, $m=1$, and $\mathcal{N}=2$.

While the linear approximation using $2$ modes and the NCRBA with a quadratic
regression model can both recover the full nonlinear behavior, the quadratic methods based
solely on the homogeneous map \(\mathbf{\Psi}_{\text{hom}}\) cannot reproduce the constant term
$w_2\,b$ in the second coefficient. \minmod{This is the case anticipated above: a single coordinate, with nonzero second moment
and vanishing third moment by symmetry of the sampling.}

\subsection{Full quadratic ansatz}
\label{sec:enhanced_quad} 
\minmod{The remedy is to stop enforcing homogeneity: we replace $\mathbf{\Psi}_{\text{hom}}$
by the full quadratic map $\mathbf{\Psi}_{\text{full}}$ of~\eqref{eq:feature-maps}, which
contains it as a sub-block and carries in addition a constant and a linear block. For
$\boldsymbol{\alpha}\in\mathbb R^{n}$ the full map takes values in $\mathbb R^{m}$ with
$m=1+n+\tfrac{n(n+1)}{2}$.} This modification integrates directly into both quadratic methods introduced earlier:

\begin{itemize}
  \item \textbf{QSVDM:} the least-squares formulation in~\eqref{eq:quad_opt} remains the same, and one simply 
        replaces $\mathbf{\Psi}_{\text{hom}}$ by $\mathbf{\Psi}_{\text{full}}$.
  \item \textbf{QGM:} the greedy selection procedure and the final step of constructing $\mathbf{W}$ remain 
identical, except that all evaluations of the decoder use $\mathbf{\Psi}_{\text{full}}$ instead of $\mathbf{\Psi}_{\text{hom}}$.
\end{itemize}
With this, the approximation~\eqref{eq:nonlinearapprox} takes the form
\begin{equation}
  \label{eq:aug_quad_arapprox}
  \tilde{\mathbf{u}}
  = \minmod{\mathbf u_{\mathrm{ref}} +} \mathbf{V}\boldsymbol{\alpha} + \mathbf{W}\,\mathbf\Psi_{\mathrm{full}}(\boldsymbol{\alpha})
  = \minmod{\mathbf u_{\mathrm{ref}} +} \mathbf{V} \boldsymbol{\alpha}
  + \mathbf{C}
  + \mathbf{L} \boldsymbol{\alpha}
  + \mathbf{H}\,\mathbf\Psi_{\mathrm{hom}}(\boldsymbol{\alpha})
  \;\in\; \mathbb{R}^{\mathcal N},
\end{equation}
where 
\(
  \mathbf W = [\,\mathbf C,\; \mathbf L,\; \mathbf H\,]
  \in \mathbb{R}^{\mathcal N \times m},
  \,
  \mathbf C\in\mathbb R^{\mathcal N},\,
  \mathbf L\in\mathbb R^{\mathcal N\times n},\,
  \mathbf H\in\mathbb R^{\mathcal N\times \frac{n(n+1)}{2}}.
\)

\minmod{Regrouping,
$\widetilde{\mathbf u}=(\mathbf u_{\mathrm{ref}}+\mathbf C)+(\mathbf V+\mathbf L)\boldsymbol{\alpha}
+\mathbf H\,\mathbf\Psi_{\mathrm{hom}}(\boldsymbol{\alpha})$: the constant block corrects the reference
state and the linear block corrects the reduced basis. The full map therefore lets the data
amend the choice of origin and the choice of basis that the homogeneous map must take as
given. The homogeneous formulation is recovered by setting $\mathbf C=\boldsymbol 0$ and
$\mathbf L=\boldsymbol 0$, so that this constitutes a minimal relaxation of the homogeneity
assumption rather than a different class of nonlinear decoder.}

In \Cref{fig:toy_reconstruction}, we illustrate on the toy example~\eqref{eq:testcasequadratic}
that the full quadratic feature map allows both QSVDM and QGM to recover the constant
contribution. \minmod{The influence of this augmentation on the approximation error is problem- and
regime-dependent, and is examined numerically in \Cref{sec:enh_num}.}

\begin{figure}[H]
\centering
\begin{tikzpicture}
\begin{axis}[
    width=0.6\textwidth,
    xlabel={first component},
    ylabel={second component},
    legend style={
        at={(0.5,0.03)},
        anchor=south,
    }  ,
    grid=both
]
\addplot[only marks, black] table {data/data_points.dat};
\addlegendentry{data points};

\addplot[only marks, mark= x, blue] table {data/qsvdm.dat};
\addlegendentry{quadratic (homogeneous)};
\addplot[only marks, mark=triangle*, red] table {data/qsvdm_aug.dat};
\addlegendentry{quadratic (full)};
\end{axis}
\end{tikzpicture}
\caption{Comparison of quadratic approximation using homogeneous (in blue marks) and full quadratic (in red marks) feature maps on the toy example of \cref{sec:problem_formulation}. The QSVDM and QGM methods are equivalent for this example.}
\label{fig:toy_reconstruction}
\end{figure}

\section{ Numerical illustrations}
\label{sec:num}

This section presents numerical results that illustrate the properties and concepts developed in this work. \minmod{The experiments are conducted on two test cases: a parametrized elliptic problem, whose Kolmogorov width decays fast, and a transport-dominated problem, for which it does not.} 

\minmod{Three conventions are common to all the experiments. Errors are measured in the norm
induced by the matrix $\mathbf X$ of~\eqref{eq:X_inner}, fixed for each test case below. Test errors are averaged, or maximised, over a set $\mathcal P_{\mathrm{test}}$ of $10^{3}$ parameters drawn independently of those used to build the bases and to fit the surrogates. Unless stated otherwise the reference state of \Cref{sec:linear-rb} is the
solution $\mathbf u_{\mathrm{ref}}=\mathbf u^{*}$ at the reference parameter
$\bar{\boldsymbol\mu}$ of the test case considered; \Cref{sec:enh_num} is the exception, both
choices being compared there.}

\subsection{A multiparameter elliptic test case}
\label{sec:test-case}
We consider the problem of designing a thermal fin to remove heat from a surface effectively, see {\emph \eg} \cite{rbpp}.
The two-dimensional fin, shown in \Cref{fig:thermal-fin}, consists of a vertical central \textbf{post} and say $N_f$ horizontal \textbf{subfins}; the fin conducts heat from a prescribed uniform flux source at the root, $\Gamma_{\text{root}}$, through the subfins to surrounding flowing air. The thermal conductivity $k_i$, for $i = 1, \ldots, N_f$ of the $i$-th subfin (normalized with respect to the post conductivity $k_0 = 1$) may depend on a parameter, \eg $k_i = \mu_i\in [0.1, 10]$, for $i = 1, \ldots, p-1$,  while $k_i = 1$ for $i = p, \ldots, N_f$. The Biot number, which quantifies convective heat transfer to ambient air, is denoted as
$\text{Bi}$ (larger $\text{Bi}$ implies better convection) and may also vary $\text{Bi} = \mu_p\in [0.01,1]$.
The parameter $\boldsymbol{\mu}=(\mu_1, \dots, \mu_p)$ may take values in the compact parameter domain $\mathcal{P} = ([0.1, 10]^{p-1}\times [0.01,1]) \subset \mathbb{R}^p$.

The full PDE (strong) formulation and boundary conditions are given in~\cite{ballout2024nonlinear}. The natural functional space is here $X = H^1(\Omega)$, and we recall here only the corresponding variational formulation: for a given $\boldsymbol{\mu} \in \mathcal{P}$, find $u(\boldsymbol{\mu}) \in X$ such that:
\begin{equation}
  \label{eq:varpb}
  a(u(\boldsymbol{\mu}), v; \boldsymbol{\mu}) = f(v), \quad \forall v \in X,
\end{equation}
where the bilinear and linear forms are given by:
\[
a(u, v; \boldsymbol{\mu}) = \sum_{i=0}^{N_f} k_i \int_{\Omega_i} \nabla u \cdot \nabla v \, d\boldsymbol{x} 
+ \text{Bi} \int_{\Gamma_{\text{ext}}} u v \, d\boldsymbol{x}, \quad
f(v) = \int_{\Gamma_{\text{root}}} v\, d\boldsymbol{x}.
\]

The domains $\Omega_i$ and boundaries $\Gamma_{\text{root}}$ and $\Gamma_{\text{ext}}$ are shown in \Cref{fig:thermal-fin}, which illustrates the geometry and labeling used in this setup.

\begin{figure}[H]
\centering
\scalebox{0.8}{
\begin{tikzpicture}[scale=1.2, every node/.style={font=\footnotesize}]

\foreach \i/\y in {1/0.5, 2/1.5, 3/2.5, 4/3.5} {
  \draw[very thick, blue] (-2.5,\y) -- (-0.3,\y);
  \draw[very thick, blue] (-2.5,\y) -- (-2.5,\y+0.5);
  \draw[very thick, blue] (-2.5,\y+0.5) -- (-0.3,\y+0.5);
  \draw[very thick, blue] (0.3,\y) -- (2.5,\y);
  \draw[very thick, blue] (2.5,\y) -- (2.5,\y+0.5); 
  \draw[very thick, blue] (0.3,\y+0.5) -- (2.5,\y+0.5);
  \draw[very thick, blue] (-0.3,\y-0.5) -- (-0.3,\y);  
  \draw[very thick, blue] (0.3,\y-0.5) -- (0.3,\y);   
}

\draw[very thick, blue] (-2.5,4) -- (2.5,4);  

\node at (-1.25,0.75) {\(\Omega_1\)};
\node at (-1.25,1.75) {\(\Omega_2\)};
\node at (-1.25,2.75) {\(\cdots\)};
\node at (-1.25,3.75) {\(\Omega_{N_f}\)};

\node at (1.25,0.75) {\(\Omega_1\)};
\node at (1.25,1.75) {\(\Omega_2\)};
\node at (1.25,2.75) {\(\cdots\)};
\node at (1.25,3.75) {\(\Omega_{N_f}\)};

\draw[draw=black] (-0.3,-0.5) rectangle (0.3,4);
\node at (0.,2.5) {\(\Omega_0\)};
\draw[very thick, red] (-0.3,-0.5) -- (0.3,-0.5);
\node[below, red] at (0,-0.5) {\(\Gamma_{\text{root}}\)};

\node[above, blue] at (0,4.05) {\(\Gamma_{\text{ext}}\)};
\draw[very thick, blue] (-0.3,-0.5) -- (-0.3,0);   
\draw[very thick, blue] (0.3,-0.5) -- (0.3,0);   
\end{tikzpicture}
}
\caption{Schematic of thermal fin geometry with labeled subdomains \(\Omega_i\) and boundaries \(\Gamma_{\text{root}}\), \(\Gamma_{\text{ext}}\).}
\label{fig:thermal-fin}
\end{figure}

A finite element (FE) approximation is then applied as a high-fidelity method, where the solution $u(\boldsymbol{\mu}) \in X$  is approximated by $u_{\mathcal{N}}(\boldsymbol{\mu}) \in X_{\mathcal{N}} \subset X$. The solution manifold contained in $X$ (which is not directly computable) is represented by an approximation $\mathcal S = \{ u_{\mathcal{N}}(\boldsymbol{\mu}),\, \boldsymbol{\mu} \in \mathcal{P} \} \subset X_{\mathcal{N}}$.
To be consistent with \Cref{sec:linear-rb,sec:taylor}, $ u_{\mathcal{N}}(\boldsymbol{\mu}) $ is represented by its coordinate (degrees of freedom) vector $ \mathbf{u}_{\mathcal{N}}(\boldsymbol{\mu}) $, allowing $ \mathcal S $ to be interpreted as a manifold in $X= \mathbb{R}^{\mathcal{N}}$. 

\minmod{For this test case, we fix $N_f=3$, and the reference parameter \(
  \bar{\boldsymbol\mu}=(1,\dots,1,0.1)\in\mathcal P ,\) whose conductivities are all equal to one and whose Biot number is $\mathrm{Bi}=0.1$, is used
throughout this section for two purposes. First, the matrix $\mathbf{X}$ defining the
$X$-inner product in \eqref{eq:X_inner} is taken to be the corresponding energy matrix,
\[
  \mathbf{X}:=\big(a(\zeta_i,\zeta_j;\bar{\boldsymbol\mu})\big)_{i,j=1}^{\mathcal N},
\]
where $\{\zeta_i\}_{i=1}^{\mathcal N}$ are the FE basis functions. Second, the associated solution
\(
  \mathbf u^{*}:=\mathbf u_{\mathcal N}(\bar{\boldsymbol\mu})\in\mathcal S
\)
is the point at which the manifold is centred and about which the expansions of \Cref{sec:taylor} are performed.}

\begin{refmodblock}
\subsection{A transport-dominated test case}
\label{sec:test-case-advection}
The test case of \Cref{sec:test-case} is smooth and its Kolmogorov width decays fast. To probe
the opposite regime we use the advecting pulse, a standard benchmark for which linear reduced bases are known to be inefficient. Consider the one-dimensional transport equation
\[
\partial_t u(x,t)+c\,\partial_x u(x,t)=0,
\qquad x\in\mathbb{R},\quad t>0,
\]
with constant velocity $c=10$ and Gaussian initial condition
\[
u_0(x)=\frac{1}{\sqrt{\pi w_0}}
\exp\!\left(-\frac{(x-x_0)^2}{w_0}\right),
\qquad x_0=0.1,\quad w_0=2\cdot 10^{-4},
\]
whose exact solution $u(x,t)=u_0(x-ct)$ is a pulse of fixed shape travelling at speed $c$. We
restrict the computation to $\Omega=[0,1]$ and evaluate the exact solution at
$\mathcal{N}=4096$ equispaced points $x_1,\dots,x_{\mathcal N}$ of $\Omega$, collected in the
vector $\mathbf{u}_{\mathcal N}(\boldsymbol\mu)\in\mathbb{R}^{\mathcal N}$; the $X$-inner product of~\eqref{eq:X_inner} is here the Euclidean one, $\mathbf{X}=\mathbf I$.

Two parameter choices are used. For $p=1$ the only parameter is time,
$\boldsymbol\mu=t\in\mathcal{P}=[0,T]$ with $T=0.08$. For $p=2$ the width of the pulse is
added, $\boldsymbol\mu=(t,w)\in\mathcal{P}=[0,T]\times[\tfrac12 w_0,\,2w_0]$, the snapshot
being
\begin{equation}
\label{eq:adv-snapshot}
\bigl(\mathbf{u}_{\mathcal N}(t,w)\bigr)_i
=\frac{1}{\sqrt{\pi w_0}}
\exp\!\left(-\frac{\bigl(x_i-x_0-ct\bigr)^2}{w}\right),
\qquad i=1,\dots,\mathcal{N},
\end{equation}
where the normalising constant is kept at its reference value $w_0$, so that $w$ changes the
width of the pulse at fixed amplitude; the width is drawn log-uniformly in $\mathcal{P}$.

The pulse travels a distance $cT=0.8$ while its own standard deviation is
$\sigma_0=\sqrt{w_0/2}=10^{-2}$. Solutions at well-separated times are therefore nearly
orthogonal, and of the order of $cT/\sigma_0\approx 80$ modes are required to represent the
family: in contrast with \Cref{sec:test-case}, the Kolmogorov width of $\mathcal S$ decays
slowly.

As in the elliptic case, a reference parameter $\bar{\boldsymbol\mu}$ is fixed, namely the
mid-time $\bar t=T/2=0.04$ together with $w=w_0$ when $p=2$, so that
$\mathbf u^{*}=\mathbf u_{\mathcal N}(\bar{\boldsymbol\mu})$ is the pulse centred at $x=0.5$.
\end{refmodblock}

\subsection{GSS}
\label{GSS2}
We compare the three basis constructions of \Cref{sec:linear-rb} --- SVD, greedy and
GSS --- on two test cases: the thermal-fin model of \Cref{sec:test-case} with $p=4$, whose
Kolmogorov width decays fast, and the advecting wave of \Cref{sec:test-case-advection} with
$p=1$, whose width decays slowly. The two regimes stress the constructions differently.

The constructions are not all priced alike. The SVD over a training set of $M$ parameters and
GSS selecting $M$ parameters out of a training set $\mathcal P_g$ of size $M_g$ both cost $M$
high-fidelity solves and both return the whole family of bases of dimension $N<M$; the greedy
search over $\mathcal P_g$ is driven by the error estimator and costs no solve of its own. The
meaningful pairing is therefore SVD against GSS at equal $M$, and we report in addition the SVD
over $M=1000$ snapshots as a reference of what a far more expensive sampling achieves.

For each reduced dimension $N$ we compute the relative projection error
\begin{equation}
\label{eq:proj-err}
\varepsilon_N(\boldsymbol{\mu})
:=\frac{\big\|\mathbf v(\boldsymbol{\mu})-P_{\mathcal V_N}\mathbf v(\boldsymbol{\mu})\big\|_X}
{\big\|\mathbf u_{\mathcal N}(\boldsymbol{\mu})\big\|_X}\qquad \boldsymbol{\mu}\in\mathcal P_{\mathrm{test}},
\end{equation}
and report its mean and its maximum over $\mathcal P_{\mathrm{test}}$. The training sets, are drawn from the uniform distribution on $\mathcal P$, and each experiment is
repeated over $20$ seeds: the curves report the geometric mean over the seeds and the bands one
standard deviation of $\log_{10}\varepsilon_N$ about it, a multiplicative band of factor $10^{\pm\sigma}$.

\begin{figure}[htbp]
\centering
\includegraphics[width=\textwidth]{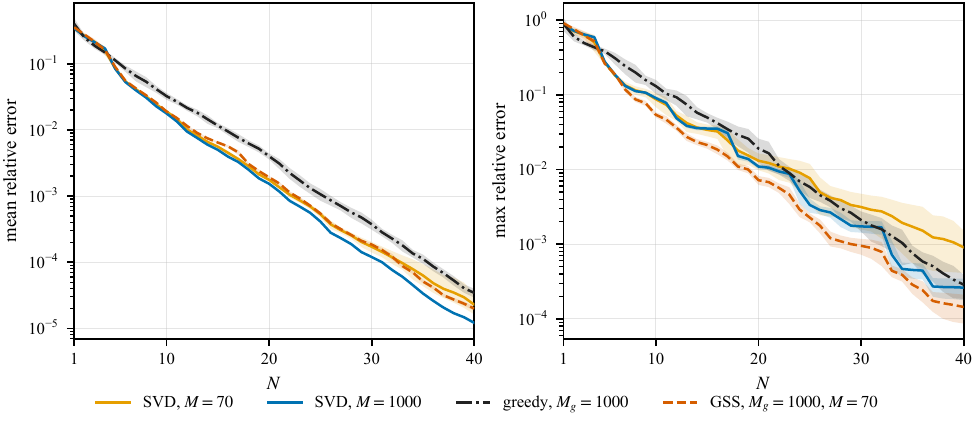}
\caption{Mean (left) and maximum (right) relative projection error~\eqref{eq:proj-err}
against the reduced dimension $N$ for the elliptic test case of \Cref{sec:test-case}, $p=4$.}
\label{fig:gss_fin}
\end{figure}

\begin{figure}[htbp]
\centering
\includegraphics[width=\textwidth]{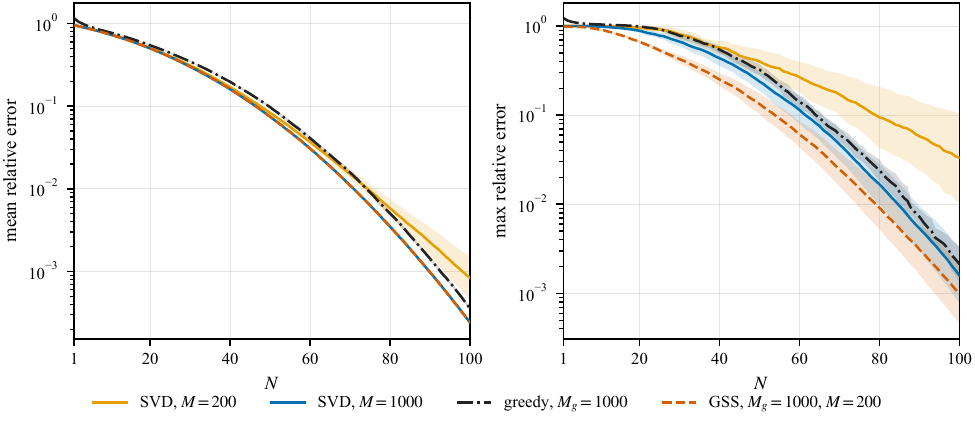}
\caption{Mean (left) and maximum (right) relative projection error~\eqref{eq:proj-err}
against the reduced dimension $N$ for the transport test case of \Cref{sec:test-case-advection},
$p=1$.}
\label{fig:gss_transport}
\end{figure}

On the thermal fin (\Cref{fig:gss_fin}), the picture separates by metric. In the maximum error
GSS is the best of the four, the reference SVD over $M=1000$ snapshots included. In the
mean error it tracks the SVD at the same $M$, both of them above that reference; what it
improves on there is the greedy, which it outperforms over the whole range of $N$.

\refmod{On the advecting wave (\Cref{fig:gss_transport}), GSS dominates both of the constructions it is
built from: it is below the SVD at the same $M$ and below the greedy in both metrics, and the
gap widens with $N$. In the maximum error it also falls below the SVD over $M=1000$ snapshots,
that is, below a sampling five times as expensive.}

This split is what the construction leads one to expect. The greedy selects parameters by
maximizing an error indicator and is therefore a worst-case procedure, whereas the SVD is
optimal in the mean over its training set, by~\eqref{eq:svd_opt}. GSS performs the SVD on the
set $\mathcal P_M$ that the greedy selects, so it inherits the worst-case selection and
restores the mean-error optimality that the greedy alone lacks: it wins on the maximum error in
both test cases, and on the mean error it recovers the behaviour of an SVD using the same
number of solves. The gain over the SVD at equal $M$ measures what the greedy selection buys
over uniform sampling, and it is largest where the width decays slowly and a uniform draw is
least likely to capture the parameters that matter.

\subsection{Convergence towards the Taylor expansion}
\label{sec:conv.taylor}
Below, we illustrate numerically, on the elliptic test case of \Cref{sec:test-case}, the theory developed in \Cref{sec:taylor}. The hypotheses of \Cref{sec:assumptions} are satisfied for this test case, with $q=p$, so we write $p$ throughout this section.


\subsubsection{Differentiation with respect to the parameter}
\label{sec:param.diff}
To compute the Taylor approximation of a solution near a given parameter, or
to represent the tangent plane at a given point, it is essential to
differentiate the solution map $u: \boldsymbol{\mu}=(\mu_1, \dots, \mu_p)
\mapsto u(\boldsymbol{\mu})$. The bilinear form $a(\cdot, \cdot;
\boldsymbol{\mu})$ in~\eqref{eq:varpb} is assumed to admit the affine
decomposition

\begin{equation}
\label{eq:affinedec}
a(u, v; \boldsymbol{\mu}) = \sum_{q=1}^{Q} \theta_q(\boldsymbol{\mu}) a_q(u, v).
\end{equation}
As long as $a(\cdot,\cdot;\boldsymbol{\mu})$ remains coercive, the regularity
of $u$ in $\boldsymbol{\mu}$ is that of the coefficients $\theta_q$\cite[Prop.~5.2]{bookRb}; for the test case considered here the $\theta_q$ are affine and the coercivity
constant is bounded below uniformly on the parameter set, so that $u$ is
$C^{\infty}$ in $\boldsymbol{\mu}$: the
regularity required in \Cref{sec:taylor} is available at every order. Differentiating~\eqref{eq:varpb}, and using that $a(\cdot, \cdot;
\boldsymbol{\mu})$ is bilinear and that $f$ is independent of
$\boldsymbol{\mu}$, the partial derivative \( \frac{\partial u}{\partial
\mu_i}(\boldsymbol{\mu}) \) satisfies the variational problem
\begin{equation}
\label{eq:partial_derivatives}
a\Bigl( \frac{\partial u}{\partial \mu_i}(\boldsymbol{\mu}), v;
\boldsymbol{\mu}\Bigr) = - \sum_{q=1}^{Q} \frac{\partial \theta_q}{\partial
\mu_i} (\boldsymbol{\mu})\, a_q(u(\boldsymbol{\mu}), v), \quad \forall v \in X.
\end{equation}
Higher-order derivatives are obtained by differentiating
\eqref{eq:partial_derivatives} recursively. Writing
$\partial_i u := \frac{\partial u}{\partial \mu_i}(\boldsymbol{\mu})$ and
$\partial_i\theta_q := \frac{\partial\theta_q}{\partial\mu_i}
(\boldsymbol{\mu})$, the second derivatives solve
\begin{equation}
\label{eq:second_derivatives}
a\bigl(\partial^2_{ij}u, v; \boldsymbol{\mu}\bigr)
= -\sum_{q=1}^{Q}\Bigl[\partial_i\theta_q\, a_q(\partial_j u, v)
 + \partial_j\theta_q\, a_q(\partial_i u, v)
 + \partial^2_{ij}\theta_q\, a_q(u, v)\Bigr],
\quad \forall v \in X ,
\end{equation}
the last term vanishing here since the $\theta_q$ are affine.
Note that the left-hand side operator is the same in \eqref{eq:varpb}, \eqref{eq:partial_derivatives} and \eqref{eq:second_derivatives}: a single factorization at $\boldsymbol{\mu}$ is reused for the solution and for all its first- and second-order derivatives, only the right-hand sides changing. Similarly to what was done for the variational problem~\eqref{eq:varpb}, an FE approximation of the derivatives in $X_{\mathcal{N}}$ can be obtained using a Galerkin approach.

\subsubsection{Convergence of the SVD modes}
\label{sec:conv.svd}
The convergence results of \Cref{sec:taylor} are all stated up to an orthogonal
transformation, as in \cref{eq:basis_convergence,eq:U2hat-align}. We
therefore report, in this subsection, a per-mode error after alignment: writing
$\mathbf{U}^r=[\mathbf{u}_1^r,\mathbf{u}_2^r,\dots]$ for the block of computed
modes, $\mathbf{U}^*$ for its limit and $\mathbf{O}^r$ for the corresponding
orthogonal alignment matrix, we plot
\begin{equation}
  e_j(r)=\bigl\|\mathbf{u}_j^r-(\mathbf{U}^*\mathbf{O}^r)_j\bigr\|_X ,
  \label{eq:aligned-mode-error}
\end{equation}
where $j$ indexes the modes of the block considered. The columns compared being $X$-orthonormal, $e_j\in[0,2]$, with $e_j=\sqrt2$ when
the computed mode is $X$-orthogonal to its target. These per-mode errors also control the error on the
subspaces they span (\Cref{prop:basis-alignment}). All the experiments below use $M=30$ snapshots and report $e_j(r)$ against the sampling radius $r$.

\Cref{fig:svd_convergence} illustrates \Cref{thm:tangent_convergence}: the first
$p$ SVD modes $\mathbf{U}_{(1)}^r$ of the centered snapshot matrix
$\mathbf{S}^r$ converge, as $r\to0$, to the SVD modes $\mathbf{U}_{(1)}^*$ of
the matrix $\mathbf{J}_U(\boldsymbol{0})\cdot\mathbf{P}$, which span the tangent
space $T_0^*$, and the observed rate is linear in $r$.

\begin{figure}[htbp]
\centering
\includegraphics[width=\textwidth]{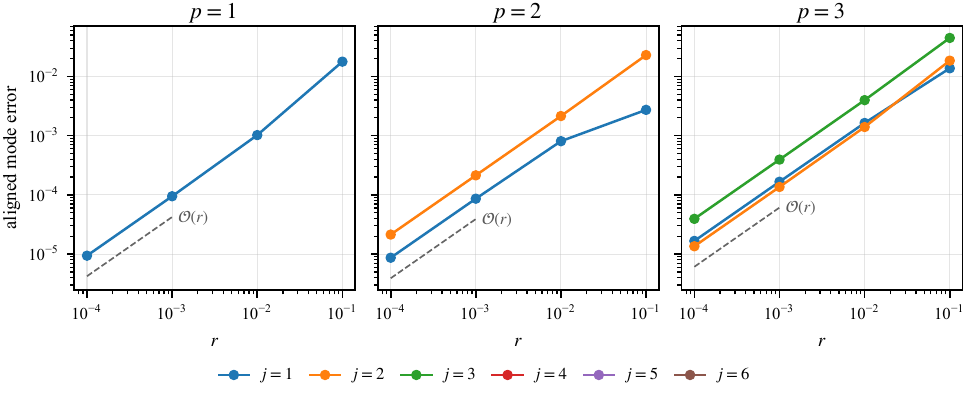}
\caption{Aligned mode error $e_j(r)$ of \cref{eq:aligned-mode-error} for the
first block of $\mathbf{S}^r$, $j=1,\dots,p$.}
\label{fig:svd_convergence}
\end{figure}

\Cref{fig:svd_convergence_order2_deflated} and
\Cref{fig:svd_convergence_order2} concern the second block of
$m=\tfrac{p(p+1)}{2}$ modes, which converges towards the
curvature subspace $\mathcal C$ of \Cref{sec:taylor.order2}; in both, the modes
of that block are indexed $j=1,\dots,m$.

\Cref{fig:svd_convergence_order2_deflated} illustrates \Cref{thm:DK_curvature}: once the projection $(\mathbf{I}-\Pi_0)$ has removed the tangent contribution, the first $m$ modes $\widehat{\mathbf{U}}_{(2)}^r$ of the projected
matrix $\widehat{\mathbf{S}}^{r}$ converge to the modes $\mathbf{U}_{(2)}^*$
spanning $\mathcal C$ at the linear rate $\mathcal{O}(r)$.
\Cref{fig:svd_convergence_order2} shows that the same rate is observed for the
second block $\mathbf{U}_{(2)}^r$ of the unfiltered matrix $\mathbf{S}^r$, as
anticipated in \Cref{rem:unfiltered-curvature}, even though a rigorous proof is
not yet established.

Finally, we note that for $r^{4}\le\epsilon$, with $\epsilon$ approximately $10^{-16}$ in
double precision, roundoff errors are expected to influence the numerical results and to
limit the observed convergence rate; the second-order experiments are therefore restricted to
$r\ge 10^{-4}$. The first-order ones in \Cref{fig:svd_convergence} are governed by the weaker condition $r^{2}\le\epsilon$
and could be carried to smaller radii, the same range being used throughout for
comparability.

\begin{figure}[htbp]
\centering
\includegraphics[width=\textwidth]{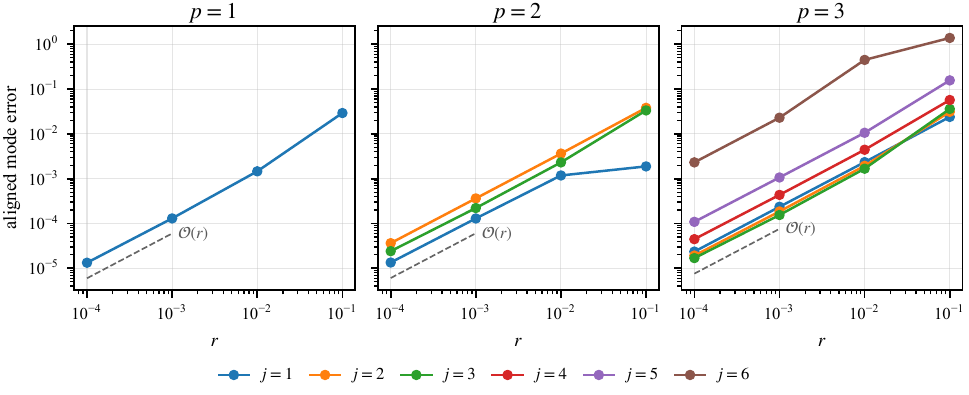}
\caption{Aligned mode error for the curvature block of the filtered matrix $\widehat{\mathbf{S}}^{r}$, $j=1,\dots,m$.}
\label{fig:svd_convergence_order2_deflated}
\end{figure}

\begin{figure}[htbp]
\centering
\includegraphics[width=\textwidth]{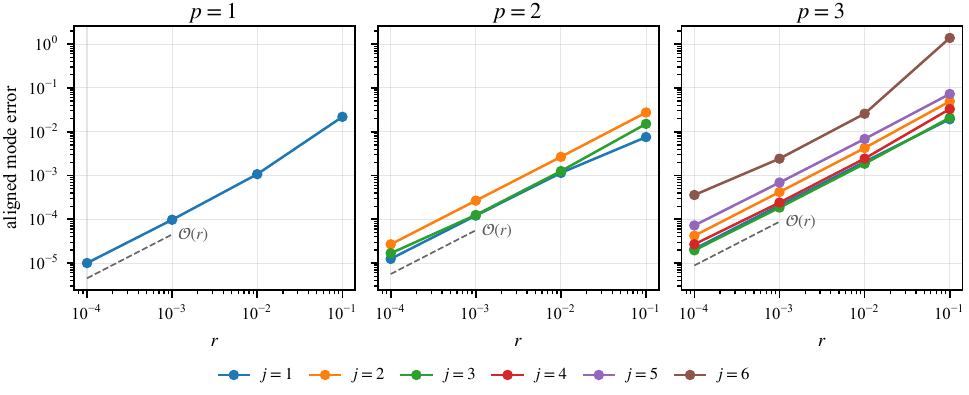}
\caption{Aligned mode error for the curvature block of
the unfiltered matrix $\mathbf{S}^r$, $j=1,\dots,m$.}
\label{fig:svd_convergence_order2}
\end{figure}

\begin{refmodblock}
\subsubsection{Behaviour of the reduced coefficients}
\label{sec:conv.quadlaw}
We now illustrate the behaviour of the coefficients carried by the two blocks,
analysed in \Cref{sec:taylor}. For a pair of tangent and curvature bases, let
$\boldsymbol{\alpha}_1$ and $\boldsymbol{\alpha}_2$ denote the coordinates of a
snapshot on each block, and let $\Psi$ be a feature map. For each $\Psi$ and
each radius $r$ we fit the coefficient matrix $\mathbf{C}$ by least squares over
the $M$ sampled directions $\boldsymbol\nu^k$, and report the relative
residual
\begin{equation}
  \rho_\Psi(r)=
  \frac{\displaystyle\min_{\mathbf{C}}\Bigl(\sum_{k=1}^{M}
    \bigl\|\boldsymbol{\alpha}_2^{k}
      -\mathbf{C}\,\Psi(\boldsymbol{\alpha}_1^{k})\bigr\|_2^2\Bigr)^{1/2}}
       {\displaystyle\Bigl(\sum_{k=1}^{M}
    \bigl\|\boldsymbol{\alpha}_2^{k}\bigr\|_2^2\Bigr)^{1/2}} .
  \label{eq:coeff-residual}
\end{equation}
Two pairs of bases are compared --- the limit bases
$\mathbf{U}_{(1)}^*,\mathbf{U}_{(2)}^*$ and the empirical ones obtained from the SVD at radius $r$ --- and two quadratic feature maps, the homogeneous $\Psi_{\hom}$ and
the augmented $\Psi_{\mathrm{full}}$ of \Cref{sec:quadratic}.

\Cref{fig:quadratic_law_deflated} illustrates
\Cref{thm:filtered-quadratic-law}: for the filtered matrix
$\widehat{\mathbf{S}}^{r}$, both feature maps give $\mathcal O(r)$. \Cref{fig:quadratic_law_raw} reports the same quantities for the unfiltered
matrix $\mathbf{S}^r$. With $\Psi_{\mathrm{full}}$ (right) the residual is again
$\mathcal O(r)$ for both limit and empirical curves. With
$\Psi_{\hom}$ (left) the limit curves decay at the rate predicted by
\Cref{thm:quadratic-law-curvature}, whereas the empirical residual plateaus. This is the tangent contamination of \Cref{rem:unfiltered-coefficients}: the surviving
first-order term is linear in $\boldsymbol{\alpha}_1$, which $\Psi_{\hom}$ cannot represent and $\Psi_{\mathrm{full}}$ absorbs exactly. The $p=1$ curves flatten at the smallest radius, where the absolute fitting residual has reached the floor of double-precision evaluation, about $4\times10^{-15}$; the curvature coefficients scaling like $r^{2}$, a constant absolute floor divided by an
$\mathcal O(r^{2})$ denominator makes the reported relative residual stagnate. Of the three
values of $p$ only $p=1$ attains that floor within the range plotted, its $\mathcal O(r^{3})$
remainder being the smallest, and the corresponding point is excluded when the rate is
estimated.

\begin{figure}[htbp]
\centering
\includegraphics[width=\textwidth]{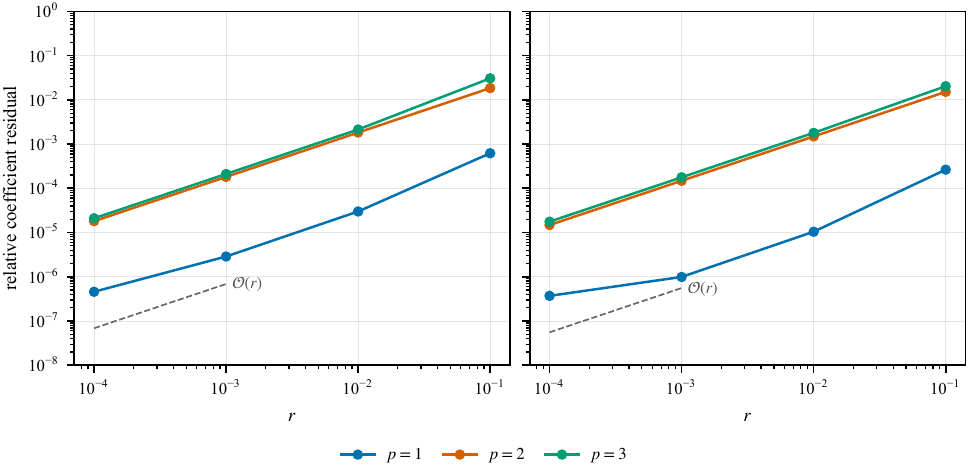}
\caption{Relative coefficient residual of \cref{eq:coeff-residual} for the
filtered matrix $\widehat{\mathbf{S}}^{r}$, with $\Psi_{\hom}$
(left) and $\Psi_{\mathrm{full}}$ (right); solid lines: empirical bases. The limit curves are not repeated: since
$\mathbf{U}_{(2)}^*\perp T_0^*$, deflation leaves them unchanged and they coincide with those
of \Cref{fig:quadratic_law_raw}.}
\label{fig:quadratic_law_deflated}
\end{figure}
\end{refmodblock}

\begin{figure}[htbp]
\centering
\includegraphics[width=\textwidth]{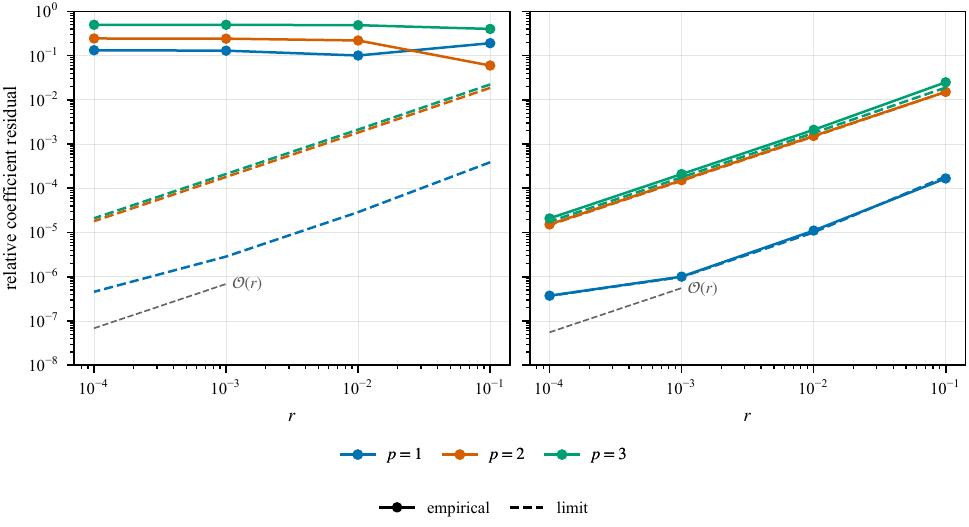}
\caption{Relative coefficient residual of
\cref{eq:coeff-residual} for the unfiltered matrix $\mathbf{S}^r$, with $\Psi_{\hom}$ (left) and $\Psi_{\mathrm{full}}$ (right). Solid lines:
empirical bases; dashed lines: limit bases.}
\label{fig:quadratic_law_raw}
\end{figure}

\subsection{Recovering the high RB modes from the low ones}
\label{sec:recovering}

Thanks to the bijection established (locally) in \Cref{sec:bijection}, the parameter $\boldsymbol{\mu}$ and therefore the higher coefficients $\{\alpha_{k}(\boldsymbol{\mu})\}_{k=n+1}^{N}$ can be inferred directly from the first $n$ coefficients. Accordingly, for $k = n+1, \dots, N$, the following relation holds:
\[
\alpha_k(\boldsymbol{\mu}) = \psi^k(\alpha_1(\boldsymbol{\mu}), \dots, \alpha_n(\boldsymbol{\mu})),
\]
where each $\psi^k$ denotes a nonlinear function.

Another fundamental step is the approximation of the functions $\psi^k$, which constitutes a regression task. The goal is to learn a nonlinear map
$\widehat{\psi}: \mathbb{R}^n \rightarrow \mathbb{R}^{N-n}$ such that:
\begin{equation}
\label{eq:nonlinear-function}
  \widehat{\psi}(\alpha_1, \dots, \alpha_n) \approx (\alpha_{n+1}, \dots, \alpha_N)
  = (\psi^{n+1}(\alpha_1, \dots, \alpha_n), \dots,\psi^{N}(\alpha_1, \dots, \alpha_n))
\end{equation}

\begin{refmodblock}
For the elliptic test case of \Cref{sec:test-case} we consider $p\in\{1,2,3,4\}$ and work in the centred setting of \Cref{sec:taylor}, with $\mathbf u_{\mathrm{ref}}=\mathbf u^{*}$, letting $n$ vary around $p$, so as to locate the dimension at which the recovery becomes possible. The surrogate is an artificial neural network (ANN) with four hidden layers of widths $32$, $64$, $128$ and $256$, one network being trained for each pair $(p,n)$ on $2\times 10^{3}$ parameters drawn from $\mathcal P$, independently of the high-fidelity dimension $\mathcal N$; its targets are the coefficients of the $N$-mode linear approximation, with $N=6,15,23,31$ for $p=1,2,3,4$ and $N$ held fixed within each panel. \refmod{The same experiment is run on the transport case of \Cref{sec:test-case-advection}, for $p\in\{1,2\}$, with the same architecture and the same training and test set sizes, and with $N=100$, the slow decay of the Kolmogorov width there requiring substantially more modes to represent the family.}

Of the $N$ coefficients $\alpha_1(\boldsymbol\mu),\dots,\alpha_N(\boldsymbol\mu)$
of~\eqref{eq:linearapprox}, the first $n$ are solved for and the remaining $N-n$ are predicted; we report
\begin{equation}
\label{eq:rec-err}
E_{\mathrm{rec}}(n)=\frac{1}{|\mathcal P_{\mathrm{test}}|}
\sum_{\boldsymbol\mu\in\mathcal P_{\mathrm{test}}}
\frac{\bigl\|\widehat\psi(\alpha_1,\dots,\alpha_n)
-(\alpha_{n+1},\dots,\alpha_N)\bigr\|_2}
{\bigl\|(\alpha_1,\dots,\alpha_N)\bigr\|_2}.
\end{equation}
Since the modes are $X$-orthonormal, \eqref{eq:rec-err} is the $X$-error of the
reconstructed state relative to its rank-$N$ projection. It isolates the surrogate, the
solved coordinates being the exact projection coordinates and the truncation beyond $N$
excluded.

\begin{figure}[H]
    \centering
    \includegraphics[width=\textwidth]{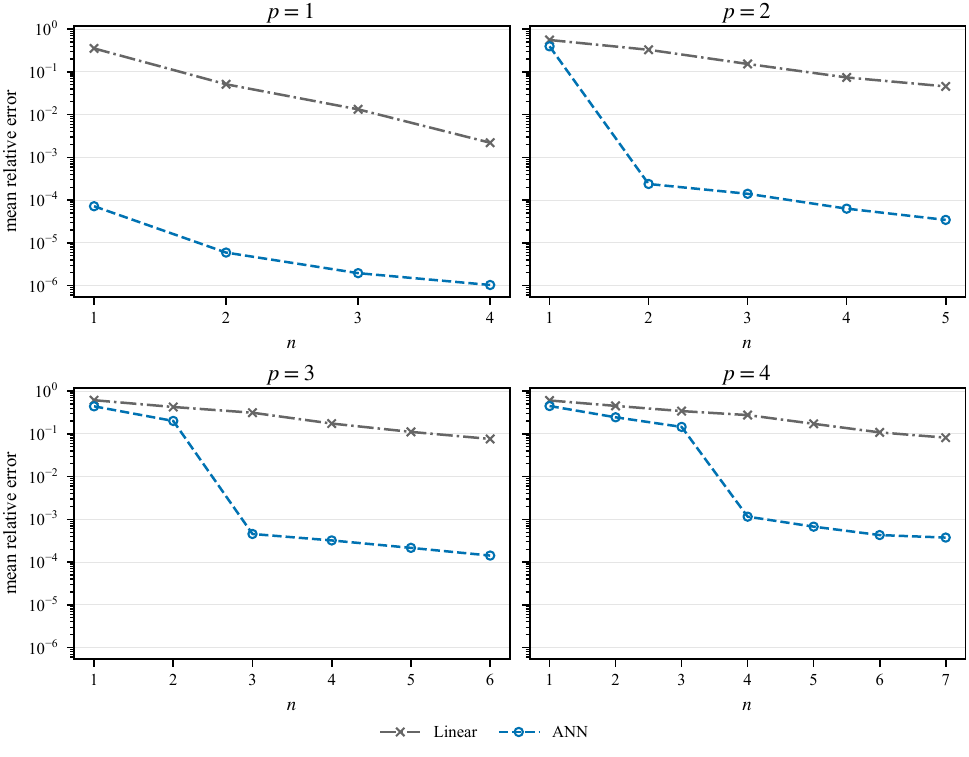}
    \caption{Reconstruction error~\eqref{eq:rec-err} against the solved dimension $n$
    for the elliptic test case of \Cref{sec:test-case}. ``Linear'' is the choice $\widehat\psi\equiv\boldsymbol 0$.}
    \label{fig:nlcrb-sweep}
\end{figure}

\Cref{fig:nlcrb-sweep} shows a sharp transition at $n=p$. For $n<p$ the surrogate brings no real gain over the linear approximation: the two curves decrease in parallel, within a small factor of each other, over the whole range. The step that brings $n$ to $p$ then lowers the error by more than two orders of magnitude in every case, and the coordinates beyond $p$ together contribute far less than that single one. For $n<p$ the map $\widehat\psi$ sought in~\eqref{eq:nonlinear-function} is not a function of its arguments, since distinct parameters share the same first $n$ coefficients, and no regressor can do better than the conditional spread of the targets; at $n=p$ the coefficient map becomes invertible and the higher coefficients become a graph over the lower ones. This is what \Cref{sec:bijection} predicts \emph{locally}, and what is observed here globally; the latter does not follow from the former.

The accuracy reached at $n=p$ nevertheless degrades as $p$ grows. This reflects the difficulty of the regression in higher dimension rather than a property of the reduced dimension: the transition itself is attributable to $n$, the architecture and the training set being the same throughout, but the residual error at $n=p$ is not, a fixed training set becoming increasingly sparse in $\mathcal P$ as $p$ grows and a comparable accuracy therefore requiring a larger one.

\begin{figure}[H]
    \centering
    \includegraphics[width=\textwidth]{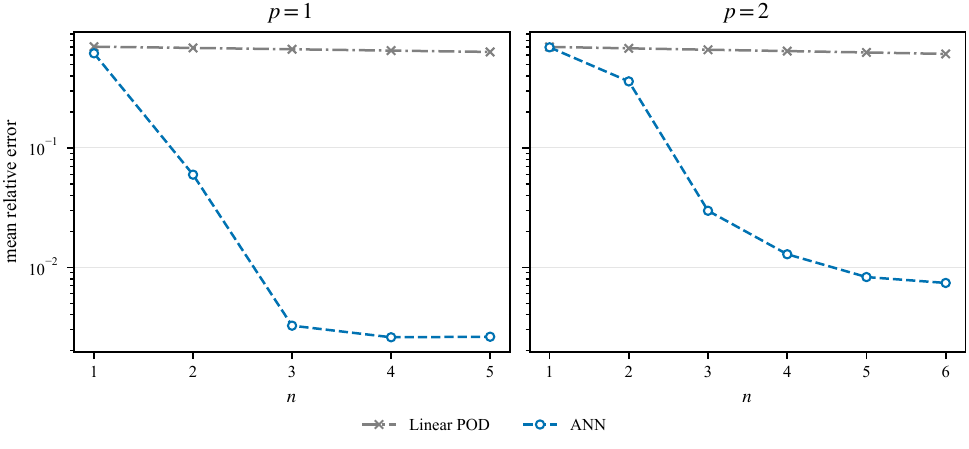}
    \caption{Reconstruction error~\eqref{eq:rec-err} against the solved dimension $n$
    for the transport test case of \Cref{sec:test-case-advection}, with $N=100$. ``Linear'' is the choice $\widehat\psi\equiv\boldsymbol 0$.}
    \label{fig:nlcrb-sweep-advection}
\end{figure}

\Cref{fig:nlcrb-sweep-advection} reports the same experiment for the transport case, and the
picture differs in two respects. First, the linear curve is near flat over the whole range, consistent with the slow decay of the Kolmogorov width. Second, the transition is neither located at $n=p$ nor confined to a single coordinate: the error
falls over several of them and saturates at $n=3$ for $p=1$ and at $n=4$ for $p=2$. The first
$p$ coefficients do not separate $\mathcal S$ globally here, each of them being the
correlation of the travelling pulse with a fixed mode, hence a non-monotone function of its
position, so that distinct parameters share the same leading coordinates; the local statement
of \Cref{sec:bijection} does not globalise. Both saturation points lie in the range
$p\le n\le 2p+1$ of \Cref{prop:global-sensing}. The agreement is an observation and not a consequence, the proposition bounding an infimum over all linear
encoders whereas the coordinates used here are the SVD ones, for which no such bound is
available.
\end{refmodblock}

\subsection{Enhanced approximation with full quadratic ansatz}
\label{sec:enh_num}

\begin{addmodblock}
We apply the quadratic approximation methods of \Cref{sec:quadmethods} and the full quadratic feature map of \Cref{sec:enhanced_quad} to the test case of \Cref{sec:test-case}, with $p=2$.
The bases and the quadratic maps are built from $M=2000$ training parameters.

We compare the two choices of $\mathbf u_{\mathrm{ref}}$ introduced in
\Cref{sec:linear-rb}: the empirical mean $\bar{\mathbf u}$, and the solution $\mathbf u^{*}$. The reference state is the same $\mathbf u^{*}$ about which the local expansions of \Cref{sec:conv.taylor} are performed, so that the left panel below is the
global counterpart of the stagnation observed there as $r\to0$. For each reduced dimension
$n$ we report the mean relative error over the test set,
\begin{equation}
\label{eq:enh-err}
E(n)=\frac{1}{|\mathcal P_{\mathrm{test}}|}
\sum_{\boldsymbol\mu\in\mathcal P_{\mathrm{test}}}
\frac{\big\|\mathbf u_{\mathcal N}(\boldsymbol\mu)-\widetilde{\mathbf u}_n(\boldsymbol\mu)\big\|_X}
     {\big\|\mathbf u_{\mathcal N}(\boldsymbol\mu)\big\|_X},
\end{equation}
where $\widetilde{\mathbf u}_n$ is given by~\eqref{eq:nonlinearapprox} with
$\mathbf\Psi_{\mathrm{hom}}$ for the homogeneous variants, and
by~\eqref{eq:aug_quad_arapprox} for the full map.
\end{addmodblock}

The goal is not to assess or promote the efficiency of the quadratic manifold approximations,
but rather to demonstrate that, whenever such methods are employed, there is no reason to
discard the affine terms in the reduced coefficients.

\begin{addmodblock}
\Cref{fig:quad_fin} shows that the two references give very different answers. With the
solution reference the homogeneous ansatz stagnates about three orders of magnitude above the
augmented one over the whole range. This is the situation described in \Cref{sec:taylor-link}:
the range of $\mathbf V$, built from the whole parameter range, does not contain the tangent
space to $\mathcal S$ at $\mathbf u^{*}$, so that $\mathbf I-\Pi_{\mathbf V}$ does not
annihilate the first-order term and the correction to be represented retains a linear part,
which $\mathbf\Psi_{\mathrm{hom}}$ cannot carry. With empirical-mean centring the gap is
smaller: the augmented map reaches a given accuracy with about two fewer modes.

The augmented curves are, by contrast, nearly identical in the two cases. The augmentation
therefore improves the approximation under the recommended centring, and in addition removes
the strong sensitivity of the homogeneous ansatz to the choice of reference state. The
magnitude of the gain naturally depends on the problem, but the example shows that the affine
contributions in the correction cannot in general be discarded.

\Cref{fig:quad_transport} reports the same experiment on the transport case of
\Cref{sec:test-case-advection}, with $p=1$ and the same training and test set sizes, and the
two panels tell the same story.  With the solution reference both homogeneous variants plateau while the augmented ones keep decreasing, the gap reaching several orders of magnitude at the largest $n$. With empirical-mean centring the plateau disappears and the two formulations come much closer, a gain of about an order of magnitude nevertheless remaining for QGM, the two QSVDM curves nearly coinciding. As in
\Cref{fig:nlcrb-sweep-advection}, the linear curve is flat over the whole range.
\end{addmodblock}

\begin{figure}[htbp]
\centering
\includegraphics[width=\textwidth]{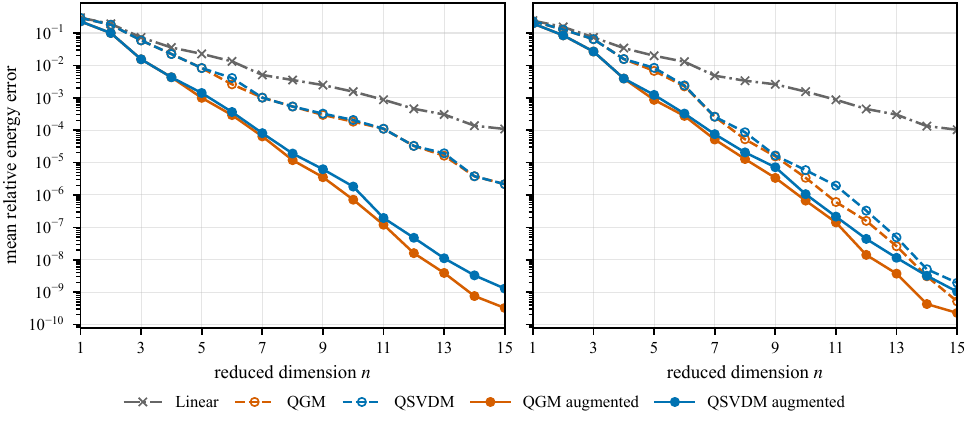}
\caption{\addmod{Error~\eqref{eq:enh-err} against the reduced dimension $n$ for the elliptic test case of \Cref{sec:test-case}. Left: snapshots
centred on $\mathbf u^{*}$. Right: snapshots centred on $\bar{\mathbf u}$.}}
\label{fig:quad_fin}
\end{figure}

\begin{figure}[htbp]
\centering
\includegraphics[width=\textwidth]{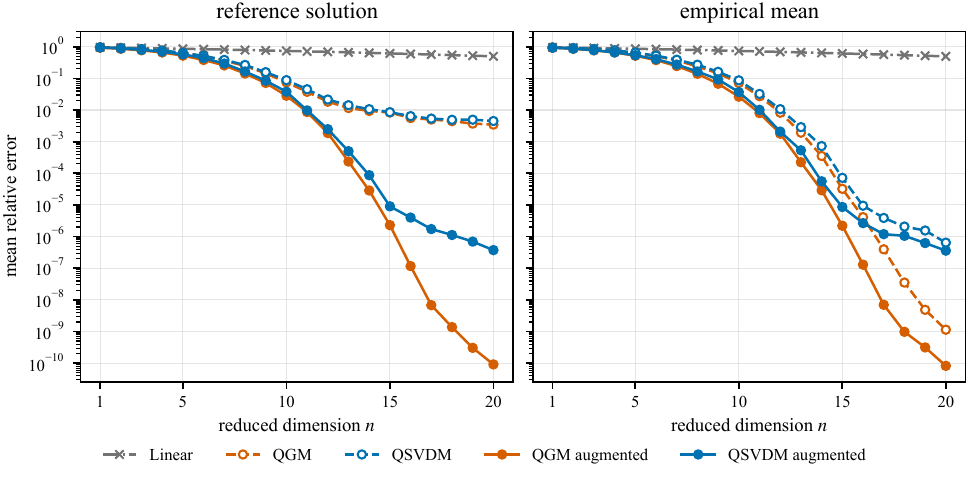}
\caption{\addmod{Error~\eqref{eq:enh-err} against the reduced dimension $n$ for the transport test case of \Cref{sec:test-case-advection}. Left: snapshots centred on $\mathbf u^{*}$. Right: snapshots centred on $\bar{\mathbf u}$.}}
\label{fig:quad_transport}
\end{figure}

Note that, unlike the NCRBA, whose effective
reduced dimension satisfies $n \approx n_s$, quadratic approximation methods require
$n \gg n_s$ to perform adequately over wide parameter ranges.

\section{Conclusion}
This paper contributes to the development of nonlinear reduced-basis approximation methods. More specifically, we focus on techniques that, in the spirit of the Gelfand widths framework, rely on approximations expressed as a truncated expansion with $N$ terms that has been named the Nonlinear Compressive Reduced Basis Approximation. In this representation, the first $n$ terms correspond to genuine unknowns—i.e., degrees of freedom (dof) of the discrete problem to be solved—while the remaining $N-n$ terms are expressed explicitly in terms of the first $n$ dof.

We first prove that, if the underlying problem depends on $p$ parameters, locally the choice $n=p$ is the right one when a centered expansion is chosen. \minmod{This local statement does not globalise on its own, but the global bound of \Cref{prop:global-sensing} motivates that it remains sufficient to take $n$ of the order of $p$.} The total number of terms $N$, is then selected solely to ensure the desired approximation accuracy, which is itself related to the Kolmogorov linear dimension. Note that this result is associated with the inverse problem: is it possible to retrieve $\boldsymbol{\mu}$ from the knowledge of the RB coefficients $\alpha_{i,\boldsymbol{\mu}}$, $i=1,\dots, n$, see~\eqref{eq:nlcrba}?

Since the construction of the (linear) reduced basis is a problem in its own right—traditionally addressed either through a singular value decomposition (SVD) approach or through greedy algorithms—we propose, in \Cref{GSS}, as a side result, a hybrid strategy that combines the advantages of both methods to further improve the quality of the reduced basis.

The second major contribution of this article is to explain why expressing the high modes as quadratic functions of the low modes arises naturally, but this quadratic relationship is only valid locally, \minmod{and its homogeneity only under exact filtering of the tangent space, which a data-driven basis does not provide; restoring the constant and linear blocks is a minimal and inexpensive remedy.} Consequently, to achieve accurate approximations over a wide range of parameter variations, it is necessary to consider more complex representations. These may be provided, for instance, by statistical learning (AI-based) approaches or by more sophisticated functional expressions, such as those proposed in \cite{bensalah2025nonlinear}.

Finally, the theory developed in this work delineates what can be achieved at a purely theoretical level, namely the selection of $n$ as a function of $p$ to solve the above inverse problem, together with explicit expressions of the coefficients indexed from $n+1$ to $N$ as a function of the first $n$ ones. In practice, however, for a prescribed target accuracy, the approximation error must be interpreted as the combination of two distinct contributions :
\begin{itemize}
\item The first contribution is associated with the truncation level $N$ of the linear expansion used to represent the approximation, whether in an SVD basis or in a more suitable reduced basis. 
\item The second contribution arises from the approximation error incurred when expressing the high-mode coefficients as functions of the $n$ low-mode coefficients.
\end{itemize}
When this latter dependence is obtained via a learning-based procedure whose accuracy depends on the size of the available dataset, the resulting error is inherently relative. Consequently, for a fixed dataset, the absolute error on the reconstructed coefficients improves as the magnitude of these coefficients decreases. This practical constraint naturally leads to selecting values of $n$ that may exceed the theoretical limit—of order $p$—suggested by the analysis presented in this paper.

Therefore, practical considerations, and in particular the use of a posteriori error estimators, must be regarded as a necessary complement to the theoretical framework developed here. The systematic incorporation of such tools lies beyond the scope of the present work and is the subject of ongoing investigations.

\appendix

\section{Different notions of principal angles and the connection with  Davis--Kahan $\sin\Theta$ Theorem}
\label{appendix:principal_angles}
{\bf Principal Angles (Canonical Angles) between Subspaces.} Let $\mathcal U$ and $\mathcal V$ be two subspaces of $\mathbb{R}^n$ with $\dim(\mathcal U)=\dim(\mathcal V)=k$. The {\sl principal angles} (also {\sl called canonical angles}) $\theta_1,\theta_2,\dots,\theta_k\in[0,\pi/2]$ between $\mathcal U$ and $\mathcal V$ are defined via the singular value decomposition (SVD). Let $\mathbf{Q}_{\mathcal U} \in\mathbb{R}^{n\times k}$ and $\mathbf{Q}_{\mathcal V}\in\mathbb{R}^{n\times k}$ be matrices whose columns form orthonormal bases of $\mathcal U$ and $\mathcal V$, respectively. Consider the $k\times k$ matrix $\mathbf{Q}_{\mathcal U}^T\,\mathbf{Q}_{\mathcal V}$. Its singular values $\sigma_1,\sigma_2,\dots,\sigma_k$ (in nonincreasing order $\sigma_1\ge \sigma_2\ge\cdots\ge\sigma_k\ge 0$) are the cosines of the principal angles. In particular, we define
$$\cos\theta_i = \sigma_i(\mathbf{Q}_{\mathcal U}^T\mathbf{Q}_{\mathcal V}) \quad (i=1,\dots,k),$$
so that $0\le \theta_1 \le \theta_2 \le  \dots \le  \theta_k \le \frac{\pi}{2}$. It is often convenient to collect these angles into a diagonal matrix
\[
\mathbf{\Theta}(\mathcal U,\mathcal V) = \operatorname{diag}(\theta_1,\dots,\theta_k),
\]
together with the associated matrices
\[
\sin\mathbf\Theta(\mathcal U,\mathcal V)=\operatorname{diag}(\sin\theta_1,\dots,\sin\theta_k),
\qquad
\cos\mathbf\Theta(\mathcal U,\mathcal V)=\operatorname{diag}(\cos\theta_1,\dots,\cos\theta_k).
\]
This matrix notation is the one commonly used in perturbation theory.

\noindent {\bf Geometrical interpretation.} Geometrically, $\theta_1$ is the smallest angle between any unit vector in $\mathcal U$ and any unit vector in $\mathcal V$, $\theta_2$ is the next smallest after "removing" the directions that achieved $\theta_1$ (more precisely by working in the intersections of $\mathcal U$ and $\mathcal V$ with the orthogonal complement of these directions), and so on, up to $\theta_k$ which is the largest angle. Equivalently, one can characterize $\theta_1$ variationally as
$$\cos\theta_1 = \max_{\mathbf{u}\in \mathcal U, \mathbf{v}\in \mathcal V} \frac{\mathbf{u}^T\mathbf{v}}{\|\mathbf{u}\| \|\mathbf{v}\|}$$
and $\theta_1$ is attained by some unit vectors $ \mathbf{u}_1\in \mathcal U$, $\mathbf{v}_1\in \mathcal V$ (called the {\sl first principal vectors}). Then recursively for $i=2,\dots,k$,

$$
\cos\theta_i = \max_{\substack{
  \mathbf{u} \in \mathcal U,\, \mathbf{v} \in \mathcal V,\\
\mathbf{u} \perp \mathbf{u}_1, \dots, \mathbf{u}_{i-1} \\
  \mathbf{v} \perp \mathbf{v}_1, \dots, \mathbf{v}_{i-1}
}} \frac{\mathbf{u}^T\mathbf{v}}{\|\mathbf{u}\| \|\mathbf{v}\|}
$$
with corresponding principal vectors $\mathbf{u}_i,\mathbf{v}_i$ in the subspaces orthogonal to the previous ones. In particular, $\theta_k$ (the largest principal angle) corresponds to the least aligned directions between $\mathcal U$ and $\mathcal V$.

If $\mathcal U$ and $\mathcal V$ share a nonzero intersection, then $\theta_1=0$; at the other extreme, if $\mathcal U$ is orthogonal to $\mathcal V$ (every $\mathbf{u}\in \mathcal U$ is orthogonal to every $\mathbf{v}\in \mathcal V$), then all their principal angles satisfy \minmod{\(\theta_1=\cdots=\theta_k=\frac{\pi}{2}\)}. More generally, principal angles characterize the relative orientation of $\mathcal U$ and $\mathcal V$ in $\mathbb{R}^n$ in an {\sl invariant} way (independent of the particular bases chosen for $\mathcal U$ and $\mathcal V$). They also coincide with the concept of canonical correlations in statistics (the cosines $\cos\theta_i=\sigma_i$ are the {\sl canonical correlation} coefficients between the subspaces spanned by $\mathbf{Q}_{\mathcal U}$ and $\mathbf{Q}_{\mathcal V}$).

It is interesting to note that the optimality conditions associated with the above max-problem lead to the following orthogonality conditions :
$$ \forall i,j, \quad i,j =1,\dots,k \quad \mathbf{u}_i^T\mathbf{u}_j = \delta_{i,j},\quad \mathbf{u}_i^T\mathbf{v}_j =\delta_{i,j} \cos\theta_i .$$

\noindent {\bf Remark (Matrices or sets of vectors).} The above definition applies equally to the column spaces (ranges) of two matrices, or to two sets of vectors, since these span subspaces. In practice, given two matrices $\mathbf{A}$ and $\mathbf{B}$ with column-space $\mathcal U=\mathrm{Range}(\mathbf{A})$ and $\mathcal V=\mathrm{Range}(\mathbf{B})$, one often computes an orthonormal basis for each (for example via QR factorization) and then obtains the principal angles between $\mathcal U$ and $\mathcal V$ via the SVD of $\mathbf{Q}_{\mathcal U}^T \mathbf{Q}_{\mathcal V}$. The same procedure works for two arbitrary finite sets of vectors spanning $\mathcal U$ and $\mathcal V$. Thus, we will speak of principal angles between subspaces, or between matrices $\mathbf{A}$ and $\mathbf{B}$, or between sets of vectors, interchangeably – the notion depends only on the two subspaces in question.

\begin{minmodblock}
\noindent {\bf Extreme Angles and Distance Between Subspaces.}
The smallest principal angle $\theta_1$ indicates the closest alignment between
$\mathcal U$ and $\mathcal V$, whereas the largest principal angle $\theta_k$
indicates the worst-case misalignment. A useful quantitative measure of the
distance between two subspaces of the same dimension is provided by the largest
principal angle. Let $\mathbf P_{\mathcal U}$ and $\mathbf P_{\mathcal V}$
denote the orthogonal projection matrices onto $\mathcal U$ and $\mathcal V$,
respectively. Then
\[
\sin\theta_k
=
\|(\mathbf I-\mathbf P_{\mathcal U})\mathbf P_{\mathcal V}\|_2
=
\|(\mathbf I-\mathbf P_{\mathcal V})\mathbf P_{\mathcal U}\|_2
=
\|\mathbf P_{\mathcal U}-\mathbf P_{\mathcal V}\|_2.
\]
To see the geometric meaning of this identity, observe that for any unit vector
$\mathbf v\in\mathcal V$,
\[
\|\mathbf v-\mathbf P_{\mathcal U}\mathbf v\|
=
\|(\mathbf I-\mathbf P_{\mathcal U})\mathbf v\|
\]
is precisely the distance from $\mathbf v$ to the subspace $\mathcal U$.
The maximum of this distance over all unit vectors of $\mathcal V$ is attained
at a principal vector associated with the largest principal angle $\theta_k$.
For such a vector, the residual
$(\mathbf I-\mathbf P_{\mathcal U})\mathbf v$ is orthogonal to $\mathcal U$
and has norm $\sin\theta_k$. Hence
\[
\max_{\mathbf{v}\in \mathcal V, \|\mathbf{v}\|=1}
\|(\mathbf I-\mathbf P_{\mathcal U})\mathbf v\|
=
\sin\theta_k.
\]
By symmetry, the same statement holds with $\mathcal U$ and $\mathcal V$
interchanged. The quantity $\sin\theta_k$ is often called the gap between the
two subspaces and defines a natural metric on the Grassmannian through the
difference of their orthogonal projectors. In particular,
$\sin\theta_k=0$ if and only if $\mathcal U=\mathcal V$, while
$\sin\theta_k=1$ if and only if $\mathcal V\cap\mathcal U^{\perp}\neq\{\boldsymbol 0\}$, that is, some direction of $\mathcal V$ is orthogonal to the whole of $\mathcal U$.
\end{minmodblock}

\noindent {\bf Davis--Kahan $\sin\Theta$ Theorem (Eigenvector Perturbation)}.
Principal angles play a key role in perturbation bounds for invariant
subspaces. A cornerstone result is the Davis--Kahan $\sin\Theta$ theorem, which
gives an upper bound on the largest principal angle between two eigenspaces in
terms of the perturbation norm. In one convenient form, it reads as follows.

\begin{theorem}[$\sin\Theta$ theorem, Davis and Kahan \cite{kahan}]
\label{thm:davis-kahan}
Let $\mathbf{A}$ be a real symmetric matrix and
$\tilde{\mathbf{A}}=\mathbf{A}+\mathbf{E}$ a perturbed version, with $\mathbf{E}$
symmetric. Let $\mathcal U$ be the invariant subspace of $\mathbf{A}$ spanned by
$k$ eigenvectors associated with the eigenvalues $\lambda_1,\dots,\lambda_k$,
and let $\tilde{\mathcal U}$ be the corresponding $k$-dimensional invariant
subspace of $\tilde{\mathbf{A}}$, with eigenvalues
$\tilde\lambda_1,\dots,\tilde\lambda_k$. If
\[
\delta=\min\bigl\{|\lambda_i-\tilde\lambda_j|\ :\ 1\le i\le k,\ k+1\le j\le n\bigr\}>0,
\]
then
\[
\|\sin\mathbf\Theta(\mathcal U,\tilde{\mathcal U})\|_2
=\sin\theta_k(\mathcal U,\tilde{\mathcal U})\le\frac{\|\mathbf{E}\|_2}{\delta},
\qquad
\|\sin\mathbf\Theta(\mathcal U,\tilde{\mathcal U})\|_F\le\frac{\|\mathbf{E}\|_F}{\delta}.
\]
Moreover, denoting by
$\delta^{*}=\min\{|\lambda_i-\lambda_j|\ :\ 1\le i\le k,\ k+1\le j\le n\}$ the
separation computed from the spectrum of $\mathbf{A}$ alone, one has, whenever
$\|\mathbf{E}\|_2<\delta^{*}$,
\[
\|\sin\mathbf\Theta(\mathcal U,\tilde{\mathcal U})\|_F\le\frac{2\|\mathbf{E}\|_F}{\delta^{*}}.
\]
\end{theorem}

This classical result provides a concise and powerful way to bound the
perturbation of an invariant subspace in terms of the perturbation of the
matrix. We emphasize that it applies specifically to symmetric (Hermitian)
matrices, whereas the bounds needed in this work concern the singular subspaces
of rectangular matrices. These are obtained by applying
\Cref{thm:davis-kahan} to the Gram matrix.

\begin{addmodblock}
\begin{corollary}
\label{cor:gram-dk}
Let $\mathbf{F}^r,\mathbf{F}\in\mathbb{R}^{n\times m}$ satisfy
$\|\mathbf{F}^r-\mathbf{F}\|_F=\mathcal{O}(r)$ as $r\to0$, and assume that, for
some $k$,
\[
\delta_k:=\sigma_k(\mathbf{F})^2-\sigma_{k+1}(\mathbf{F})^2>0 ,
\]
where $\sigma_i(\mathbf{F})$ denotes the $i$-th largest singular value of
$\mathbf{F}$. Denote by $\mathcal V^r_{1:k}$ and $\mathcal V_{1:k}$ the
subspaces spanned by the first $k$ left singular vectors of $\mathbf{F}^r$ and
of $\mathbf{F}$ respectively. Then, for $r$ small enough,
\[
\bigl\|\sin\mathbf\Theta\bigl(\mathcal V^r_{1:k},\mathcal V_{1:k}\bigr)\bigr\|_F
\;\leq\;\frac{2\|\mathbf{E}^r\|_F}{\delta_k}=\mathcal{O}(r),
\qquad
\mathbf{E}^r:=\mathbf{F}^r(\mathbf{F}^r)^{\top}-\mathbf{F}\mathbf{F}^{\top}.
\]
\end{corollary}

\begin{proof}
Apply \Cref{thm:davis-kahan} with $\mathbf{A}:=\mathbf{F}\mathbf{F}^{\top}$ and
$\mathbf{E}:=\mathbf{E}^r$, so that
$\tilde{\mathbf{A}}=\mathbf{F}^r(\mathbf{F}^r)^{\top}$; both are symmetric. The
eigenvectors of $\mathbf{F}\mathbf{F}^{\top}$ are the left singular vectors of
$\mathbf{F}$, with eigenvalues $\sigma_i(\mathbf{F})^2$, and likewise for
$\mathbf{F}^r$. Since squaring preserves the ordering, $\mathcal V_{1:k}$ and
$\mathcal V^r_{1:k}$ are the invariant subspaces associated with the $k$
largest eigenvalues of $\mathbf{A}$ and of $\tilde{\mathbf{A}}$ respectively, and
the separation of that cluster from the rest of the spectrum of $\mathbf{A}$ is
$\delta^{*}=\delta_k$. The perturbation satisfies
\[
\mathbf{E}^r=(\mathbf{F}^r-\mathbf{F})(\mathbf{F}^r)^{\top}
+\mathbf{F}(\mathbf{F}^r-\mathbf{F})^{\top}=\mathcal{O}(r),
\]
since $\|\mathbf{F}^r\|\leq\|\mathbf{F}\|+\|\mathbf{F}^r-\mathbf{F}\|$ and
$\|\mathbf{F}^r-\mathbf{F}\|=\mathcal{O}(r)$. The one-sided gap bound of
\Cref{thm:davis-kahan} then applies as soon as $\|\mathbf{E}^r\|_2<\delta_k$,
and gives the stated estimate.
\end{proof}
\end{addmodblock}

Sharper bounds, established directly for rectangular matrices and involving the
gap $\sigma_k-\sigma_{k+1}$ rather than $\sigma_k^2-\sigma_{k+1}^2$, are due to
Wedin \cite{wedin}. We shall not need them: passing through the Gram matrix
costs a larger constant but yields the same rate.

\section{SVD modes convergence}
\label{appendix:svd_modes_convergence}
In this appendix, we discuss the convergence behavior of the left singular
vectors (SVD modes) of a perturbed matrix
\(\mathbf{F}^r \in \mathbb{R}^{n \times m}\) converging to a fixed matrix
\(\mathbf{F} \in \mathbb{R}^{n \times m}\) as \( r \to 0 \) with
\(\|\mathbf{F}^r - \mathbf{F}\| = \mathcal{O}(r)\).
We assume that \( \mathrm{rank}(\mathbf{F}) = k\).

\subsection{Problem statement}
Let the singular value decompositions be given by
\[
\mathbf{F}^r = \mathbf{U}^r \boldsymbol{\Sigma}^r (\mathbf{Z}^r)^\top,
\qquad
\mathbf{F} = \mathbf{U} \boldsymbol{\Sigma} \mathbf{Z}^\top,
\]
with 
\[
\boldsymbol{\Sigma}^r = \mathrm{diag}(\sigma_1^r, \dots, \sigma_k^r, \dots),
\qquad 
\boldsymbol{\Sigma} = \mathrm{diag}(\sigma_1, \dots, \sigma_k, 0, \dots),
\]
where \( \sigma_1 \geq \cdots \geq \sigma_k > 0  \), and \( \sigma_k^r > 0 \) for all sufficiently small \(r\) (Weyl’s theorem). We are particularly interested in the first \(k\) left singular vectors 
\[
\mathbf{U}_{(1)}^r := \mathbf{U}^r_{(:,1:k)} =\begin{bmatrix}\boldsymbol{\varphi}_1^r ,& \cdots & ,\boldsymbol{\varphi}_k^r \end{bmatrix},
\qquad
\mathbf{U}_{(1)} := \mathbf{U}_{(:,1:k)} =\begin{bmatrix}\boldsymbol{\varphi}_1, & \cdots & ,\boldsymbol{\varphi}_k
\end{bmatrix},
\]
and we ask whether
\[
\|\mathbf{U}_{(1)}^r - \mathbf{U}_{(1)}\| \to 0,
\]
as \( r \to 0\).

\subsection{Non-uniqueness of SVD modes}
The above question is subtle because the singular vectors are not uniquely defined. 
Two cases are of interest:

\begin{enumerate}
    \item \textbf{Simple singular value:}  
    If \( \sigma_j \) is simple, then the corresponding singular vector is unique only up to a sign flip. That is, both \( \boldsymbol{\varphi}_j\) and \(-\boldsymbol{\varphi}_j\) are valid singular vectors.

    \item \textbf{Multiple singular values:}
    If \( \sigma_j = \cdots = \sigma_{j+\ell-1} \) is repeated, then the associated singular vectors span a subspace of dimension \(\ell\). Any orthonormal basis of this subspace is valid, so the singular vectors are determined only up to an \(\ell \times \ell\) orthogonal transformation.
\end{enumerate}

Thus, the matrix of singular vectors is only defined up to multiplication by an orthogonal transformation. Concretely, if \( \mathbf{U}_{(1)}' \in \mathbb{R}^{n \times k} \) is another valid choice of left singular vectors 
associated with the same singular values, then there exists an orthogonal matrix 
\( \mathbf{O} \in \mathbb{R}^{k \times k} \) such that
\[
\mathbf{U}_{(1)}' = \mathbf{U}_{(1)} \mathbf{O} .
\]

\subsection{Convergence up to orthogonal transformation}
\Cref{cor:gram-dk}, applied with this $k$ --- the gap being
$\delta_k=\sigma_k(\mathbf{F})^2$ since $\mathrm{rank}(\mathbf{F})=k$ ---
settles the convergence of the subspaces,
\[
\|\sin\mathbf\Theta(\mathbf{U}_{(1)}^r,\mathbf{U}_{(1)})\|_F=\mathcal{O}(r).
\]
It says nothing, by itself, about the individual singular vectors. One can
approach that question by explicitly searching for the \emph{best} orthogonal
transformation of \(\mathbf{U}_{(1)}\) aligning it with
\(\mathbf{U}_{(1)}^r\). Note that the transformed basis
\(\mathbf{U}_{(1)}\mathbf{O}^r\) does not necessarily correspond to an SVD
basis, but it provides the optimal alignment in the least-squares sense.

\begin{proposition}
\label{prop:basis-alignment}
Let $\mathbf{O}^r\in O(k)$ solve the orthogonal Procrustes problem
\[
\mathbf{O}^r=\arg\min_{\mathbf{O}\in O(k)}
\|\mathbf{U}_{(1)}^r-\mathbf{U}_{(1)}\mathbf{O}\|_F^2 .
\]
Then
\[
\|\sin\mathbf\Theta(\mathbf{U}_{(1)}^r,\mathbf{U}_{(1)})\|_F
\;\leq\;\|\mathbf{U}_{(1)}^r-\mathbf{U}_{(1)}\mathbf{O}^r\|_F
\;\leq\;\sqrt{2}\,\|\sin\mathbf\Theta(\mathbf{U}_{(1)}^r,\mathbf{U}_{(1)})\|_F ,
\]
and in particular $\|\mathbf{U}_{(1)}^r-\mathbf{U}_{(1)}\mathbf{O}^r\|_F=\mathcal{O}(r)$.
\end{proposition}

\begin{proof}
The solution of the Procrustes problem is obtained via the SVD $\mathbf{U}_{(1)}^{\top}\mathbf{U}_{(1)}^r
=\mathbf{Y}^r_1\mathbf{D}^r(\mathbf{Y}^r_2)^{\top}$, namely $\mathbf{O}^r=\mathbf{Y}^r_1(\mathbf{Y}^r_2)^{\top}$. Since $\mathbf{U}_{(1)}$ and $\mathbf{U}_{(1)}^r$ have orthonormal columns, $\mathbf{D}^r=\cos\mathbf\Theta(\mathbf{U}_{(1)}^r,\mathbf{U}_{(1)})$ by definition of the principal angles. Expanding the Frobenius norm and substituting $\mathbf{O}^r$, which gives $(\mathbf{O}^r)^{\top}\mathbf{U}_{(1)}^{\top}\mathbf{U}_{(1)}^r =\mathbf{Y}^r_2\mathbf{D}^r(\mathbf{Y}^r_2)^{\top}$, one obtains
\[
\|\mathbf{U}_{(1)}^r-\mathbf{U}_{(1)}\mathbf{O}^r\|^2_F
= 2k-2\operatorname{tr}(\mathbf{D}^r)
= 2\sum_{i=1}^{k}\bigl(1-\cos\theta_i\bigr).
\]
Since $\sin^2\theta_i=(1-\cos\theta_i)(1+\cos\theta_i)$ with $\cos\theta_i\in[0,1]$, one has $\sin^2\theta_i\leq 2(1-\cos\theta_i)\leq 2\sin^2\theta_i$, which is the stated two-sided bound. The last claim follows from \Cref{cor:gram-dk}.
\end{proof}

Note finally that, writing $e_j(r)$ for the per-mode error
\cref{eq:aligned-mode-error} reported in \Cref{sec:conv.svd}, one has
$\sum_j e_j(r)^2=\|\mathbf{U}_{(1)}^r-\mathbf{U}_{(1)}\mathbf{O}^r\|_F^2$, so
that convergence of the individual modes is equivalent to convergence of the
subspaces they span.


\section*{Acknowledgments}

As part of the “France 2030” initiative, this work has benefited from a national grant managed by the French National Research Agency (Agence Nationale de la Recherche), attributed to the ExaMA project of the NumPEx PEPR program under reference ANR-22-EXNU-0002, and has also been partially funded by the European Research Council (ERC) under the European Union’s Horizon 2020 research and innovation program (grant No. 810367), project EMC2 (YM).

\clearpage

\bibliographystyle{plain}
\bibliography{main}

\end{document}